\documentclass[final,onefignum,onetabnum]{siamart171218}

\usepackage{amssymb}
\usepackage[acronym]{glossaries}
\usepackage{enumerate}

\newtheorem{assumption}{Assumption}
\newtheorem{claim}{Claim}
\newtheorem{remark}{Remark}
\newtheorem{setting}{Setting}
\usepackage{xspace}
\usepackage{algorithmic} 

\usepackage{tikz}
\usepackage{wrapfig}
\usepackage{subcaption}
\usepackage{scalerel}

\usepackage{macros}
\newcommand{\interior}[1]{{\kern0pt#1}}
\newcommand{\argmin}{\operatorname*{arg\,min}}
\newcommand{\exploss}{\mathcal{L}}

\newacronym{HMM}{HMM}{heterogeneous multiscale method}
\newacronym{FNO}{FNO}{Fourier neural operator}
\newacronym{BIE}{BIE}{boundary integral equation}
\newacronym{FNO-HMM}{FNO-HMM}{FNO-based HMM}
\newacronym{BIE-HMM}{BIM-HMM}{boundary integral method-based HMM}
\newacronym{BIM}{BIM}{boundary integral method}
\newacronym{BEM}{BEM}{boundary element method}

\title{Deep Micro Solvers for Rough-Wall Stokes Flow in a Heterogeneous Multiscale Method}
\author{Emanuel Str\"om \and Anna-Karin Tornberg \and Ozan \"Oktem}

\begin{document}
\maketitle


\begin{abstract}
We propose a learned precomputation for the heterogeneous multiscale method (HMM) for rough-wall Stokes flow. A Fourier neural operator is used to approximate local averages over microscopic subsets of the flow, which allows to compute an effective slip length of the fluid away from the roughness. The network is designed to map from the local wall geometry to the Riesz representors for the corresponding local flow averages. With such a parameterisation, the network only depends on the local wall geometry and as such can be trained independent of boundary conditions. We perform a detailed theoretical analysis of the statistical error propagation, and prove that under suitable regularity and scaling assumptions, a bounded training loss leads to a bounded error in the resulting macroscopic flow. We then demonstrate on a family of test problems that the learned precomputation performs stably with respect to the scale of the roughness. The accuracy in the HMM solution for the macroscopic flow is comparable to when the local (micro) problems are solved using a classical approach, while the computational cost of solving the micro problems is significantly reduced. 
\end{abstract}

\section{Introduction}
When devising simulation studies of macroscopic properties of fluid flows over a rough terrain, the large scale separation between terrain structure and the size of the simulation region causes problems for classical solvers. Grids that fully resolve the roughness are typically too costly to use, whereas coarse grids fail to capture the effects of the terrain (see \cref{fig: coarse grid}). A classical approach to remedy the problem of mixed scales is to model them separately; one first approximates the effects of the rough topology on a localized region close to the wall, and translates the effects to a simplified computational region with slightly altered boundary data, called a  \emph{wall law}~\cite{engquist_E_hmm}.

\emph{Asymptotic homogenization} methods have been used to derive wall-laws for viscous flows over a surface with periodic roughness~\cite{jager_poiseuille,Jger2003_couette,bodart_stokes_homo}, as well as random, ergodic roughness~\cite{basson2006wall,dalibard_2011_ergodic}. For these flows one commonly uses a \emph{Navier-slip} condition, a Robin-type boundary condition that relates the \emph{wall shear stress} to the fluid velocity at the wall by a proportionality constant (the \emph{slip amount}). The slip amount is determined from the solution to a problem on a reference domain that includes the boundary, also called a \emph{unit cell problem}.
    
Carney and Engquist applied the \gls{HMM} to \emph{laminar viscous} flow over a non-periodic, non-ergodic rough wall in~\cite{carney2021heterogeneous}. \gls{HMM} is a general framework for multiscale modelling which assumes that the governing equations are known at both the macroscopic and microscopic scales, up to a set of simulation parameters. The key building blocks of \gls{HMM} are a \emph{compression operator} that maps micro scale data to key parameters that govern the macro scale behavior (e.g. slip amount), and a \emph{reconstruction operator} that extracts relevant macroscopic data which modulate the micro scale (e.g. flow rates at the boundary)~\cite{engquist_E_hmm}. 
    
In the paper by Carney and Enquist, the compression operator extracts pointwise slip amount via not only one (as in asymptotic homogenization), but several microscopic simulations of the linearized flow along the rough wall. Contrary to homogenization, the boundary conditions of the microscopic flow are set by the macroscopic flow. This results in a coupled system that involves both the macroscopic and microscopic variables. 
    
In cases when the microscopic solution, compression and reconstruction operators are all \emph{linear}, one can efficiently precompute the composition of these three operators (or a subset of them) using the \emph{representation theorem}. One important note is that solving multiple microscopic problems can often be more expensive than solving the macroscopic system. Furthermore, wall-laws incur a modelling error which cannot be overcome by \gls{HMM}, even with exact micro solvers. Hence, there is a high tolerance for error in the micro solvers for \gls{HMM} compared to fully resolving methods, and a simultaneous need for reduced computational cost. 
    
One promising idea, which we explore further in this paper, is to \emph{learn} a microscopic solution procedure in \gls{HMM} by training against simulation data~\cite{eenquist2007_hmm_review}. The key underlying observation motivating this approach is that deep learning based approximations are, once trained, often faster to evaluate than traditional numerical approximations, and that a decrease accuracy compared to classical numerical schemes can be acceptable in the case of \gls{HMM}. As the geometry of the micro domains vary, it is essential that the learned approximation generalizes over the types of geometries that one expect to encounter.

Most work on deep learning for multiscale problems focuses on \emph{sequential methods} where information propagates between the scales in only one direction. An example is the work on \emph{deep potentials} by E et al that leverages deep neural networks to find potentials for atomistic \emph{particle interactions}, trained on \emph{quantum dynamics}~\cite{e2022_deep_potentials}. Other examples are prediction of \emph{stress fields}~\cite{gupta2022_stress_field} and \emph{crystalline structure}~\cite{pokharel2021_chrystal_plasticity}, and homogenized quantities~\cite{liu2019_topology_learning} in microstructures in materials science. See the review  by Baek et al for a comprehensive overview of ML methods for multiscale problems~\cite{multiscale_ml_2023}. 
    
In contrast to sequential methods, \emph{concurrent methods} employ a feedback loop in which information flows both from small to large scales and from large to small scales. There is less work on ML for concurrent methods. An example is~\cite{liu2019_topology_learning}, which uses the \emph{PCA-Net} architecture from~\cite{stuart2021hilbertpca} to learn solutions to a local \emph{unit-cell problem} in the micro scale variables. The averages feed into a dynamical macro scale \emph{deformation} model, which in turn updates the boundary condition of the unit-cell. 
    
\subsection{Outline and Contributions}
We consider the setting of rough-wall laminar flow, and propose a learned precomputation for the composition of micro solution and compression operators in the \gls{HMM} approach from~\cite{carney2021heterogeneous}. Our approach uses a \gls{FNO}~\cite{li2022_geofno}, whose input is a parameterization of the wall geometry, and whose output is the \emph{Riesz representor} for the composition of the microscopic solution operator with the compression operator. Leveraging Riesz representation theorem in this way is beneficial for any \gls{HMM} approach, but especially so when the Riesz representor is learned from data. We effectively cut the input dimension in half and subsequently improve sample complexity, since there is no need to incorporate boundary values as a part of the training data. The network architecture makes use of a formulation based on a \gls{BIE}, which enables a significant reduction in dimensionality, and generalizes over a range of geometries. 
    
The paper is structured as follows. \Cref{sec: background} presents a mathematical formulation of the rough-wall flow problem along with a review of the \gls{HMM} method from~\cite{carney2021heterogeneous}. There we also formulate a precomputation method based on adjoint calculus. We describe the network architecture and related learning problem in \cref{sec: method}. Then, in~\cref{sec: error analysis} we show that under suitable regularity assumptions, a bounded training loss leads to a bounded error in the homogenized solution. Lastly, in \cref{sec: experiments} we present numerical results that support the theoretical findings. We conclude in \cref{sec: conclusion} with a discussion of future work.

\clearpage
\section{A Multiscale Flow Problem}\label{sec: background}
We consider rough-wall laminar flow in two dimensions based on the homogeneous Stokes equations, which are obtained by linearizing  Navier Stokes equations. Our method relies on \emph{homogeneity} for dimensionality reduction and \emph{linearity} for efficient precomputation. Although the Stokes equations are generally a limiting model, they are good approximations of near-wall flow if surface roughness is small compared to the boundary layer~\cite{achdou1998_rough_bc}. In such cases one may use our method for the so-called micro problems presented below, even if the macroscopic flow is non-linear and non-homogeneous. We first state the general setting for which our method can be applied, and then present problem-specific details.

\subsection{Acceleration of HMM through the Representation Theorem}\label{sec: hmm}
Enquist, E and Huang first proposed \gls{HMM}  in~\cite{engquist_E_hmm} as a general-purpose framework for computing solutions to multiscale problems with scale separation. The method is \emph{concurrent} in the sense that it couples the \emph{microscopic} and a \emph{macroscopic} models through \emph{compression} and \emph{reconstruction} operators. 

At an abstract level, \gls{HMM} can be formally written as the solution to the coupled continuum problems
\begin{equation}
 \begin{cases}
     \macm(\macq, \macd)=0, & \macd=\comp\bigl((\micq_n, \micdom_n)_{n=1}^N\bigr) \\
     \micm(\micq_n, \micd_n)=0, & \micd_n=\rec(\macq, \micdom_n),\quad n=1,\dots,N
 \end{cases}
 \label{eq: hmm_problem}
\end{equation}
where $\macq\colon\dom\to\Reals$ and $\micq_n\colon\micdom_n\to\Reals,\; n=1,\dots, N$ are the \emph{macro} and \emph{micro quantities of interest}, $\macm$ and $\micm$ are the macro and micro \emph{models}, $\dom$ and $\micdom_n$ are the \emph{macro} and \emph{micro domains} on which the quantities are defined, and $\macd$ and $\micd_n$ are the macro and micro \emph{data} that are obtained from the quantities of interest through a \emph{compression} operator $\comp$ and a \emph{reconstruction} operator $\rec$. We refer to the fixed-data subproblems
\begin{align}
 \macm(\macq, \macd)&=0 \label{eq: macro_problem}\\
 \micm(\micq_n, \micd_n)&=0,\; n=1,\dots,N \label{eq: micro_problem}
\end{align}
as the \emph{macro} and \emph{micro problem(s)} respectively, and~\eqref{eq: hmm_problem} as the \emph{\gls{HMM} problem}. The \gls{HMM} problem can be solved as a \emph{monolithic} system, but in the case when the macro and micro problems can be solved efficiently, a more computationally feasible approach is to iterate between solving the macro and micro scales until convergence. 
    
In practice the geometry differs for each micro problem, and one must solve multiple problems at each iteration. Depending on the complexity of the micro problems and the size of $N$, such an iterative method can become prohibitively expensive. In the following, we make an assumption about the compression operator that allows us to improve the efficiency of \gls{HMM} considerably. This assumption holds in the specific case of rough-wall Stokes flow, but also extends to other multiscale problems.

\begin{assumption}
\label{asm: micro_problem}
The microscopic quantities $\micq_n$ is an element in a Banach space $\mathcal{X}$, and the microscopic data $\micd_n$ in a Hilbert space $\micdspace$. The solution to the micro problem \eqref{eq: micro_problem} is \emph{bounded} and \emph{linear} with respect to the data $\micd$, and $\comp$ factors into
\[
 \comp\bigl((\micu_n,Y_n)_{n=1}^N\bigr) = f\bigl(g(\micu_1,Y_1), \dots, g(\micu_N,\micdom_N)\bigr),
\]
where $g(\cdot,Y)\colon \mathcal{X}\to \Rn{M}$ is linear and bounded for all domains $Y$, and $f$ maps from $\Rn{NM}$ to the macro data $D$.
\end{assumption}

\begin{center}
\begin{minipage}{0.5\textwidth}
\begin{algorithm}[H]
\centering
\caption{HMM}
\label{alg: HMM}
\begin{algorithmic}[1]
\STATE{\textbf{Input:} \\ 
Reconstruction oper. $\rec$,\\
compression operator $\comp$,\\
macro problem \eqref{eq: macro_problem},\\
micro problem(s) \eqref{eq: micro_problem},\\
initial guess $\macd_0$.}
\STATE{\textbf{Output:} Macroscopic data $\macq$.}
\STATE{$k\gets 0$} 
\STATE{$\macq_0\gets $ solve \eqref{eq: macro_problem} with $\macd =\macd_0$}
\STATE{}
\WHILE{not converged}
  \FOR{$n\gets 1,\dots,N$}
  \STATE{$\micd_{k,n} \gets \rec(\macq_k,Y_n)$} 
  \STATE{$\micq_{k,n} \gets $ solve \eqref{eq: micro_problem} w. data $\micd_{k,n}$}
  \ENDFOR
  \STATE{$\macd_{k+1}\gets \comp((\micq_{k,n}, \micdom_{k,n})_{n=1}^N)$}
  \STATE{$\macq_{k+1}\gets$ solve~\eqref{eq: macro_problem} w. data $\macd_{k+1}$}
  \STATE{$k\gets k+1$}
\ENDWHILE
\STATE{$\macq\gets \macq_k$}
\end{algorithmic}
\end{algorithm}
\end{minipage}%
\begin{minipage}{0.5\textwidth}
\begin{algorithm}[H] 
\centering
\caption{Precomputed HMM}\label{alg: adjoint HMM}
\begin{algorithmic}
\STATE{\textbf{Input}: \\
Reconstruction operator $\rec$, \\
Factorized compression $\comp = f\circ g$,\\
macro problem \eqref{eq: macro_problem},\\
adjoint micro problem(s) \eqref{eq: adjoint_problem},\\
initial guess $\macd_0$.}
\STATE{\textbf{Output:} Macroscopic data $\macq $.}
\STATE{$k\gets 0$}
\STATE{$\macq_0\gets $ solve \eqref{eq: macro_problem} with $\macd =\macd_0$}
\STATE{$\adjmicu_m(Y_n)\gets$ solve~\eqref{eq: adjoint_problem}, for all $m,n$.}
\WHILE{not converged}
 \FOR{$n\gets 1,\dots,N$}
     \STATE{$\micd_{k,n} \gets \rec(\macq_{k},\micdom_n)$}
     \STATE{}
 \ENDFOR
 \STATE{$D_{k+1}\gets f((\inner{d_{k,n}}{\adjmicu_m(Y_n)})_{m,n=1}^{M,N})$}
 \STATE{$\macq_{k+1}\gets$ solve~\eqref{eq: macro_problem}, data $D_{k+1}$}
 \STATE{$k\gets k+1$}
\ENDWHILE
\STATE{$\macq\gets \macq_k$}
\end{algorithmic}
\end{algorithm}
\end{minipage}
\end{center}

\begin{lemma}\label{lemma: dual_adjoint}
Let \cref{asm: micro_problem} hold, where $(d,\micdom)\mapsto\micq(\micd, \micdom)$ solves the micro problem~\eqref{eq: micro_problem}. Then, $\micd\mapsto g(\micq(d, \micdom), \micdom)$ is linear and bounded and there exist elements $r_1(Y), \dots r_M(Y) \in \mathcal{D}$ such that, for each $m=1,\dots,M$,
\begin{equation}
     \inner{r_m(\micdom)}{\micd} = g_m\bigl(\micq(\micd, \micdom), \micdom\bigr),\quad \text{for all $d\in\mathcal{D}$,}
     \label{eq: adjoint_problem}
\end{equation}
where $\inner{\cdot}{\cdot}$ denotes the inner product  on $\micdspace$ and $g_m$ is the $m$:th element of $g$.
\end{lemma}
\begin{proof}
The compositions of bounded linear operators is itself a bounded linear operator. The claim now follows directly from the Riesz representation theorem.
\end{proof}
\begin{remark}
In the following, we will denote the $r_m(\micdom)$, $m=1,\ldots,M$ the {\em Riesz representors} of the problem. Note that they depend only on the domain $\micdom$ and not on the boundary conditions. 
\end{remark}
By \cref{lemma: dual_adjoint}, we can replace the $N$ micro problems \eqref{eq: micro_problem} with $MN$ solutions  of the \emph{adjoint problems} \eqref{eq: adjoint_problem}, $M$ problems for each micro domain $Y=Y_n$. After that, the compression operator can be evaluated using inner products:
\[
 \comp\bigl((\micq_n,\micdom_n)_{n=1}^N\bigr) = f\Bigl(\bigl(\inner{r_m(Y_n)}{d_n}\bigr)_{m=1,n=1}^{M,N}\Bigr).
\]
    
Solving the adjoint problems is typically as expensive as solving the ordinary micro problems~\eqref{eq: micro_problem}. On the other hand, when the domain $\micdom_n$ remains unchanged (up to a rigid body transformation), the adjoint problems only need to be solved once, and the Riesz representors can be stored and reused for subsequent micro problems, even if the data changes. Hence, if $M$ is sufficiently small ($M=2$ in our case as we will see in the next subsection), the adjoint formulation is computationally cheaper. \Cref{alg: HMM,alg: adjoint HMM} illustrate the differences between \gls{HMM} with and without precomputations. Note that inside the loop, the computational cost of \cref{alg: HMM} is dominated by repeatedly solving the micro problems, whereas \cref{alg: adjoint HMM} only uses inner products.
    
The main idea behind the \emph{learned precomputation} method we propose is to use deep learning to predict the Riesz representors given $\micdom$, i.e. to \emph{learn the mapping $\micdom\mapsto r_m(\micdom)$}. 
The dependence on the geometry $\micdom$ of the micro domain is in general non-linear and nontrivial, which is why a learned approach is suited for this task.


\subsection{Stokes Flow over a Rough Wall}\label{sec: problem_formulation}\label{sec: main problem}
Consider a bounded two-dimensional open, connected set (\emph{domain}) $\dom\subset \Rn{2}$ with boundary $\dombdry$. Let $\uvec \in \contspace_2(\dom, \Rn{2})$ be the velocity field and $p\in\contspace_1(\dom,\Reals)$ the pressure field. The \emph{Stokes equations} are given by the following system of \emph{partial differential equations} (PDEs) and \emph{boundary conditions} (BCs):
\begin{align}
 \label{eq: stokes_mom}
 -\Lap \uvec + \grad p &= \force \quad \text{in } \dom, \\
 \label{eq: stokes_div}
 \dive\uvec &= 0 \quad \text{in } \dom, \\
 \label{eq: stokes_bc}
 \uvec &= \bdryvel \quad \text{on } \dombdry,
\end{align}
where $\force\colon \dom \to \Rn{2}$ is a given \emph{external force} and $\bdryvel\colon \dombdry \to \Rn{2}$ is a given \emph{boundary velocity}. In general~\eqref{eq: stokes_mom} includes the \emph{viscosity}, but we set it to 1, which is always possible by reduction to dimension free form. We refer to~\eqref{eq: stokes_mom} as the \emph{momentum equation},~\eqref{eq: stokes_div} as the \emph{divergence constraint} and~\eqref{eq: stokes_bc} as the \emph{Dirichlet  boundary condition}. Regularity of $\force, \bdryvel$ and $\Omega$ determines regularity of the solutions $\uvec, p$. We always assume that the boundary $\dombdry$ is \emph{Lipschitz} continuous. The momentum equation~\eqref{eq: stokes_mom} and divergence constraint~\eqref{eq: stokes_div} will appear regularly in the following sections, and we adopt a shorter notation:
    \begin{equation}
    \diffoperator(\uvec, p) = 0,  
    \label{eq:up_operator}  
    \end{equation}
where $\diffoperator$ is an operator that maps $\uvec, p$ onto the pair $(-\Delta \uvec + \nabla p - \force, \nabla\cdot \uvec)$. 

The central object of study in this work is the \emph{multiscale Stokes problem}, which is to find a solution $(\uvec_\epsilon, p_\epsilon)$ to the Stokes equations~\eqref{eq: stokes_mom}-\eqref{eq: stokes_bc} in a \emph{multiscale domain} $\dom_\epsilon$ with boundary conditions $\bdryvel_\epsilon$ and external force $0$. We write the full problem for completeness:
    \begin{equation}
 \label{eq: stokes_mom_eps}
 \begin{cases} 
     \diffoperator(\uvec_\epsilon, p_\epsilon) = 0 &\text{in  $\dom_\epsilon$,}\\
     \uvec_\epsilon = \bdryvel_\epsilon &\text{on $\dombdry_\epsilon$.}
 \end{cases}
    \end{equation}
The multiscale domain $\dom_\epsilon$ has a smooth boundary apart from the \emph{rough wall} $\roughness_\epsilon$, a subset that has rapid oscillations with an amplitude and wavelength at some small scale $\epsilon$ -- see \cref{fig: rough_wall} for an illustration. We will later introduce a more rigorous description of $\dom_\epsilon$. We assume the velocity field satisfies a \emph{no-slip condition} at the steady rough wall:  \begin{equation}\label{eq:NoSlip} \bdryvel_\epsilon(\point) = 
 \begin{cases}
     \bdryvel(\point) & \text{if $\point \in \dombdry_\epsilon \setminus \roughness_\epsilon$,}  \\
     0 & \text{if $\point \in \roughness_\epsilon$.}
 \end{cases}
    \end{equation} 
The problem~\eqref{eq: stokes_mom_eps} is typically difficult to solve for small $\epsilon$ since it requires fine discretization near the rough wall. A naive solution is to approximate~\eqref{eq: stokes_mom_eps} by setting $\epsilon=0$, which entirely does away with any dependence on $\epsilon$. However, the surface roughness can contribute to significant changes in surface drag which affects the flow. \Cref{fig: coarse grid} shows how this naive approach results in poor accuracy. We will now formulate a \emph{concurrent} method based on \gls{HMM}, that achieves higher accuracy. The following sections closely follow~\cite{carney2021heterogeneous}.

\begin{figure}[!ht]
    \centering
    \includegraphics[width=\textwidth]{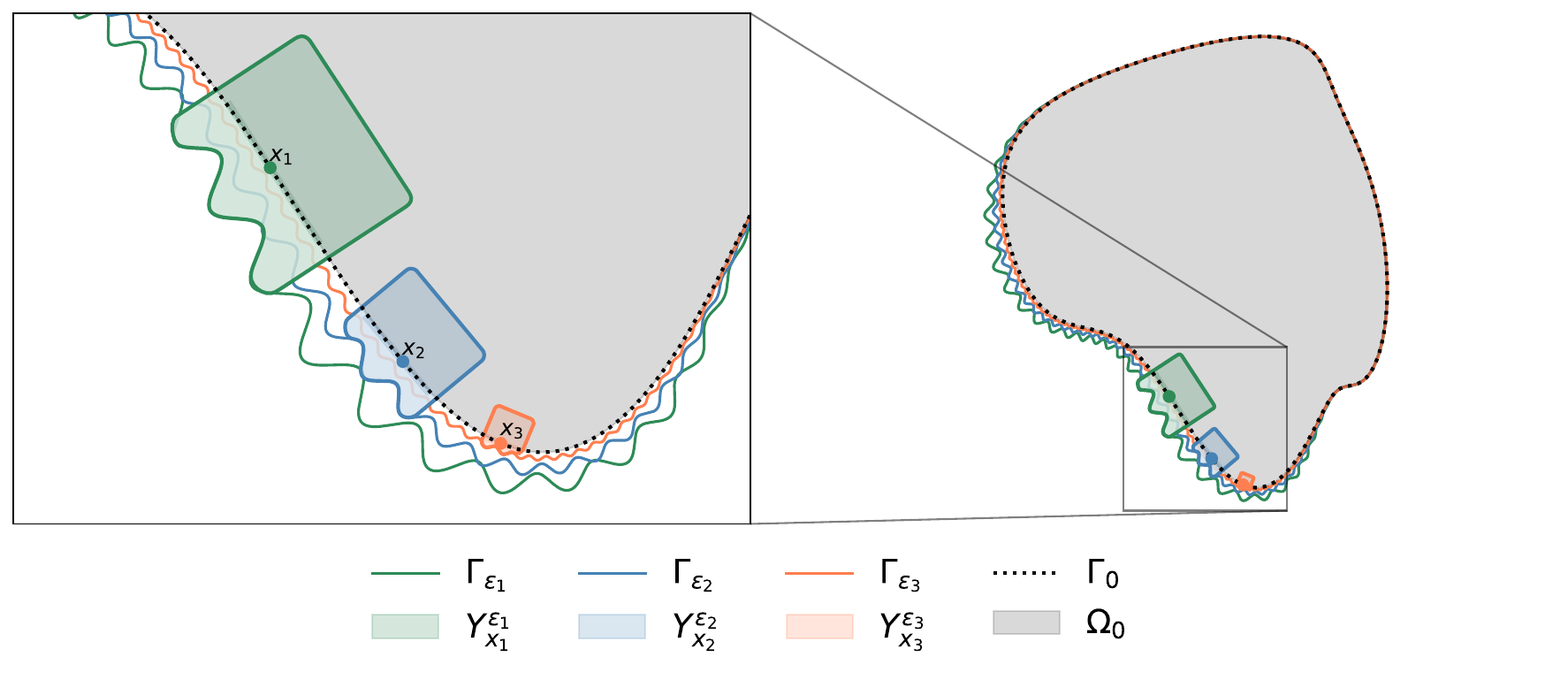}
    \caption{Example illustrating Multi-Scale domains $\dom_\epsilon$ with rough walls $\roughness_\epsilon$ for different values of $\epsilon$ (here $0.03, 0.02$, and $0.001$). $\dom_0$ is the Macro domain with (non-oscillatory) boundary $\roughness_0$ and $\micdom_\point^\epsilon$ is the Micro domain associated to a point $\point \in \roughness_0$.} 
    \label{fig: rough_wall}
\end{figure}

\begin{figure}[!ht]
\includegraphics[width=\textwidth]{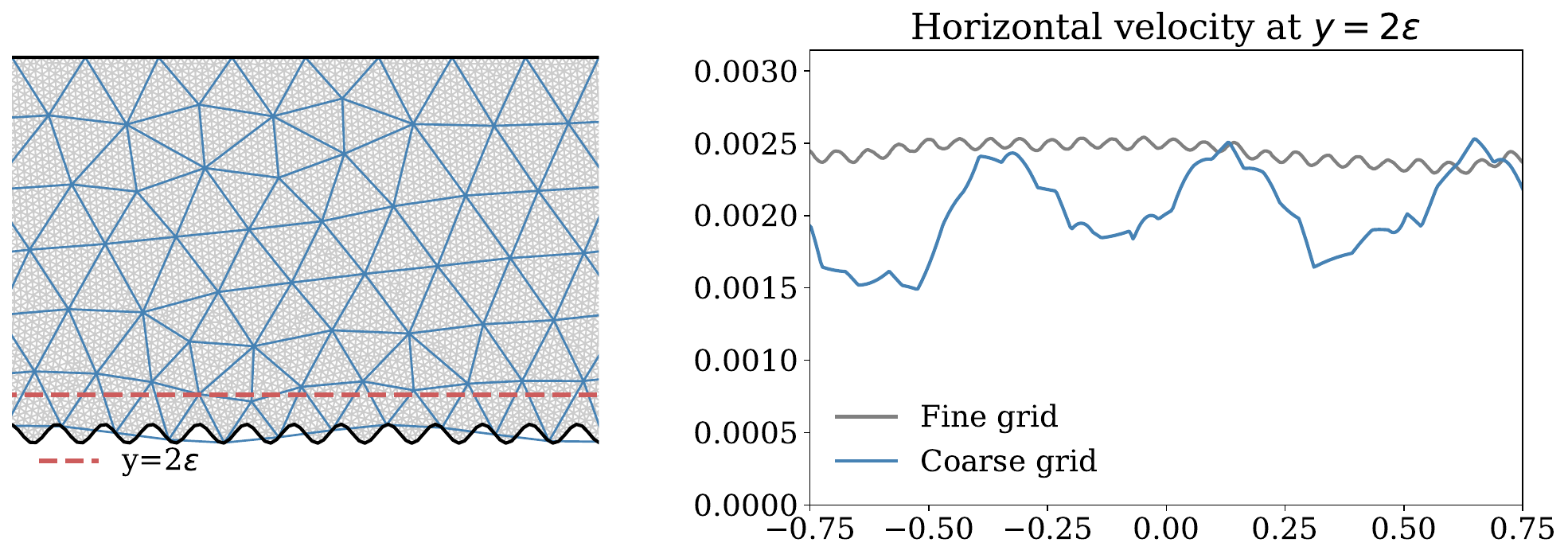}
\caption{Example of how a classical numerical method can fail for rough-wall flow. The left figure shows an undersampled coarse grid in blue overlayed on a fine grid, that fully resolves the roughness. The right figure shows the two FEM solutions on the respective grids, at a distance $2\epsilon$ from the boundary.}
\label{fig: coarse grid}
\end{figure}
    
\subsection{HMM for Rough Wall Flow}\label{sec: hmm_rough_wall}
Enquist and Carney formulate \gls{HMM} for a non-linear stationary Navier-Stokes problem with a rough wall in~\cite{carney2021heterogeneous}. We adopt the same constitutive model for linearized Stokes flow~\eqref{eq: stokes_mom_eps}. Recall (\cref{sec: hmm}) that the \gls{HMM} requires a macro problem~\eqref{eq: macro_problem} with unknown quantities $\macq$ and known data $\macd$, and a micro problem~\eqref{eq: micro_problem} with unknown quantities $\micq$ and known data $\micd$, as well as a compression operator $\comp$ that maps $\micq$ to $\macd$, and a reconstruction operator $\rec$ that maps $\macq$ to $\micd$. We present these building blocks as derived by Engquist and Carney. For the remainder of this section we fix an $\epsilon > 0$.
    
\subsubsection{Macro Problem} 
The macroscopic quantity $\macu$ of interest is the macroscopic velocity field. The pressure $\macp$ is treated as an auxilliary variable. The macro problem is stated in~\eqref{eq: macro_problem stokes}. This is a model for Stokes flow restricted to the smooth domain $\dom_0$; effectively cutting away the rough wall $\roughness_\epsilon$ and replacing it with the smooth wall $\roughness_0$. One compensates for the smoothing by imposing a \emph{Robin condition} with a slip amount $\alpha$ at $\roughness_0$ instead of a non-slip condition \eqref{eq:NoSlip}:
\begin{equation}\label{eq: macro_problem stokes}
    \begin{cases}
    \diffoperator(\macu, \macp) = 0_{\;\;} 
 & \text{in $\dom_0$,} \\
 \macu + (\alpha \tangent \cdot \partial_\normal \macu)\tangent = 0_{\;\;} & \text{on  $\roughness_0$,} \\
 \macu = \bdryvel & \text{on $\dombdry_0\setminus\roughness_0$.}
    \end{cases}
\end{equation}
Here the operator $\diffoperator$ was introduced in \eqref{eq:up_operator}, $\normal\colon\dombdry_0\to \torus$ is the outward-pointing \emph{unit normal vector} to $\dombdry_0$, $\tangent\colon\dombdry_0\to \torus$ the counter-clockwise \emph{unit tangent vector}, $\partial_\normal = \normal\cdot \grad$ is the derivative in the normal direction, and $\alpha\colon \roughness_0\to [0,\infty)$ is a positive scalar function (\emph{slip amount}). 
    
In order for the solution to~\eqref{eq: macro_problem stokes} to match that of~\eqref{eq: stokes_mom_eps}, one must find a suitable candidate for $\alpha$. Homogenization techniques can approximate $\alpha$ in specific settings. For example, when the roughness is a periodic fluctuation in the normal direction of a flat line segment, Enquist and Carney show in~\cite{carney2021heterogeneous} that $\alpha$ can be computed from a \emph{unit-cell problem} of Stokes type, posed at the boundary. Basson et al derive a similar method for a \emph{random, homogeneous roughness} in~\cite{basson2006wall}. In \gls{HMM}, we view $\alpha$ as the unknown data, which will be determined from the micro scale problem -- similar to the unit cell problem in homogenization theory.
      
\subsubsection{Micro Problem}\label{sec: micro problem} 
The micro problem consists of a set of $N$ Stokes problems on the \emph{micro domains} $\{\micdom_{n,\epsilon}\}_{n=1}^N$ centered at points $\{\point_n\}_{n=1}^N\subset \roughness_0$ along the smooth wall $\roughness_0$ and that overlap with the rough boundary $\roughness_\epsilon$, see ~\cref{fig: rough_wall}.
Each domain $\micdom_{n,\epsilon}$ defines a solution $\chi_{n,\epsilon}\colon \micdom_{n,\epsilon} \to \Rn{2}$ to a Stokes problem with pure Dirichlet boundary conditions on the \emph{micro boundary} $\partial\micdom$:
\begin{equation}
    \begin{cases}
    \diffoperator(\micu_{n,\epsilon}, \micp_{n,\epsilon}) = 0_{\;\;}  
      &\text{in $\micdom_{n,\epsilon}$} \\
    \micu_{n,\epsilon} = \micbdryvel_{n,\epsilon} 
      & \text{on $\partial\micdom_{n,\epsilon}$}
    \end{cases}
    \quad\text{for $n=1,\ldots, N$.}
\label{eq: micro_problem stokes}
\end{equation}
Here, $\micbdryvel_{n,\epsilon}$ is the microscopic data that agrees with the macro solution $\macu$ in the interior of $\dom_0$ and vanishes at the rough boundary. Once again, we treat the pressure $\micp_{n,\epsilon}\colon \micdom_{n,\epsilon}\to\Reals$ as an auxiliary variable. 

To improve upon readability, we will when referring to solutions of (or quantities related to) \eqref{eq: micro_problem stokes} often omit the dependency on indices $n$ (or point $\point_{n}$) and the scale $\epsilon$.

\subsubsection{Reconstruction Operator}\label{sec: reconstruction operator} 
In our setup, the reconstruction operator is supposed to map the macro flow $\macu$, defined on $\dom_0$, to boundary conditions for the micro domains. The resulting boundary value problems on the micro domains should be well-posed and well approximate the fully resolved solution. The macro domain $\dom_0$ does not contain the whole of $\partial\micdom$ in general (see ~\cref{fig: rough_wall}). Instead, we evaluate $\macu$ on the intersection $\dom_0\cap \partial\micdom$ and \emph{extrapolate} to the remaining part of the micro boundary. We now state requirements that guarantee well-posedness, and to some extent, closeness to the true solution.

\begin{setting}\label{setting: extrapolation}
The extrapolation is given by  $R\colon\contspace(\micdombdry\cap\dom_0,\Rn{2})\to\contspace(\micdombdry,\Rn{2})$ (reconstruction operator) that satisfies the following properties:
\begin{enumerate}[(i)]
\item $\micbdryvel=R[\uvec]$ is continuous on $\partial\micdom$.\label{item_cont}
\item The net flow through $\partial\micdom$ is zero: $\inner{\micbdryvel}{\normal}_{\micdombdry}=0$.\label{item_net_flow}
\item $h$ is consistent with the macro solution: $\micbdryvel=\macu$ on $\dom_0\cap \partial\micdom$.\label{item: macro_agreement}
\item $h$ satisfies the no-slip condition: $\micbdryvel=0$ on $\roughness_\epsilon\cap \partial\micdom$.\label{item: no_slip}
\end{enumerate}
\end{setting}
\cref{item_net_flow,item_cont} ensure well-posedness of the Stokes problem, and \cref{item: macro_agreement,item: no_slip} ensure that the micro solution $\micu$ is a good approximation of the true flow. Any operator $\rec$ such that $\micbdryvel =  \rec[\macu]$ satisfies~\cref{item_cont,item_net_flow,item: macro_agreement,item: no_slip} is considered a valid reconstruction operator.

\subsubsection{Compression Operator}\label{sec: compression operator} 
In our setup, the compression operator should map micro solutions $\{\micu_{\point_n}^\epsilon\}_{n=1}^N$ in the micro domains $\{\micdom_{\point_n}\}_{n=1}^N$ that are given by \eqref{eq: micro_problem stokes} to the effective slip amount $\interslip\colon\roughness_0\to \Reals$, a function defined on the smooth boundary. It should smooth out the $\epsilon$-scale but preserve the macroscopic flow properties, and result in a well-posed solution when $\interslip$ is plugged into~\eqref{eq: macro_problem stokes}. We do this in two steps. First, we will define a smoothing operator that maps the local solution from \eqref{eq: micro_problem stokes}, $\micu_{\point_n}^\epsilon$, to a local slip amount $\pointslip_n \in \Reals$, and then we define a continuous function on the smooth boundary $\roughness_0$ by \emph{interpolating} in-between the pairs $\bigl\{ (\point_n, \pointslip_n)\bigr\}_{n=1}^N \subset \roughness_0 \times \Reals$.
\par\smallskip
Consider an arbitrary micro solution $\micu$ to~\eqref{eq: micro_problem stokes} on $\micdom$ and let $\lineseg \subset \dom_0\cap \micdom$ be a straight line segment. We define the \emph{effective slip amount} (same as in~\cite{carney2021heterogeneous}) as
\begin{equation}
\pointslip := -\frac{\inner{\micu}{\tangent}_{\lineseg}}{\inner{\partial_\normal\micu}{\tangent}_{\lineseg}} = -\dfrac{\displaystyle{\int_{\lineseg}}\micu(\ypoint)\cdot \tangent(\ypoint)\,\mathrm{d}S(\ypoint)}{\displaystyle{\int_{\lineseg}} \partial_{\normal(\ypoint)}(\micu(\ypoint)\cdot \tangent(\ypoint))\,\mathrm{d}S(\ypoint)}.
\label{eq: slip avg}
\end{equation}
Here, $\inner{\cdot}{\cdot}_{\dom}$ is the standard $\Lp^2$-\emph{inner product} on the set $\dom$ and unit tangent vector $\tangent(\ypoint) \in \torus$ to the line segment $\lineseg$ in the point $\ypoint$, i.e., it is the constant directional vector for the line segment $\lineseg$ that approximates $\roughness_0\cap \micdom$ (part of the smooth boundary that intersects with the micro domain). We now state sufficient conditions for well-posedness:

\begin{setting}\label{setting: interpolation}
Define $\interslip \colon \roughness_0 \to \Reals$ as $\interslip := \interp\bigl(\bigl\{(\point_n, \pointslip_{\point_n})\bigr\}_{n=1}^N \bigr)$ where the \emph{interpolation operator} $\interp\colon \roughness_0^N \times \Rn{N} \to \contspace(\roughness_0,\Reals)$ satisfies:
\begin{enumerate}[(i)]
 \item  $\interslip$ is continuous on $\roughness_0$.\label{item: slip continuity}
 \item  $\interslip(\point_n)=\pointslip_{\point_n}$ for all $n=1,\dots,N$.\label{item: slip interpolation}
 \item  $\interslip(\point) > 0$ for all $\point\in\roughness_0$.\label{item: slip positivity}
\end{enumerate} 
\end{setting}
Any composition $\comp$ of pointwise evaluation of $\pointslip$ followed by an interpolation operator that satisfies assumptions stated in \cref{setting: interpolation} constitutes a valid compression operator. Enquist and Carney propose a linear interpolation in~\cite{carney2021heterogeneous}. 
To connect to \cref{asm: micro_problem}, in this setting 
$M=2$ with $g(\micu,\micdom) =(\inner{\micu}{\tangent}_\lineseg, \inner{\partial_\normal\micu}{\tangent}_\lineseg)$, and $f(y_1,y_2,\dots, y_N)=\interp(\{(x_n,y_{n1}/y_{n2})\}_{n=1}^N)$.

\subsubsection{Method Summary}\label{sec: method summary} 
We are now able to fomulate the full \gls{HMM} proposed in~\cite{carney2021heterogeneous}. Fix $\epsilon > 0$. First, we construct a smooth approximation $\roughness_0$ of the rough wall $\roughness_\epsilon$ in such a way that the resulting smooth domain~$\dom_0$ is in the interior of $\dom_\epsilon$. Then pick $N$ evaluation points $\{\point_n\}_{n=1}^N$ on $\roughness_0$. Given $\roughness_\epsilon$, for each $n=1,\ldots, N$ we construct a micro domain $\micdom_{n,\epsilon}$ and a reconstruction operator  $\rec_{n,\epsilon} \colon \contspace(\micdom_{n,\epsilon}\cap\dom_0)\to \contspace(\micdombdry_{n,\epsilon})$ that satisfies the assumptions stated in \cref{setting: extrapolation}. Next, we pick an interpolation method $\interp\colon \roughness_0^N\times \Rn{N} \to \contspace(\roughness_0,\Reals)$ that satisfies assumptions stated in \cref{setting: interpolation}. The \gls{HMM} method for finding a homogenized solution to the multi scale Stokes problem~\eqref{eq: stokes_mom_eps} then proceeds to repeat the following steps until convergence:
\begin{enumerate}
 \item Given a macroscopic solution $\macu^{(k)}$,
 \item \textbf{Reconstruction:} Compute $\micbdryvel_{n,\epsilon}^{(k+1)}\!\!:=\rec_{n,\epsilon}[\macu^{(k)}]$ for all $n=1,\dots,N$,
 \item \textbf{Micro:} Solve~\eqref{eq: micro_problem stokes} on $\micdom_{n,\epsilon}$ and data $\micbdryvel_{n,\epsilon}^{(k+1)}$ to obtain $\micu_{n,\epsilon}^{(k+1)}$ for all $n$,\label{item: summary micro}
 \item \textbf{Compression:} \label{item: summary compress}
 \begin{itemize}
     \item Compute the slip amount $\pointslip_n^{(k+1)}$ from $\micu_{n,\epsilon}^{(k+1)}$ using~\eqref{eq: slip avg} for all $n$,
     \item Interpolate $\{\pointslip_n^{(k+1)}\}_{n=1}^N$ to $\interslip^{(k+1)}$,
 \end{itemize}
 \item \textbf{Macro:} Solve macro prob.~\eqref{eq: macro_problem stokes} with data $\interslip^{(k+1)}$ to obtain $\macu^{(k+1)}$.\label{item: summary macro}
\end{enumerate}

In the precomputed version of \gls{HMM} (see \cref{alg: adjoint HMM}), which we will now introduce, step~\eqref{item: summary micro} above is replaced by an initial precomputation before entering the loop:
\begin{itemize}
\item[0.] \textbf{Precomputation:} Solve~\eqref{eq: adjoint_problem}, obtaining Riesz representors $\adjmicu_{n,1},\adjmicu_{n,2}$ for all $n$.
\end{itemize}
Furthermore, in the first substep of \cref{item: summary compress} one can directly compute $\pointslip_n^{(k+1)}$ by leveraging~\cref{eq: adjoint_problem}. The precomputation reduces computational cost considerably when solving \gls{HMM}. 
The adjoint formulation is generally known and was indeed used in convergence proofs in~\cite{carney2021heterogeneous}, but is rarely explained as a part of the \gls{HMM} method. To implement it for FEM, one must either derive the variational formulation for the adjoint problem, or use automatic differentiation techniques to extract the adjoint numerically. We will formulate~\cref{eq: adjoint_problem} using the boundary integral formulation of Stokes equations. This approach is well suited for machine learning, as it results in a dimensionality reduction in the input and output data. We proceed with a brief introduction to \gls{BIE}. 
 
\section{Computing Riesz Representors}
Carney and Enquist use a \emph{finite element method} (FEM) to solve both the micro and micro problems coupled with polynomial interpolation of the slip amount and choose a periodic roughness to avoid extrapolation. We instead use a \emph{boundary integral formulation} of the Stokes equations to solve the micro problem~\eqref{eq: micro_problem stokes}. We outline the method below.

\subsection{Dimensionality Reduction with Boundary Integrals}\label{sec: dimensionality reduction}  
The \gls{BEM} and \gls{BIM}, are both solution methods based on reformulating PDE systems as integral equations on the boundary of the domain. The main advantage of such methods is \emph{dimensionality reduction}, since one only needs to discretize the boundary $\dombdry$ instead of the entire domain $\dom$. The main disadvantage is that evaluating the solution often involves computations with nearly singular integral kernels for points close to the boundary. Furthermore, the kernel decays slowly, resulting in dense system matrices. We refer to~\cite{Pozrikidis1992_bie_book} for an introduction to BIM.
 
Let $\stresslet$ denote a \emph{fundamental solution} for the Stokes equations, known as the \emph{stresslet}:
\begin{equation}\label{eq:Stresslet}
     \stresslet_{ijk}(\dpoint)=-\frac{1}{\pi}\frac{\dpointi_i\dpointi_j\dpointi_k}{\|\dpoint\|^4},  \quad\text{for $i,j,k\in\{1,2\}$ and $\dpoint\in\Rn{2}$}.
\end{equation}
The stresslet is related to the gradient of the \emph{Green's function} for the Stokes equations, known as the \emph{Stokeslet}. The stresslet can be used to construct a \emph{double layer potential} $\dlp$ that maps vector fields $\dens$ on the boundary $\dombdry$ to velocity fields in the interior $\dom$:
\begin{equation}
     \dlp[\dens]_i(\point) = \sum_{j=1}^2\sum_{k=1}^2\int_{\dombdry} 2 \stresslet_{ijk}(\point - \ypoint) \normal_j(\ypoint)\dens_k(\ypoint)\mathrm{d}s(\ypoint), \quad\text{for $\point\in\dom$ and $i\in\{1,2\}$,}\label{eq: layerpotential}
\end{equation}
where $\normal(y)$ is the unit vector normal to the wall $\dombdry_0$ at the point $y$. Then, $\uvec := \dlp[\dens]$ solves the Stokes equations~\eqref{eq: stokes_mom}--\eqref{eq: stokes_div} with $\force=0$. Next, if $\uvec$ solves the full Stokes problem, it has to satisfy the boundary condition~\eqref{eq: stokes_bc}. Taking the limit of the right-hand-side of \eqref{eq: layerpotential} as $\point$ goes from the interior towards a point $\point_0$ on the boundary an equating with the Dirichlet data from~\eqref{eq: stokes_bc}, we obtain a \gls{BIE}~\cite[p.~110]{Pozrikidis1992_bie_book}:

\begin{equation}
     \bdryvel(\point_0) = \lim_{\point\to\point_0}\uvec(\point) = \layerpotential[\dens](\point_0) \quad \text{for $\point_0 \in \dombdry$
     where $\layerpotential := (\tfrac{1}{2}\Idop + \dlp_{\princval})$.}
     \label{eq: bie} 
\end{equation}
In the above, $\dlp_\princval$ denotes the \emph{Cauchy principal value} of~\eqref{eq: layerpotential} evaluated on the boundary $\dombdry$ and $\Idop[\dens]=\dens$ is the identity. 
Then, \emph{solving the \gls{BIE} \eqref{eq: bie} for $\dens$ ensures that $\uvec(\point) := \dlp[\dens](\point)$ in \eqref{eq: layerpotential} solves Stokes equations~\eqref{eq: stokes_mom}--\eqref{eq: stokes_div} with boundary conditions~\eqref{eq: stokes_bc}.}
 
Due to compactness of $\dlp_{\princval}$ when $\dombdry$ is a Lyapunov surface~\cite{Pozrikidis1992_bie_book}, the operator $\layerpotential$ in \eqref{eq: bie} is a \emph{Fredholm integral of the second kind}, which has a bounded inverse. 
Furthermore, although the inverse of $\layerpotential$ is not known analytically, we can formally express the solution $\uvec$ and its directional derivatives $\partial_{a}\uvec$ in some arbitrary direction $a\in \Rn{2}$ in terms the boundary data $\micbdryvel \in \contspace(\dombdry,\Rn{2})$:
\begin{equation}
     \uvec = \dlp[\layerpotential^{-1}[\micbdryvel]] \quad\text{and}\quad \partial_{a}\uvec = \partial_{a}\dlp[\layerpotential^{-1}[\micbdryvel]].\label{eq: layerpotential gradient}
\end{equation}
Note that when $\uvec$ is evaluated away from the boundary $\dombdry$, the directional derivative $\partial_{a} \dlp$ is given as an integral operator whose kernel is the directional derivative of the stresslet. 

\subsection{Adjoint Micro Problems}\label{sec: adjoint micro problem} 
For this section we consider a fixed micro domain $\micdom$ with boundary data $\micbdryvel$ and a curve $\lineseg$ over which averages of the velocity field are computed. We drop the subscripts $\epsilon$ and $\point$. The goal is then to find Riesz representors, which we denote $\rone,\rtwo$, on the boundary $\micdombdry$ so that $\inner{\rone}{\micbdryvel}_{\micdombdry} = \inner{\micu}{\tangent}_{\lineseg}$ and $\inner{\rtwo}{\micbdryvel}_{\micdombdry} = \inner{\partial_\normal\micu}{\tangent}_{\lineseg}$ for all $\micbdryvel$. The Riesz representors provide a shortcut to compute the slip amount \eqref{eq: slip avg} instead of solving the micro problem~\eqref{eq: micro_problem stokes}:
\begin{equation}
     \pointslip = -\frac{\inner{\micu}{\tangent}_{\lineseg}}{\inner{\partial_\normal\micu}{\tangent}_{\lineseg}} = -\frac{\inner{\micbdryvel}{\rone}_{\micdombdry}}{\inner{\micbdryvel}{\rtwo}_{\micdombdry}}.\label{eq: slip avg adjoint}
\end{equation}
The Riesz representors are solutions to the adjoint problem~\eqref{eq: adjoint_problem} that depend on $\lineseg$ and $\micdombdry$ and can be posed using a variational, classical, or boundary integral formulation of Stokes equations. We consider the latter, and use~\eqref{eq: layerpotential gradient} to express the micro solution as $\micu = \dlp\bigl[\layerpotential^{-1}[\micbdryvel]\bigr]$ and its directional derivative as $\partial_\normal\micu=\partial_\normal\dlp\bigl[\layerpotential^{-1}[\micbdryvel]\bigr]$. 
 
One can derive exact expressions for the Riesz representors $\rone$ and $\rtwo$ directly from the boundary integral formulation \eqref{eq: bie}. Define the \emph{adjoint} of the \emph{double layer potential} restricted to a compact subset $K \subset \interior{\Omega}$ as $\dlp_K^{*}\colon \contspace(K,\Rn{2})\to\contspace(\dombdry,\Rn{2})$ where
\begin{equation}
     \dlp_K^*[\uvec]_i(\ypoint) = 2\normal_j(\ypoint)\int_{K} \stresslet_{ijk}(\point - \ypoint)\uvec_k(\point)\,\mathrm{d}s(\point) \quad\text{for $\ypoint\in \dombdry$.}\label{eq: layerpotential adjoint}
\end{equation}
Boundedness of the adjoint $\dlp_K^*$ follows from the fact that $(\point,\ypoint)\mapsto\stresslet_{ijk}(\point - \ypoint)\normal_j(\ypoint)$ is in $\contspace(K\times\dombdry, \Rn{2\times 2})$ since $\dombdry$ and $K$ are separated by a positive minimal distance. Since both $\dombdry$ and $K$ are compact, the continuity is also uniform. Similarly, the adjoint of $\layerpotential^{-1}$ is
\begin{equation}
     \layerpotential^{-*}[\dens] = (\tfrac{1}{2}\Idop + \dlp_{\princval}^*)^{-1}[\dens].\label{eq: layerpotential inverse ajdoint}
\end{equation}
We can now compute $\inner{\micu}{\tangent}_\lineseg$ and $\inner{\partial_\normal\micu}{\tangent}_\lineseg$ by the representation formula \eqref{eq: layerpotential gradient} for solutions to the Stokes problem, and then move operations to the tangent function $t$ using the adjoint of the double layer potential. The Riesz representor $\rone$ amounts to solving the \emph{adjoint equations of the first kind:}
\begin{equation}
 \rone = \layerpotential^{-*}[\trone], \quad \text{where}\quad \trone = \dlp^{*}_{\lineseg}[\tangent].\label{eq: rone}
\end{equation} 
Hence, one must first compute the adjoint of the \emph{restriction} $\dlp_\lineseg$ of the double layer potential $\dlp$ to the line segment $\lineseg$, then compute the adjoint of and invert the \emph{double layer potential} $\layerpotential$. We similarly compute $\rtwo$ as the solution to the \emph{adjoint equations of the second kind:}
\begin{equation}
 \rtwo = \layerpotential^{-*}[\trtwo] \quad \text{where}\quad \trtwo := (\partial_\normal\dlp)_\lineseg^{*}[\tangent],\label{eq: rtwo}
\end{equation}
where $(\partial_\normal\dlp_\lineseg)^*$ is the adjoint of the composition of the double layer potential and normal derivative operator. 
      
In the following, we consider the special case when $\lineseg$ is a line segment contained in $\micdom$. Then, the \emph{intermediate representors} $\trone$ and $\trtwo$ have closed-form expressions. 

\begin{proposition}\label{prop: intermediate representors}
Let $\lineseg$ be a straight line segment between two points $a, b$ inside $\micdom$. Let $\normal_{\lineseg}$ denote the normal vector of $\lineseg$, and similarly let $\normal(x)$ denote the normal vector at $x\in\micdombdry$. The intermediate Riesz representors $\trone, \trtwo$ can then be expressed as
\begin{multline*}
  \trone(\point) = \frac{1}{\pi}\Bigg(\frac{iz_{\normal_{\lineseg}}}{2}\left(\frac{z_{\normal(\point)}}{z_{\normal_\lineseg}}\log\left(\frac{z_\point - z_a}{z_\point - z_b}  \right)+\left(\frac{z_{\normal(\point)}}{z_{\normal_{\lineseg}}}-2\frac{\overline{z_{\normal(\point)}}}{\overline{z_{\normal_{\lineseg}}}}\right)\overline{\log\left(\frac{z_\point - z_a}{z_\point - z_b}\right)}\right)
  \\
  \shoveright{%
   + \overline{z_{\normal(\point)}}\frac{\ImOp\left[(z_\point - z_a)(\overline{z_\point} - \overline{z_b})\right]}{(\overline{z_\point} - \overline{z_a})(\overline{z_\point} - \overline{z_b})}\Bigg),}
  \\
  \shoveleft{%
  \trtwo(\point) = \frac{1}{\pi}\Bigg(z_{\normal_{\lineseg}}\ImOp\left[\frac{z_{\normal(\point)}(z_b - z_a)}{(z_\point - z_a)(z_\point - z_b)}\right] - \frac{(\overline{z_b}-\overline{z_a})\ImOp\left[z_{\normal_{\lineseg}}z_{\normal(\point)}\right]}{(\overline{z_\point} - \overline{z_a})(\overline{z_\point} - \overline{z_b})}
  }
  \\
  -\frac{(\overline{z_\point} - \overline{z_b})^2\ImOp\bigl[z_{\normal_{\lineseg}}(\overline{z_\point} - \overline{z_a})\bigr]-(\overline{z_\point} - \overline{z_a})^2\ImOp[z_{\normal_{\lineseg}}(\overline{z_\point} - \overline{z_b})]}{(\overline{z_\point} - \overline{z_b})^2(\overline{z_\point} - \overline{z_a})^2}\Bigg),
\end{multline*}
where $z_\point := \point\cdot (1, i)$ for any $\point\in \mathbb{R}^2$ and $\ImOp,\ReOp$ are the imaginary and real parts, respectively. 
\end{proposition}
\begin{proof}
 See Appendix, \cref{sec: appendix: intermediate reps}.
\end{proof}
To find the Riesz representors $\rone,\rtwo$, we have to approximate $\layerpotential^{-*}$ in \eqref{eq: layerpotential inverse ajdoint}, e.g., with a Nyström discretisation or Galerkin projection of the spaces. 
\begin{remark}\label{rmk: bounded adjoint inverse}
 In the general case, $\layerpotential^{-*}$ is defined on the dual space of $\contspace(\dombdry,\Rn{2})$, which is the space of functions with bounded variation $BV([0,2\pi], \Rn{2})$ when $\dombdry$ is 1-dimensional. However, $\layerpotential^{*}\colon \contspace_0(\dombdry, \Rn{2}) \to \contspace(\dombdry,\Rn{2})$ is invertible, where $\contspace_0(\dombdry, \Rn{2})$ is the space of continuous vector fields on $\dombdry$ with zero net flow over the boundary, i.e., 
 \[ 
    \contspace_0(\dombdry, \Rn{2}) := 
    \biggl\{
    v\in \contspace(\dombdry, \Rn{2}) :
    \int_{\dombdry}\!\! v\cdot \normal= 0
    \text{ where $\normal$ is the normal to $\dombdry$}
    \biggr\}.
 \]
 By~\cite[theorem~1, p.~60]{Ladyzhenskaia2014-ko}, we get $\contspace_0(\dombdry, \Rn{2}) = \mathrm{span}\{\normal\}^\perp$.
 Hence, the inverse $\layerpotential^{-*}$ of the adjoint is well-defined as an operator from the space $\contspace(\dombdry, \Rn{2})$ to itself. This is sufficient, since we are only concerned with acting with $\layerpotential^{-*}$ on the continuous functions $\trone$ and $\trtwo$.
\end{remark}
Altogether, the slip amount $\pointslip$ in \eqref{eq: slip avg} can be expressed as a function of the Riesz representors as stated in~\eqref{eq: slip avg adjoint}. \Cref{alg: adjoint HMM} in \cref{sec: hmm} outlines the improved procedure, which starts with precomputing $\rvec_{1}, \rvec_{2}$ for each of the $N$ micro problems.

 \subsection{Discretization and Numerical Solution}\label{sec: discretization}
This section deals with discretising the problems~\eqref{eq: rone} and \eqref{eq: rtwo} to solve for the Riesz representors $\rone$ and $\rtwo$. We discretize each micro domain $\micdombdry$ into points $(\ypoint_i)_{i=1}^J$, \emph{uniformly} distributed in the parameter domain meaning $\ypoint_i = \diffeo(t_i)$, with $t_i = 2\pi i/J$ for $i=1,2,\dots, J$, where $\diffeo\colon [0,2\pi]\to \Rn{2}$ is a parameterization of the boundary. By \cref{prop: intermediate representors}, the intermediate representors $\tilde\rvec_{1,i} = \trone(\ypoint_i)$ and $\tilde\rvec_{2,i} = \trtwo(\ypoint_i)$ can be expressed analytically. We discretize the boundary integral~\eqref{eq: layerpotential adjoint} using the \emph{Nyström method} with \emph{trapezoidal rule}, and evaluate at each discretisation point $\ypoint_i$, i.e., we compute  $\rvec_{m,i}$ for $m=1,2$ by solving
\begin{equation}
 \rvec_{m,i} + \sum_{j=1}^J\layerpotential_{i,j}[\rvec_{m,j}]\Delta \ypointi_j= \tilde\rvec_{m,i}  \quad\text{with $\Delta \ypointi_j := \frac{2\pi}{J}\|\diffeo'(t_j)\|$ and $i=1,\dots,J$,}\label{eq: discrete layerpotential}
\end{equation}
Here $\layerpotential_{i,j}$ is a $2\times 2$ matrix for any fixed $i,j$, and the $n$:th element of the output ($n=1,2$) for a given input $v\in \Rn{2}$ is computed as:
\[
 (\layerpotential_{i,j}[\datavectwo])_n = \sum_{k,l=1}^2\normal_l(\ypoint_j)\stresslet_{nlk}(\ypoint_i - \ypoint_j)v_k.
\]
The resulting system matrix is the same for both representors $m=1,2$, but with different right hand sides. The apparently singular $i=j$ terms are well defined as continuous limits along the boundary (see Appendix, \cref{sec: appendix: boundary limits}). The resulting system of equations is dense but well-conditioned, and we obtain $\rvec_{m}$ with high accuracy using only a moderate number of GMRES iterations. Here we benefit from leveraging the \gls{BIE}, since it allows the data dimension to scale as $J$ instead of $J^2$, which is the dimension of inputs when representing the solution on a 2D $J\times J$ grid. 

In the case of multiscale Stokes flow, once discretized, the adjoint equations~\eqref{eq: rone} and~\eqref{eq: rtwo} result in dense but well-conditioned systems that can be solved efficiently with iterative solvers like \emph{GMRES} \cite{youcef_schultz_gmres}.
    
\section{Learning Riesz Representors}\label{sec: method}
It is clear from~\eqref{eq: discrete layerpotential} that the representors $r_m$ depend non-linearly on the boundary shape $\micdombdry$. 
We propose to approximate the non-linear maps $\micdom \mapsto r_m(\micdom)$ for $m=1,2$ with a \glsfirst{FNO}~\cite{li2022_geofno}. Approaches based on neural networks are well suited to capture the non-linear behavior of these two maps, but current neural network architectures and training protocols rarely achieve the same accuracy as classical numerical solvers based on FEM or BIM. 

Recall that for the multiscale Stokes problem~\eqref{eq: stokes_mom_eps}, one must solve multiple adjoint micro problems of the type~\eqref{eq: rone} and~\eqref{eq: rtwo}. The goal of this section is to present a unified framework for learning the map $\micdom \mapsto r_m(Y)$ for any micro boundary $Y$, where $r_m\in \micbdryspace(\micdombdry, \Rn{M})$. The key idea is to define the map in terms of a \emph{reference domain}, which in our case is the circle $\torus$. The geometry can then be represented by a function, and mapping the boundary to data on the boundary corresponds to a map between function spaces.
    
\subsection{The Geometric Learning Problem}
We let $\diffeospace$ be some set of twice differentiable \emph{Jordan curves} that maps the circle $\torus$ to a simple smooth curve in $\Rn{2}$:
\[ \diffeospace \subset \bigl\{\diffeo\in\contspace_2(\torus; \Rn{2})\mid \phi\;\text{injective} \bigr\}.
\]
Curves in $\diffeospace$ can now be used to generate domains $\micdom \subset \Rn{2}$. 
Simply note that for every curve $\diffeo\in\diffeospace$ there is a (micro) domain $\micdom \subset \Rn{2}$ that has this curve as boundary, i.e., $\micdombdry = \diffeo(\torus)$.
Next, formulating a learning problem over such domains involves introducing a $\diffeospace$-valued random variable. 
This is possible if we can define a probability space that has $\diffeospace$ as sample space. Then, randomly sampling from $\diffeospace$ corresponds to randomly sampling domains in $\Rn{2}$.

\begin{remark}
It is non-trivial to rigorously define a notion of probability measure on the set $\diffeospace$ in cases when it does not admit a finite dimensional parametrisation, and in practice, we restrict ourselves to this case.
\end{remark}

To formulate the learning problem, we adopt the notational convention where random variables are typeset in boldface. As an example, $\rv\uvec$ is a random variable (or random element) whereas $\uvec$ is an element in some vector space.

\begin{definition}[Geometric Supervised Operator Learning]
Let $\diffeospace$ and $\velspace_1\subset\micbdryspace(\torus, \Rn{d_1})$, $\velspace_2\subset\micbdryspace(\torus, \Rn{d_2})$ be measurable sets for some $d_1,d_2 \in \mathbb{N}$. Assume now that $(\rv\diffeo, \rv\datavec, \rv\datavectwo)$ is a $(\diffeospace\times\velspace_1\times \velspace_2)$-valued random variable consisting of a curve $\rv\diffeo$, input data $\rv\datavec$ and output data $\rv\datavectwo$. Given a  \emph{loss function} $\loss\colon \velspace_2\times\velspace_2\to\Reals$, we define \emph{geometric operator learning} as the task of learning a measurable map $\widehat{\oper}\colon\diffeospace\times \velspace_1\to\velspace_2$ that maps from the tuple $(\rv\diffeo,\rv\datavec)$ to $\rv\datavectwo$, by solving the supervised learning problem
\begin{equation}
 \widehat{\oper} \in \!\!\!\argmin_{\oper \colon\diffeospace\times \velspace_1\to\velspace_2}
 \!\!\!
 \mathbb{E}_{(\rv\diffeo, \rv\datavec, \rv\datavectwo)}\Bigl[\loss\bigl(\oper[\rv\diffeo, \rv\datavec], \rv\datavectwo\bigr)\Bigr].\label{eq: learning problem}
\end{equation}
\end{definition} 
    
The geometric learning problem \eqref{eq: learning problem} was introduced to the SciML community by Li et al in~\cite{li2022_geofno} for solving PDE problems on a 2D geometry with a geometry-aware architecture, called the \emph{geometric Fourier neural operator (geo-\gls{FNO})}, that parametrised the maps $\oper \colon\diffeospace\times \velspace_1\to\velspace_2$.
For the loss, we use the relative $\micbdryspace$-error:
\begin{equation}\label{eq: h1 loss}
 \loss(\datavec, \datavectwo) = \frac{\|\datavec - \datavectwo\|_{\Lp^2(\torus, \Rn{d_2})}^2}{\|\datavec\|^2_{\Lp^2(\torus, \Rn{d_2})}} + \frac{\|D\datavec - D\datavectwo\|_{\Lp^2(\torus, \Rn{d_2})}^2}{\|D\datavec\|^2_{\Lp^2(\torus, \Rn{d_2})}},
\end{equation}
where $D\datavec$ is the derivative of $\datavec$ tangent to the boundary. If we take $\rv\datavec$ to be the intermediate representors $(\rv\trone\circ\rv\diffeo, \rv\trtwo\circ\rv\diffeo)\in \micbdryspace(\torus, \Rn{4})$ and $\rv\datavectwo$ to be the representors $(\rv\rone\circ\rv\diffeo, \rv\rtwo\circ\rv\diffeo)\in\micbdryspace(\torus, \Rn{4})$, we jointly recover~\eqref{eq: rone} and \eqref{eq: rtwo} by solving \eqref{eq: learning problem} with $d_1=d_2=4$.

There are three key challenges that arise in working with the geometric operator learning problem~\eqref{eq: learning problem}. 
The first relates to computing the expectation, which requires access to the probability distribution for the $\Reals$-valued random variable $\loss\left(\oper[\rv\diffeo, \rv\datavec], \rv\datavectwo\right)$. We do not have access to this distribution and \cref{sec: learning problem} discusses how to address this challenge by means of training data.
Second, the minimisation in \eqref{eq: learning problem} is over all possible $\velspace_2$-valued measurable mappings  on $\diffeospace\times \velspace_1$. This set is far too large to work with from a computational point of view. \Cref{sec: architecture} suggests finite dimensional parametrisations of these mappings based in neural network architectures.
Finally, the random varible $\loss\left(\oper[\rv\diffeo, \rv\datavec], \rv\datavectwo\right)$ is defined in terms of random variables $\rv\diffeo, \rv\datavec$, and $\rv\datavectwo$ that take values in infinite dimensional spaces. Hence, numerically working with the latter is only possible after appropriate discretisation.

\subsection{Empirical learning problem} \label{sec: learning problem}

Given a finite set of (supervised) training data 
\[ \bigl((\diffeo_1, \datavec^1), \datavectwo^1 \bigr), \ldots, 
    \bigl((\diffeo_K, \datavec^K), \datavectwo^K \bigr) \in  (\diffeospace\times \velspace^{d_1}) \times \velspace^{d_2}
\]
consisting of curves $\diffeo_k$, input data $\datavec^k$ and output data $\datavectwo^k$, and a neural architecture $\oper_\param$ with $\param\in\Param$ as in \eqref{eq: fno}, we define the \emph{empirical learning problem} as minimizing $\param \mapsto \exploss(\param)$ for $\param\in\Param$ where
\begin{equation}
 \exploss(\param) := \frac{1}{K}\sum_{k=1}^K\frac{\|\oper_\param[\diffeo_k, \datavec^k] - \datavectwo^k\|^2}{\|\datavectwo^k\|^2}+\frac{\|D\oper_\param[\diffeo_k, \datavec^k] - D\datavectwo^k\|^2}{\|D\datavectwo^k\|^2}.\label{eq: discrete learning}
\end{equation} 
The main differences from~\eqref{eq: learning problem} is that the expectation is w.r.t. the probability distribution that is given by the training data, data and associated operators are discretised, and the minimization in \eqref{eq: learning problem} is over a finite dimensional space $\Param$. The norm is computed using the trapezoidal rule, and the differential operator $D$ is approximated using FFT.

\subsection{Neural Architecture}\label{sec: architecture}
There is a growing body of research on \emph{neural operators}, which are \emph{neural networks} defined as mappings between function spaces. A common feature of neural operators is \emph{discretisation-invariance}, meaning that the network can handle inputs of varying resolution. Although trained on finite-dimensional data, neural operators are typically designed to approximate operators defined on more general \emph{Banach/Hilbert} spaces, like $\Lp^{2}$. Notable examples include \emph{deep operator networks} (DeepONet)~\cite{karniadakis2019_deeponet}, \gls{FNO}s~\cite{li2021fno}, as well as \emph{BI-TDONet}~\cite{meng2024solvingpde} that was recently proposed for solving \gls{BIE}. Additional example are given in  \cite{cao2023lno,ying2019_bcrnet,ying2019_switchnet,goswami2022physicsinformed,karniadakis2024_en_deeponet}. 
    
We will use a variant of \gls{FNO} called \emph{geo-\gls{FNO}} to represent (discretised) mappings $\oper \colon\diffeospace\times \velspace_1\to\velspace_2$.
Let $\zspace^d=\contspace(\torus, \Rn{d})$ where $\torus$ is the \emph{unit circle}. We define the \emph{spectral convolution operator} $\fblock \colon \zspace^d \to\zspace^d$ (originally referred to as Fourier integral operator in~\cite{li2022_geofno}, we change the name to avoid confusion with the more well-known Fourier integral operator, as defined in PDE analysis) with respect to the kernel $\fkern_{nm}\colon \torus \to \Reals$, $m,n=1,\dots,d$:
\begin{equation}
 \fblock[z]_n(\point) = \invftfm\left[\sum_{m=1}^d\hat\fkern_{nm}(\cdot)\hat z_m(\cdot)\right](\point)\label{eq: fourier integral}
\end{equation}
where $z\in \zspace^d$ with components $z_m$ and $\point\in\torus$, $\ftfm$ is the \emph{Fourier transform} on $\torus$, and $\hat f(k) = \ftfm[f](k)$ is the Fourier transform of a function $f$ evaluated at the point $k$. To discretize, truncate $\hat\fkern$ to the $K$:th mode and replace $\ftfm$ with the \emph{discrete Fourier transform} (DFT). The spectral convolution is combined with a bias function $\bias\in \zspace^d$ and a pointwise linear transformation $\linop\colon \Rn{d}\to \Rn{d}$ to form a single \emph{Fourier neural operator block}:
\begin{equation}
 \fnoblock_{(W,\hat\fkern, b)}[z](\point) = \linop z(\point) + \fblock[z](\point) + \bias(\point).\label{eq: fno block}
\end{equation}
We write $\theta = (W,\hat\fkern, b)$ for notational compactness. The \emph{Fourier neural operator} 
\[ 
    \fno_\theta \colon \zspace^{d_1}\to \zspace^{d_2}
\]
is the composition of $L$ Fourier neural operator blocks $\fnoblock_{\theta^\ell}\colon \zspace^{d}\to \zspace^{d}$ with parameters $\linop^\ell\in\Rn{d\times d}$, $\hat\fkern^\ell\in\Cn{K\times d\times d}$, $\widehat\bias^\ell\in\Cn{K\times d}$ and constant \emph{width} $d$ for $\ell=1,2,\dots,L$, separated by pointwise-acting non-linear \emph{activation functions} $\activation\colon\Reals\to\Reals:$
\begin{equation}
 \fno_\theta = \fnolast\circ\fnoblock_{\theta^L}\circ\activation\circ\fnoblock_{\theta^{L-1}}\circ\dots\circ\activation\circ\fnoblock_{\theta^1}\circ\fnofirst
    \quad\text{where $\theta = (\fnofirst, \theta^1,\theta^2,\dots,\theta^L,\fnolast)$.}
    \label{eq: fno}
\end{equation}
The linear operators $\fnolast\in \Rn{d_1\times d}$ and $\fnofirst\in \Rn{d\times d_2}$ act pointwise similar to $\linop^\ell$. Typically, $\fnofirst$ \emph{lifts} the dimensions of the input to $d > M$ and $\fnolast$ \emph{projects} the output back to the desired dimension.

The geo-\gls{FNO} architecture, originally proposed in~\cite{li2022_geofno}, is a variant of the \gls{FNO} that incorporates a pullback $\diffeo^{-1}$ from arbitrary domains to the \emph{computational reference domain} $\torus$. The geo-\gls{FNO} replaces initial Fourier transform by a \emph{geometric Fourier transform}, defined as $\ftfm_\diffeo[\datavec] := \ftfm\bigl[|\mathrm{det}D \diffeo|\cdot (\datavec\circ\diffeo)\bigr]$. The original work computes $\ftfm_\diffeo$ as an integral on the domain of interest $\micdombdry$ by a change of variables, but we express all integrals on $\torus$ -- which means that the geo-\gls{FNO} $\geofno$ is simply a pushforward of the standard \gls{FNO} $\fno_\param$:
\begin{equation}
 \qquad \geofno_\param[\diffeo, \datavec](\point) = \fno_\param[\diffeo, \datavec\circ\diffeo]\bigl(\diffeo^{-1}(\point)\bigr)\quad\text{for $\point\in\diffeo(\torus)=\micdombdry$ and $ \param\in\Param$.}\label{eq: geofno}
\end{equation}
Here, $\param\in \Param\subset\Rn{(d_1+d_2)d + d^2L+d(1+d)KL}$ is a real-valued parameter that parameterizes the \gls{FNO}. We train the \gls{FNO} using the learning problem~\eqref{eq: learning problem}, and only construct the geo-\gls{FNO} after training. 
    
\subsection{Computational complexity at inference}
Given a $J$-point discretization of the micro problem boundary, classical numerical approaches must solve a dense linear $2J\times 2J$ system, which has complexity $\mathcal{O}(J^2P)$ for an iterative solver like GMRES, with $P$ iterations. Fast summation methods as the fast multipole method (FMM) or FFT based methods using Ewald decomposition can reduce this cost to $\mathcal{O}(JP)$ or $\mathcal{O}(J\log J P)$, but such approaches are advantageous only for sufficiently large systems. For example,~\cite{fmm_for_3d_stokes} finds conventional direct summation methods advantageous for $J<3000$ for FMM applied to BEM in 3D Stokes flow. Methods based on a Fourier decomposition of the problem data and solution have complexity $\mathcal{O}(K_{\mathrm{max}}^2P)$ where $K_{\mathrm{max}}$ is the maximal wave number. Importantly, the system matrix of these methods needs to be rebuilt for each new domain, whereas a neural network operates on stored weights.
    
To compare, the cost of an \gls{FNO} at inference is $\mathcal{O}(DCJ\log J + DK_{\mathrm{max}}C^2)$, where $C$ is the number of channels, $D$ is the network depth, and $K_{\mathrm{max}}$ the maximal wave number in the Fourier domain. However, estimates for the scaling of $C$ and $D$ typically rely on proving that the neural network can approximate a spectral solver~\cite{lanthaler2024nonlocalitynonlinearityimpliesuniversality}, which causes $C$ and $D$ to scale with the desired accuracy. We will not explore such complexity results in this paper.

\section{Error Analysis}\label{sec: error analysis}
The discrepancy between the optimal solution to~\eqref{eq: learning problem} and its empirical counterpart~\eqref{eq: discrete learning} can be split into three main terms representing the approximation error, generalization error and discretization error, respectively.
    
\Cref{sec: universality} discusses the \emph{approximation error} that quantifies the difference between the optimal solution and the best possible solution in the search space defined by the neural network architecture. The same section also mentions the \emph{generalization error}, which quantifies the difference between the expected loss~\eqref{eq: learning problem} on the full data distribution (corresponds to having all possible training data) and the empirical loss~\eqref{eq: discrete learning} defined by the finite training data. As such, it relates to the expressivity of the search space. However, the emphasis in this work (\cref{sec: model error}) is on assessing how the generalization error for learning the Riesz representors propagates to an error in the point wise estimates for the slip coefficient (\cref{cor: slip gen error}), which after interpolation causes a global error in the macro problem (\cref{thm: bounded homogenized slip error}). 

\subsection{Universality and Generality}\label{sec: universality}
Work on developing universal approximation theory for neural operators dates back to Chen and Chen in 1995, who showed that continuous operators on compact subsets of Banach spaces can be approximated to arbitrary precision by a shallow neural operator~\cite{chenchen1995}. The architecture proposed by Chen would later be used under the name ``DeepOnet''~\cite{karniadakis2019_deeponet}. Recently, it was shown that the \gls{FNO} in particular is a universal approximator of continuous operators between some Banach spaces, specifically Sobolev spaces~\cite{kovachki2021universal} and the space of continuously differentiable functions~\cite{lanthaler2024nonlocalitynonlinearityimpliesuniversality}. 
    
The operator $\layerpotential_{\dombdry}^{-*}\colon\contspace(\dombdry, \Rn{2})\to\contspace(\dombdry, \Rn{2})$ defined in~\eqref{eq: layerpotential inverse ajdoint} is continuous (see~\cref{rmk: bounded adjoint inverse}) and can therefore be approximated by the \gls{FNO} to arbitrary precision on a \emph{fixed geometry} $\dombdry$. We aim to learn the operator for any geometry, by considering the map $(\diffeo, \tilde\rvec_\ell)\mapsto \layerpotential_{\diffeo(\torus)}^{-*}[\tilde\rvec_\ell]$, which involves the curve $\diffeo$ as well as the data $\tilde\rvec_\ell$ for $\ell=1,2$. Universality of \gls{FNO} for this type of operator is less obvious, since the set of Jordan curves is not a Banach space. 

Luckily, the micro problems in this work can be well approximated by a coarse numerical method, which is a continuous operator on the discretized data. Discretization amounts to evaluation on a finite set of points (see \cref{sec: discretization}), and the evaluation operator is continuous on $\contspace(\dombdry,\Rn{2})$. Hence, the composition of the discretization followed by the numerical routine is itself continuous and therefore, an \gls{FNO} can approximate the numerical routine to arbitrary precision.  
    
\subsection{Model and Discretization Error}\label{sec: model error}
In the context of \gls{HMM}, there are additional error terms than those mentioned in~\cref{sec: universality} that affect the solution. One is the \emph{model error} inherent in the \gls{HMM}, which is the difference between the \gls{HMM} approximation $u_{\mathrm{HMM}}$ and the fully resolved solution $u_\epsilon$ to~\eqref{eq: stokes_mom_eps}. The model error was analyzed in~\cite{carney2021heterogeneous} for a specific toy problem, but so far there is no rigorous analysis of it in the general case. 

In this section we study the \emph{\cpl error} that is made with the learned approach to compute the slip amount. Our approach learns Riesz representors (that depend only on the geometry) instead of the slip amount itself (that depends also on the boundary conditions), which allows us to train on much less data, but significantly complicates the analysis. We will prove two main results. First, given sufficient alignment conditions on the background flow, we can show (\cref{cor: slip gen error}) that a bounded training error will result in a bound on the expected error of the slip amount $\pointslip(\point)$, given one arbitrary fixed micro domain centered at some point $\point$ on the rough boundary. 
    
Second, we show (\cref{thm: bounded homogenized slip error}) that a bounded expected error of the slip amount leads to a bound on the interpolated slip function $\interslip$ on the boundary, given sufficient regularity of the slip amount. This result uses a particular choice of sampling and interpolation strategy and stronger assumptions about the properties of the slip operator (\cref{asm: slip operator taylor}). We include justification for these assumptions in \cref{sec: lipschitz slip}. Finally, we connect the error bound from \cref{thm: bounded homogenized slip error} to the error in the macroscopic solution in \cref{prop: pde error bound}, using standard PDE energy norm estimates.

\subsubsection{A Pointwise Error Estimate}
To relate the training error, which is independent of boundary data, to the error in the slip coefficient, we must pose an additional constraint on the representors $\rone$, $\rtwo$ -- and in particular how well aligned they are to the boundary data $\micbdryvel$ in the sense of \cref{asm: well aligned}. We will avoid involving the characteristic micro scale $\epsilon$ in this section, but it becomes relevant for \cref{thm: bounded homogenized slip error}. 

\begin{definition}\label{asm: well aligned}
  Let $H$ be a Hilbert space endowed with the inner product $\inner{\cdot}{\cdot}$ and norm $\|\cdot\|$. Two elements $v$,$w$ of $H$ are \emph{$\eta$--aligned}  for some $0<\eta<1$ if 
 \[
     \frac{\inner{v}{w}}{\|v\|\|w\|}\geq \eta.
 \]
 Furthermore, let $\rv u, \rv v$ be $H$-valued random variables. Then, $\rv u,\rv v$ are almost surely uniformly $\eta$--aligned if
 \[
    \mathbb{P}\bigl(\{\rv u \text{ and } \rv v \text{ are not $\eta$--aligned}\} \bigr)=0. 
 \]
\end{definition}
We can now bound the error in the slip coefficient in the statistical sense, given that the correlation between the roughness structure at two different points on the wall decays sufficiently fast with respect to the distance between the points. Our main result in this section relies on the following general Lemma:
\begin{lemma}\label{lemma: expected ratio}
Let $\rv a_1, \rv a_2, \rv b_1, \rv b_2$ be random variables such that $\left|\rv a_1/\rv  a_2\right| < C$ holds a.s.
Furthermore, assume 
\[
     \expectp{\left(\frac{\rv a_1- \rv b_1}{\rv a_1}\right)^2} \leq \delta_1^2
     \quad\text{and}\quad 
     \expectp{\left(\frac{\rv a_2-\rv b_2}{\rv b_2}\right)^2} \leq \delta_2^2.
\]
Then $\expectp{\Bigl|\rv a_1/\rv a_2-\rv b_1/\rv b_2\Bigr|}<C(\delta_1+\delta_2+\delta_1\delta_2)$.
\end{lemma}

The main result of this section follows almost immediately from the above lemma:  
\begin{corollary}\label{cor: slip gen error}
 Let $\rv\micdombdry$ be a random domain generated by a $\diffeospace$-valued random element. We write $\inner{\cdot}{\cdot}_{\rv\micdombdry}$ and $\|\cdot\|_{\rv\micdombdry}$ for the standard $L^2$-inner product and norm over $\rv\micdombdry$, respectively. Let $\oper[\rv\trone,\rv\micdombdry]$ and $\oper[\rv\trtwo,\rv\micdombdry]$ be approximate solutions to the Riesz representors $\rv\rone$ and $\rv\rtwo$ on $\rv\micdombdry$, all taking values in $\contspace(\rv\micdombdry,\Rn{2})$. 
 Assume $\rv\rone$ and $\rv\oper[\rv\trtwo,\rv\micdombdry]$ are both almost surely uniformly $\eta_1$- and $\eta_2$-aligned with $h$, respectively. Furthermore, assume also that the slip amount $\inner{\rv\rone}{\micbdryvel}_{\rv\micdombdry}/\inner{\rv\rtwo}{\micbdryvel}_{\rv\micdombdry}$ is almost surely bounded by $C$ for some $C>0$. Finally, assume that the training losses are bounded by $\delta>0$ in the sense that
 \begin{equation}\label{eq: training bound}
     \expectp{\frac{\|\oper[\rv\trone,\rv\micdombdry]-\rv\rone\|^2_{\rv\micdombdry}}{\|\rv\rone\|^2_{\rv\micdombdry}}}\leq \delta^2
     \qquad\text{and}\qquad \expectp{\frac{\|\oper[\rv\trtwo,\rv\micdombdry]-\rv\rtwo\|^2_{\rv\micdombdry}}{\|\oper[\rv\trtwo,\rv\micdombdry]\|_{\rv\micdombdry}^2}}\leq \delta^2.
 \end{equation}
 Then, the error in the estimated slip amount is bounded as
 \[   \expectp{\left|\frac{\inner{\rv\rone}{\micbdryvel}_{\rv\micdombdry}}{\inner{\rv\rtwo}{\micbdryvel}_{\rv\micdombdry}} - \frac{\inner{\oper[\rv\trone,\rv\micdombdry]}{\micbdryvel}_{\rv\micdombdry}}{\inner{\oper[\rv\trtwo,\rv\micdombdry]}{\micbdryvel}_{\rv\micdombdry}}\right|}\leq C\left(\delta(\eta_1^{-1}+\eta_2^{-1})+\delta^2\eta_1^{-1}\eta_2^{-1}\right).
 \]
\end{corollary}
\begin{proof}
 See \cref{sec: appendix: slip estimates}.
\end{proof}
Equation~\eqref{eq: training bound} approximates the training loss when $\rv\micdombdry$ is sampled from the same distribution that we trained on. Note that even though our construction requires $\rv\rone,\rv\rtwo,\rv\trone,\rv\trtwo$ to be random variables, they are in practice $\rv\micdombdry$-measurable. 

The training error in~\eqref{eq: training bound} is slightly different from the bound used in~\eqref{eq: h1 loss}. First, we need to control the inverse of the norm of $\oper[\rv\trtwo,\rv\micdombdry]$, which is why it appears in the denominator. Second, we use the $H^1$-norm in practice, but this is not a problem for the theory since a bound in the $H^1$-norm also yields a bound in the $L_2$-norm. 
\begin{remark}
 We need to assume that training data is $\eta$-aligned. This can be hard to verify in general, but it can be done for some special cases. One is when the micro problem has periodic boundary conditions as shown in \cref{prop: aligned bounds}.
\end{remark}

\subsubsection{An Interpolation Error Estimate}
We now aim to show that interpolating over the boundary results in an error that is bounded. For the sake of analysis, it makes sense to model the boundary as a stochastic process. In order to control the interpolation error,  we need to assume that the slip amount is bounded and Lipschitz as a function on the smooth surface $\roughness_0$. For the purpose of this section, we assume that the rough wall is parameterized by a normal offset:
\[
    \roughness_\epsilon = \bigl\{\point+\epsilon\rv\varphi(\point_t/\epsilon)\normal(\point)\mid \point\in \roughness_0\bigr\},
\]
where $\rv\varphi$ is a random $C^2$ function.
\begin{definition}\label{setting: slipop}
 The slip amount $\alpha\colon\roughness_0\to\mathbb{R}$ at the point on the smooth surface $\point \in \roughness_0$ is computed as follows: 
 \begin{itemize}
     \item construct a micro domain $\micdom_x$ such that $\micdombdry_x \cap \roughness_\epsilon = \mathcal{B}_\epsilon(x)\cap \roughness_\epsilon$, where $\mathcal{B}_\epsilon(x)$ is an $\epsilon$-sized ball around $x$.
     \item solve the adjoint problems~\eqref{eq: rone}--\eqref{eq: rtwo} on $\micdom_x$, and
     \item compute the slip coefficient using~\eqref{eq: slip avg}.
 \end{itemize}
\end{definition}
\begin{assumption}\label{asm: slip operator taylor}
 The \emph{slip amount} is uniformly bounded and Lipschitz. Specifically, there exists $C>0$ and $\gamma_1,\gamma_3>0$ such that
 \[
     \alpha(x) \leq \gamma_1 C
     \quad\text{and}\quad 
     |\alpha(x)-\alpha(y)|\leq \frac{\gamma_1}{\gamma_3} C|x-y|
     \quad\text{for any $x,y\in \roughness_0$}
 \]
 where $\lim_{\epsilon\to 0}\epsilon/\gamma_1=\lim_{\epsilon\to 0}\gamma_1/\gamma_3 = 0$.
\end{assumption}
    
\begin{remark}
The condition that $\rv\varphi$ is continuous is needed for $x \mapsto \alpha(x)$ to be well defined. \cref{asm: slip operator taylor} is motivated by a simple example in \cref{sec: lipschitz slip}. Roughly, $\gamma_1$ is the position of the evaluation line $\lineseg$, and $\gamma_3$ is the finite width of the micro problem.

\end{remark}
Finally, we define an interpolation operator. Although we use a general construction, a simple piecewise linear interpolation is sufficient to satisfy our requirements. 
\begin{definition}\label{def: interpolation}
 Let $\{\point_n\}_{n=1}^N\subset K$ be a discretization of a set $K$. We define an interpolant as a linear operator $\mathcal{I}\colon \Rn{N}\to C(K,\Reals)$ that satisfies $\mathcal{I}\bigl[(\alpha_j)_{j=1}^N\bigr](\point_n) = \alpha_n$ for all $n=1,\dots,N$. We additionally require $\mathcal{I}$ to be continuous and such that there is some $C>0$ where
 \[
    \sup_{z\in K} \Bigl|\mathcal{I}\bigl[(\alpha_j)_{j=1}^N\bigr](z)\Bigr| \leq C\max_{n=1,\dots,N}|\alpha_n|
    \quad\text{holds for any $(\alpha_j)_{j=1}^N\subset \Rn{N}$.}
 \]
 Finally, we say $\mathcal{I}$ tightly interpolates the Lipschitz function $\alpha\colon K\to \Reals$ to the first order if 
 \[     \Bigl|\mathcal{I}\bigl[(\alpha(\point_j))_{j=1}^N\bigr](z) - \alpha(z)\Bigr|\leq C L \max_{z\in K}\min_{n=1,\dots,N}\|z-\point_n\|,
 \] 
 holds with $C>0$ independent of $L>0$ (the Lipschitz constant of the function $\alpha$).
\end{definition}
We are now ready to state the main result of this section.
\begin{theorem}\label{thm: bounded homogenized slip error}
 Let $\rv\varphi, \rv{\alpha}$ be as in \cref{setting: slipop}, and suppose $\rv{\alpha}$ satisfies \cref{asm: slip operator taylor} uniformly with $C,\gamma_1,\gamma_3$ known. Next assume that we have sample points $\{\point_n\}_{n=1}^N$ in $\roughness_0$ with 
 \[\max_{z\in \roughness_0}\min_{n=1,\dots,N}\|z-\point_n\|\leq |\roughness_0|/N
 \quad\text{where $|\roughness_0|$ is the length of the domain.}
 \]
 Finally, assume $\{\rv{\tilde\alpha}_n\}_{n=1}^N$ estimates $\rv\alpha$ with an expected error bounded as
 \[ \expectp{\bigl|\rv{\tilde\alpha}_n-\rv{\alpha}(\point_n)\bigr|}\leq \delta \gamma_2
 \quad\text{for some $\delta, \gamma_2>0$ and all $n=1,\dots,N$.}
 \]
 Then, for any interpolant $\rv{\tilde\alpha}(z)=\mathcal{I}\bigl[(\rv{\tilde\alpha}_n)_{n=1}^N\bigr](z)$ with $z\in\roughness_0$ that interpolates $\rv{\alpha}$ to the first order (\cref{def: interpolation}), there exists $N$ so that the error is bounded over $\roughness_0$ as
 \[     \expectp{\sup_{z\in\roughness_0}\bigl|\rv{\tilde\alpha}(z) - \rv\alpha(z)\bigr|} 
     \leq 2\gamma_1 C\sqrt{|\roughness_0| \delta \gamma_2/(\gamma_1\gamma_3)}.
 \]
\end{theorem}
\begin{proof}
The proof is essentially a repeated application of the triangle inequality. Let $z\in \roughness_0$ be fixed. Combining linearity, continuity and tightness of $\mathcal{I}$ with Lipschitz continuity of $\rv\alpha$,
\begin{align*}
     \bigl|\rv{\tilde\alpha}(z) - \rv\alpha(z)\bigr| &= \Big|\mathcal{I}\bigl[(\rv{\tilde\alpha}_n-\rv\alpha(x_n))_{n=1}^N \bigr](z) + \mathcal{I}\bigl[(\rv\alpha(x_n))_{n=1}^N\bigr](z) - \rv\alpha(z)\Big|\\
     &\leq C\max_{n=1,\dots,N}\bigl|\rv{\tilde\alpha}_n-\rv\alpha(x_n)\bigr| + CL\max_{z\in\roughness_0}\min_{n=1,\dots,N}\|z-x_n\|\\
     &\leq C\max_{n=1,\dots,N}\bigl|\rv{\tilde\alpha}_n-\rv\alpha(x_n)\bigr| + \frac{\gamma_1 C|\roughness_0|}{\gamma_3 N}.
\end{align*}
Taking the supremum over $z \in \roughness_0$ followed by the expectation yields
\begin{align*}
 \expectp{\sup_{z\in\roughness_0}\bigl|\rv{\tilde\alpha}(z) - \rv\alpha(z)\bigr|} &\leq \frac{\gamma_1 C|\roughness_0|}{\gamma_3 N}+C\expectp{\max_{n=1,\dots,N}\bigl|\rv{\tilde\alpha}_n-\rv\alpha(x_n)\bigr|}\leq \frac{\gamma_1 C|\roughness_0|}{\gamma_3 N}+ \gamma_2\delta C N.
\end{align*}
The above estimate is minimized when $N = \sqrt{|\roughness_0|\delta^{-1}\gamma_1/(\gamma_2\gamma_3)}$, giving the result.
\end{proof}

A key observation here is that taking the supremum norm results in a loss of accuracy -- a point wise error of $\delta$ gives rise to a interpolation error $\mathcal{O}(\sqrt{\delta})$. It is possible to improve this result to $\mathcal{O}(\delta\log(1/\delta))$, but that requires strong assumptions on the error distribution (existence of a moment generating function), which are outside the scope of this work.

\subsubsection{Combining the Two Estimates}
Given \cref{thm: bounded homogenized slip error} and \cref{cor: slip gen error}, we can now formulate a theorem relating the generalisation error to the error in the slip amount. We refer to the toy example studied in \cref{sec: appendix: alignment estimates} as a motivation for our choices of $\eta_1$ and $\eta_2$. For this final step, we need to assume that the distribution of the random micro domains is stationary, which we will do by simply assuming that \cref{cor: slip gen error} holds uniformly for all $x\in \roughness_0$. 

\begin{theorem}\label{thm: final interpolation estimate}
 Let $\rv\varphi, \rv{\alpha}$ be as in \cref{setting: slipop}, and suppose $\rv{\alpha}$ satisfies \cref{asm: slip operator taylor} uniformly with $C,\gamma_1,\gamma_3$ known. Assume the setting in \cref{cor: slip gen error} holds uniformly for all micro problems $\rv\micdombdry_{x}$ with $x\in \roughness_0$, and for all $\epsilon>0$ with $\eta_1 = \gamma_1/\gamma_2$ and $\eta_2 = 1/\gamma_2$, and $\gamma_2$ satisfying $\lim_{\epsilon\to 0}\gamma_1/\gamma_2=\lim_{\epsilon\to 0}\gamma_2/(\gamma_1\gamma_3)=0$. Let $\{x_n\}_{n=1}^N$ discretize the smooth boundary $\roughness_0$ with resolution $|\roughness_0|/N$, and define $\{\rv{\tilde\alpha}_n\}_{n=1}^N$ by 
 \[
     \rv{\tilde\alpha}_n := \frac{\inner{\oper[\rv\trone,\rv\micdombdry_{x_n}]}{h}_{\rv\micdombdry_{x_n}}}{\inner{\oper[\rv\trtwo,\rv\micdombdry_{x_n}]}{h}_{\rv\micdombdry_{x_n}}}.
 \]
 Then, for any weighted interpolant $\rv{\tilde\alpha}(z)=\mathcal{I}[(\rv{\tilde\alpha}_n)_{n=1}^N](z)$ that interpolates $\rv{\alpha}$ to the first order (\cref{def: interpolation}), there is a suitable choice of $N$ so that the interpolated slip amount is bounded in error over $\roughness_0$:
 \[
     \expectp{\sup_{z\in\roughness_0}\bigl|\rv{\tilde\alpha}(z) - \rv\alpha(z)\bigr|} 
     \leq 2\gamma_1C\sqrt{\delta|\roughness_0|\gamma_2/(\gamma_1\gamma_3)}.
 \]
\end{theorem}
\begin{proof}
 The result follows from combining \cref{cor: slip gen error} and \cref{thm: bounded homogenized slip error}.
\end{proof}

With a bound on the expected error in supremum norm, standard sensitivity analysis for linear elliptic PDE produce an error in the macroscopic solution that is of the same order as the error in the slip coefficient: 

\begin{proposition}\label{prop: pde error bound}
    Let $u,v \in H^1(\dom_0,\Rn{2})$ solve the macro equations~\eqref{eq: macro_problem stokes}, with the same Dirichlet data $g\in C(\dombdry_0\setminus\roughness_0,\Rn{2})$ and slip functions $\alpha,\beta\in C(\roughness_0,\Reals)$, respectively. Assume $\alpha,\beta$ are both uniformly bounded from below by some $\alpha_0>0$ and from above by $1$. Then, there is a constant $C>0$ independent of $\alpha,\beta$ so that $u-v$ is bounded in the $H^1$-seminorm:
    \[
        \left|u - v\right|_{H^1(\dom_0,\Rn{2})}\leq \alpha_0^{-1}C\sup_{x\in\roughness_0}|\alpha(x)-\beta(x)|.
    \]
\end{proposition}

\begin{proof}
    See \cref{sec: appendix: pde bounds}.
\end{proof}

Taking the expectation on both sides, we then obtain a statistical bound on the macroscopic error.

\section{Numerical Results}\label{sec: experiments}
In the following sections, we outline the data generation and training method used for the \gls{FNO-HMM} solver, and evaluate its performance compared to a \gls{BIE-HMM}. We distinguish between different types of errors. 
Namely, let 
\begin{align*}
 \sobspace^1(\dom_0)^2 \ni \tuhmm
 &= \text{\gls{FNO-HMM} solution,}
 \\
 \sobspace^1(\dom_0)^2 \ni \uhmm
 &= \text{\gls{BIE-HMM} solution with accurate micro solver (error $\ll \epsilon$)}
 \\
 \sobspace^1(\dom_\epsilon)^2 \ni \macu_\epsilon
 &= \text{true solution to the multiscale flow problem~\eqref{eq: stokes_mom_eps}.}
\end{align*}
Furthermore, let $\{(s_1^n,s_2^n)\}_{n=1}^N$ be the \gls{FNO-HMM} estimates of the Riesz representors over the $N$ micro problems, and let $\{(\rone^n,\rtwo^n)\}_{n=1}^N$ denote the true Riesz representors, with micro domains $\{\micdom^n\}_{n=1}^N$ and $\rone^n,\rtwo^n,s_1^n,s_2^n$ all in $\contspace(\micdombdry^n, \Rn{2})$. We define the \emph{Normalized $\Lp^2$-error} $\nrmse{\uvec}{\vvec}{M}$ of $\uvec$ relative to $\vvec$ over an arbitrary compact set $M\subset \overline{\dom}_0$ as follows:
\begin{equation}\label{eq: normalized error}
 \eletter(\uvec, \vvec; M) := \frac{\|\uvec-\vvec\|_M}{\|\vvec\|_M},
\end{equation}
where $\|\cdot\|_M$ is the $\Lp^2$-norm over a measurable set $M$. Using the above expression, we define the \emph{model error} $\emdl$, \emph{\cpl error} $\ecpl$, \emph{total error} $\etot$, \emph{smoothing after-error} $\elo$ and \emph{smoothing before-error} $\ehi$ as references for lower and upper bounds on the error:
 as follows:
\begin{xalignat}{2}
\emdl(M) &:= \nrmse{\uhmm}{\macu_\epsilon}{M},&\ecpl(M) &:= \nrmse{\tuhmm}{\uhmm}{M}, \nonumber \\
\elo(M) &:= \nrmse{\phi_\epsilon\circledast\macu_\epsilon}{\macu_\epsilon}{M},&\ehi(M) &:= \nrmse{\macu_0}{\macu_\epsilon}{M},\label{eq: hi}\\
\etot(M) &:= \nrmse{\tuhmm}{\macu_\epsilon}{M}.&&\nonumber
\end{xalignat}
Here, $\phi_\epsilon$ is a low-pass filter with a cutoff at a wavelength of $\mathcal{O}(\epsilon)$, and $\macu_0$ is the naive solution obtained by solving Stokes equations~\eqref{eq: macro_problem stokes} with no-slip imposed at the wall~$\roughness_0$ instead of the Navier slip condition. The set $M$ will usually be a \emph{level set} of the wall-distance function, consisting of points a distance $d$ from the macroscopic boundary $\roughness_0$:
\begin{equation}\label{eq: off-wall eval set}
 M_\delta = \left\{\point\in\dom_0\;\mid\; \dist(\point,\roughness_0)= \delta\right\}, \quad \text{where}\quad \dist(\point, \roughness_0) := \min_{\ypoint\in\roughness_0}\{\|\point-\ypoint\|\}.
\end{equation}
Usually, we look at $\delta$ in $(0, 20 \epsilon)$. Lastly, we define the $\emph{micro error}$ $\emic$ as the maximal error in the Riesz representors:
\begin{equation}\label{eq: micro error}
 \emic := \max_{n\in\{1,2,\dots,N\}}\bigl\{\nrmse{s_1^n}{\rone^n}{\micdombdry^n}+\nrmse{s_2^n}{\rtwo^n}{\micdombdry^n}\bigr\}.
\end{equation}

\subsection{Data Generation}
We generate training data in three steps. First, the micro domain $\rv\micdom$ is defined similarly to the unit cell used in homogenization:
\begin{equation}
 \rv\micdom := \smoothop\Bigl(\bigl\{\point\in\Rn{2}\;\mid\; 0\leq \pointi_1\leq 1,\; \gpfunc(\pointi_1)\leq \pointi_2\leq \rvmicdomh\bigr\}\Bigr)\label{eq: domain generation}
\end{equation}
where $\gpfunc\colon\Reals \to \Reals$ is a 1D \emph{Gaussian process} with \emph{exponential kernel} that simulates a rough wall, $\rvmicdomh>\|\gpfunc\|_\infty$ is a random height and $\smoothop$ is a \emph{smoothing operator} that rounds the corners of $\rv\micdom$ so that the maximal curvature of $\gpfunc$ is also the maximal curvature of $\rv\micdom$.  The smoothing operator ensures that no unnecessary high frequency content is introduced to the micro problem apart from the roughness. From $\rv\micdom$ we approximate the corresponding Jordan curve $\rv\diffeo\in\diffeospace$ by uniform discretization. In particular, we use in total eight polynomials of degree 5 (2 coordinates times 4 corners) to ensure that $\rv\diffeo$ is twice differentiable everywhere.

Second, we generate a random line segment $\rv\lineseg = \{\point\in\Rn{2}\mid 0\leq x_1\leq 1,\; x_2=\rvlinesegh\}$ that passes through $\micdom$ at some random height $\|\gpfunc\|_\infty < \rvlinesegh<\rvmicdomh$. ~\cref{fig: micro domain examples} shows samples from $\rv\micdom,\rv\lineseg$ that were generated from this process. The domain $\rv\micdom$ and line segment $\rv\lineseg$ define intermediate representors $\rv\trone,\rv\trtwo$ as described in \cref{sec: dimensionality reduction}, from which we compute Riesz representors $\rv\rone,\rv\rtwo$ using the boundary integral formulation~\eqref{eq: rone} and~\eqref{eq: rtwo}. 

\begin{figure}[!ht]
 \includegraphics[width=\textwidth]{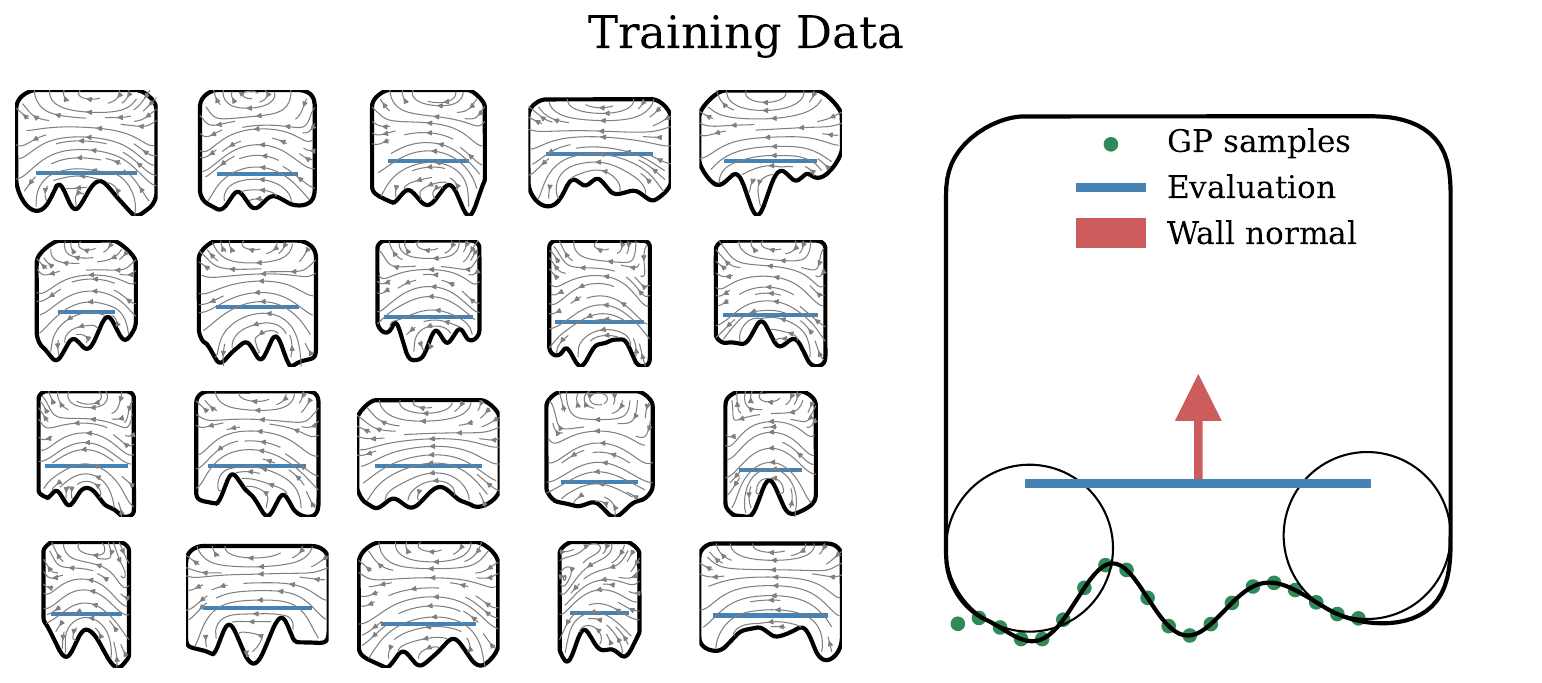}
 \caption{Examples of randomly generated problem data. The line $\lineseg$ is shown in blue. The green points are taken from a Gaussian Process, and the circles indicate the curvature bound on the connecting corners. The stream plots shows the flow field that would induce highest mean flow across the blue line.}
 \label{fig: micro domain examples}
\end{figure}

\subsection{Training}
We sampled a dataset of $K=20\;000$ micro domains $\micdom_k$ and line segments $\lineseg_k$, and computed the corresponding discretized intermediate representors $\tilde\rvec_{1}^{k},\tilde\rvec_{2}^k \in \Rn{2M}$ using \cref{prop: intermediate representors} and Riesz representors $\rvec_{1}^{k},\rvec_{2}^{k}\in\Rn{2M}$ for all $k=1,2,\dots,K$ using \gls{BIE} with a GMRES solver, with $J$ large enough to be accurate. We then trained the geo-FNO to minimize empirical loss~\eqref{eq: discrete learning} with the Adam optimiser~\cite{kingma2017adam} and a learning rate of $10^{-3}$ and no explicit regularization. We used a batch size of 32, and trained for $40\;000$ epochs on an NVIDIA GeForce RTX 3090, which took around 15 minutes per run. Due to statistical uncertainty involved in the training process, the learned operator is a \emph{random variable} that depends on the training data. We trained $\nretrain$ times with different parameter initializations and different random seeds for ADAM, to estimate the sensitivity to training and initializations. 

\subsection{Evaluation}
The training was run on a single NVIDIA GeForce RTX 3090 with 24 GB memory. Most of the evaluation tasks were carried out on an Intel Xeon 4215R 3.2GHz 2400MHz 8C CPU, to keep the comparisons between the FNO and BIM-based precomputation as fair as possible. 

\begin{figure}[!ht]
 \includegraphics[width=\textwidth]{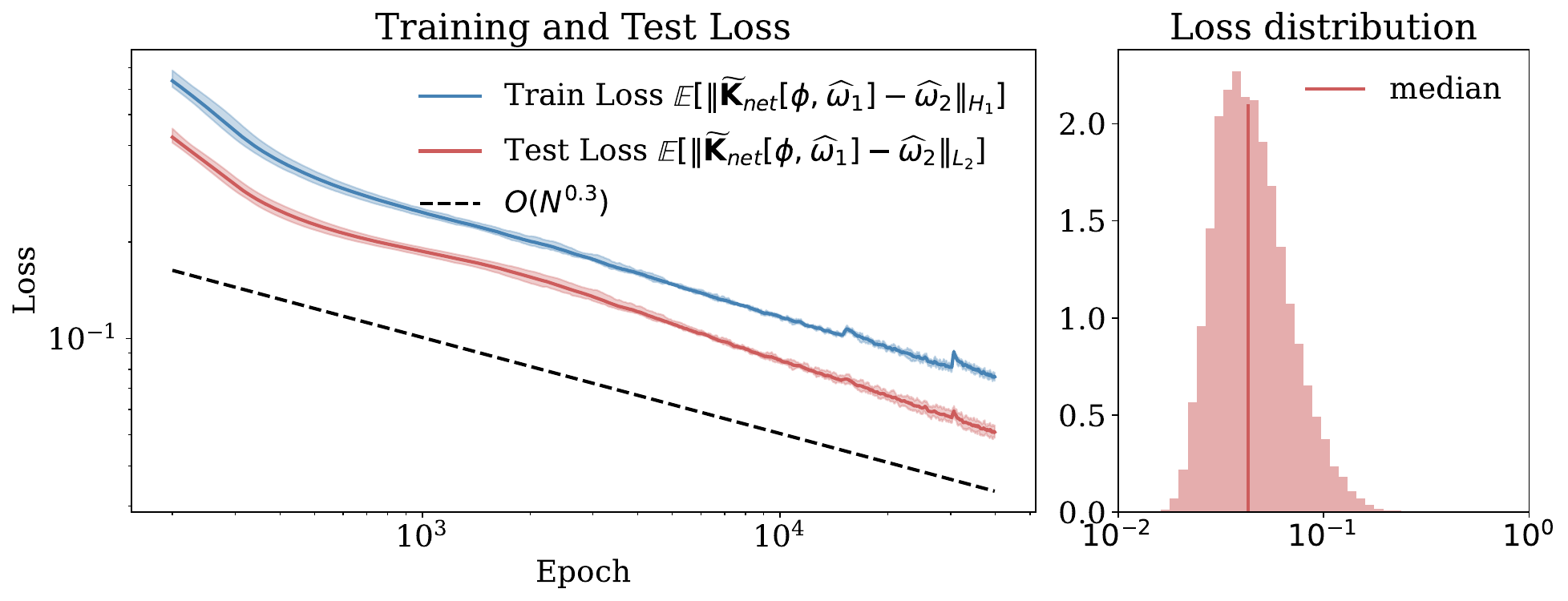}
 \caption{The left plot shows training and test loss \eqref{eq: discrete learning} for $\nretrain$ training runs with $2\cdot 10^5$ data samples and $40 000$ epochs. The full range of losses is marked with a shaded area. The right plot shows the error distribution (the distribution of the terms in~\eqref{eq: discrete learning} before summation) on the test set at the final training step.}
\end{figure}

We evaluate the learned operator on three test cases. First, we test accuracy and speed of the learned Riesz representors on different micro domains. Then, we implement the full \gls{HMM} method on a \emph{periodic channel} with a rough wall at the bottom and flat wall at the top, whose roughness is a sine function with wavelength $\epsilon$ and amplitude $\epsilon$. Third, we run \gls{HMM} on a more complicated domain (\cref{fig: circular inset}) with \emph{circular insets} whose walls are generated by a Gaussian random field that dictates the offset from the circle in the outwards normal direction. In both cases, the flow is driven by \emph{tangental flow conditions} at the smooth boundaries, and non-slip at the rough wall. For the two latter cases, we compare the learned operator to a full resolution solver, and a naive solver that smooths the roughness and solves the problem on the smoothed domain with no-slip boundary conditions. Lastly we demonstrate the capacity of \gls{FNO-HMM} on a hard problem for which running a full solver is too costly on our system.

\subsubsection{Case I: Random Domains}
\cref{fig: micro solver generalizes}a--c show prediction ranges of 6 networks trained with different seeds for parameter initialization and batch randomization but the same training data set. In (a), a domain that resembles the training data, (b), domains with higher curvature than the data set and (c), domains where the wall cannot be parameterized as normal displacement of the smooth wall, also not in the training data. We distinguish between \emph{model uncertainty} (the prediction range), and \emph{bias} (systematic error over all predictions). The model uncertainty in case (b) is comparable to (a), but the bias is considerably higher in the high curvature region $[0,\pi/2]$. This is likely due to the fact that the training data is generated with a fixed curvature bound, and the model has not learned to extrapolate to higher curvatures. In addition to increased bias, (c) also results in higher model uncertainty compared to (a).

\begin{wrapfigure}{R}{0.5\textwidth}
\centering
\includegraphics[width=\linewidth]{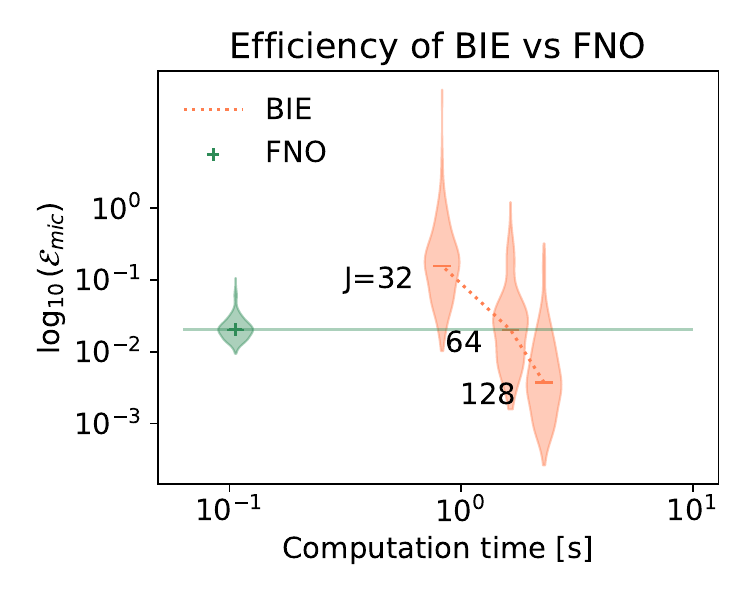}
\captionsetup{width=\linewidth}
\caption{Distribution of the logarithm of the error $\emic$ \eqref{eq: micro error} over $K=200$ training data, as a function of GPU time for the classical \gls{BIE} method with $J\in\{32,64,128\}$ (orange), and \gls{FNO} (green).}
\label{fig: fno vs gmres time}    
\end{wrapfigure}

\begin{table}
    \centering
     \begin{tabular}{l|c c c c c}
         \hline
         Solver &  $J=256$ & $J=128$ & $J=64$ & $J=32$ & FNO\\
         \hline
          Time &  $2.90\pm0.82$ & $1.38\pm0.43$ & $0.70\pm0.19$ &$0.42\pm0.15$ & $0.25\pm0.09$ \\
         \hline
     \end{tabular}
     \caption{Time measurements on the CPU for FNO at $J=256$ discretisation points and BIE at $J\in\{32,64,128,256\}$. }\label{tab: data from experiments}
\end{table}
    
\cref{fig: fno vs gmres time} compares the inference time for the learned \gls{FNO} solvers to a \gls{BIE} scheme that employs GMRES iterations. The GMRES implementation computes matrix-vector products with a discretized boundary integral operator stored as a dense matrix. \cref{tab: data from experiments} similarly lists run times and fraction of the total \gls{HMM} solve dedicated to the micro solvers for Case II (\cref{sec: case II}) on CPU. The learned approach is around 3 times faster than GMRES on systems ranging from 32 to 64 discretization points (20 times on GPU), and achieves the same accuracy as a 64 point discretization on samples not in the training data (but drawn from the same distribution). We note that the error distribution is significantly tighter for the learned method, visualized by the length of the violin plots of \cref{fig: fno vs gmres time}. 

\begin{figure}[!ht]
 \centering
 \scalebox{0.97}{
 \begin{tikzpicture}
 \node[anchor=north west] at (0,0)
 {\includegraphics[width=\textwidth]{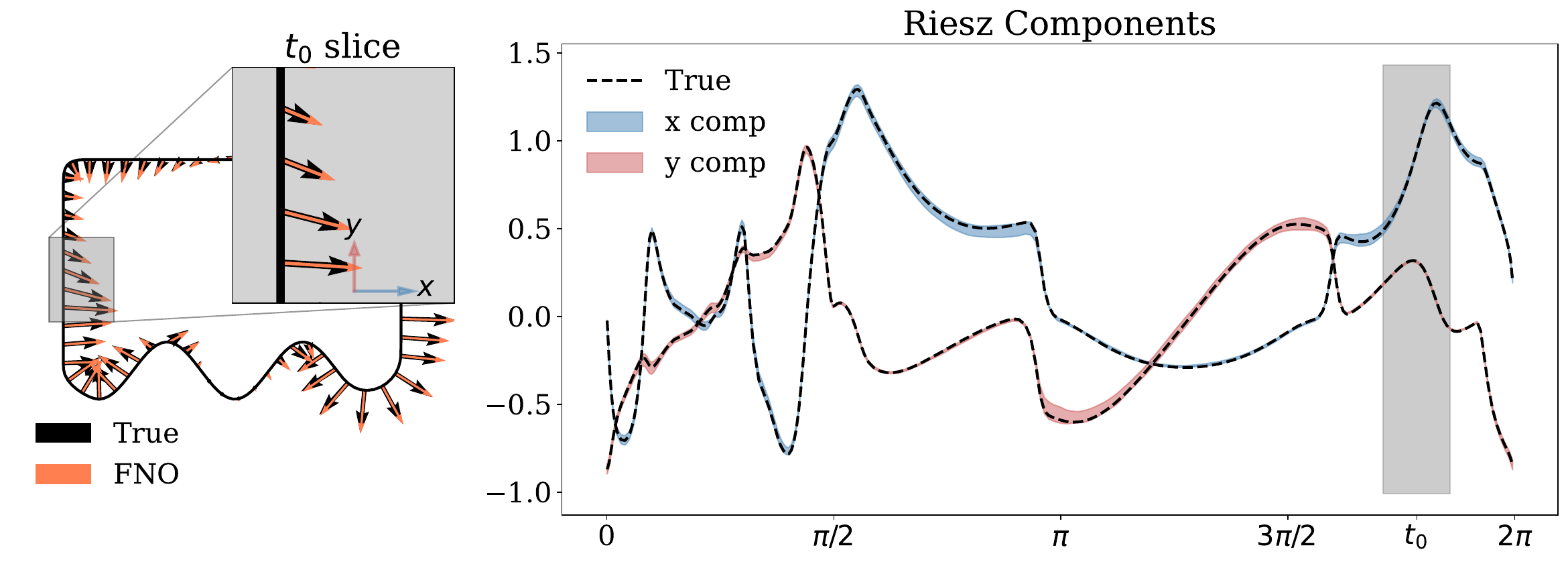}};
 \node[anchor=north west] at (0,-5)
 {\includegraphics[width=\textwidth]{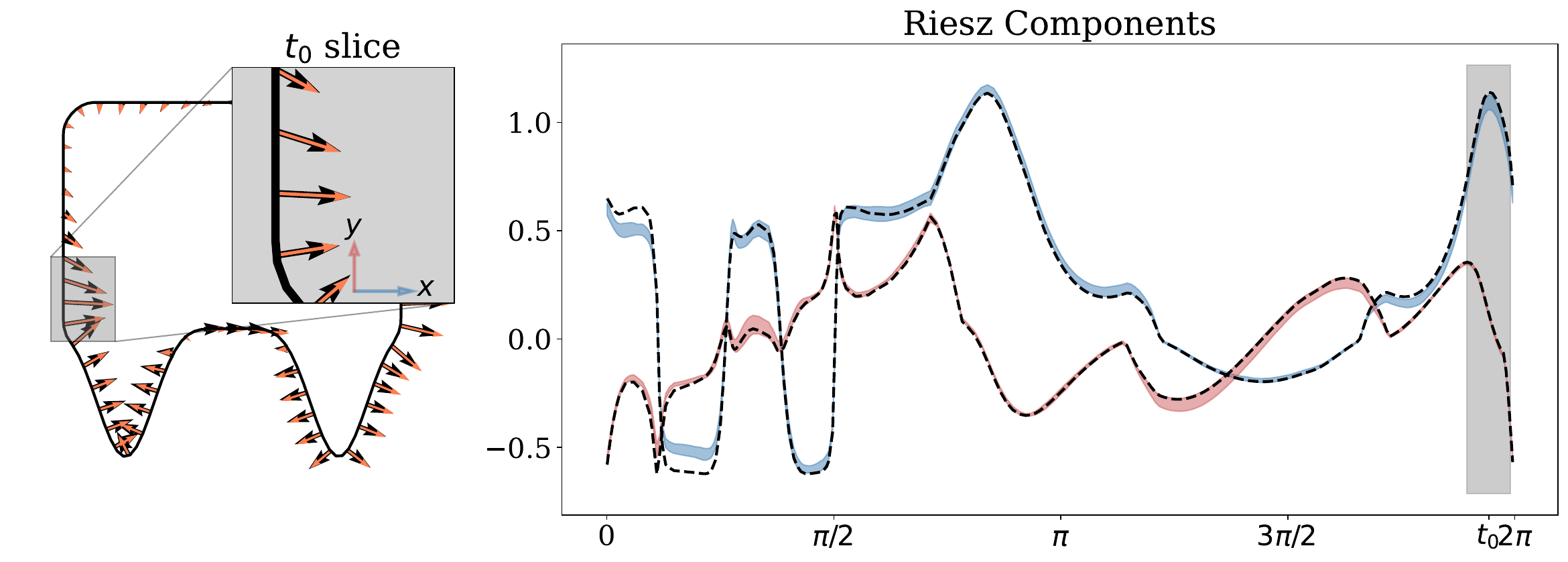}};
 \node[anchor=north west] at (0, -10)
 {\includegraphics[width=\textwidth]{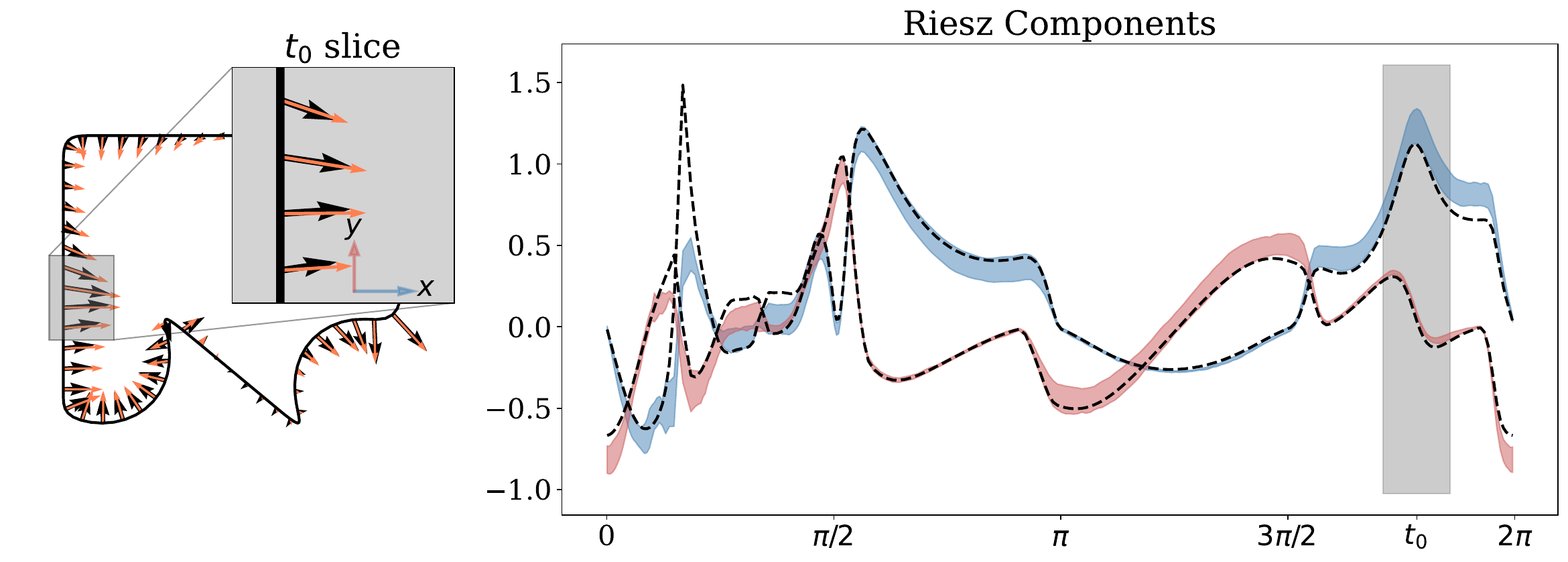}};

 \node[below, right,yshift=-1.5em] at (0,0) {\scalebox{1.5}{(a)}};
 \node[below, right,yshift=-1.5em] at (0,-5) {\scalebox{1.5}{(b)}};
 \node[below, right,yshift=-1.5em] at (0,-10) {\scalebox{1.5}{(c)}};
 \end{tikzpicture}}
 \caption{Case I, Prediction ranges from 6 independently trained networks on three test domains. Case $(a)$ has curvature and aspect ratio that resembles the training data, whereas $(b)$ and $(c)$ are \emph{outside} the training distribution. The colored region indicates the full range (max to min) of the predictions. The shaded grey areas indicates a segment of the parameter space that the left figure zooms in on.}\label{fig: micro solver generalizes}
\end{figure}
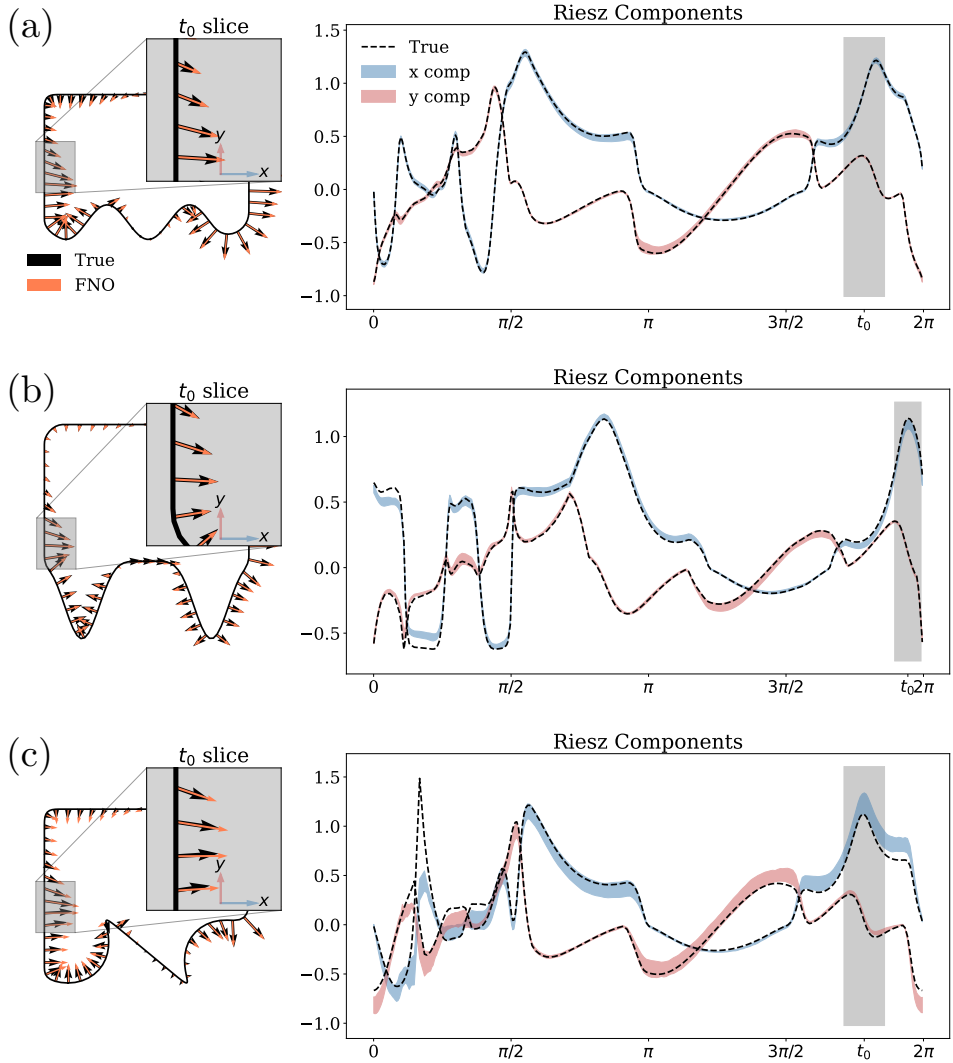
\subsubsection{Case II: Periodic Channel}\label{sec: case II}
We first demonstrate the relation between the error in the macro solution and the micro solution by studying a simple domain, consisting of a periodic channel with a roughness on the lower boundary. We define
\[
 \dom_\epsilon := \Bigl\{\point=(\pointi_1,\pointi_2)\mid\;0< \pointi_1< \perL,\;\epsilon\bigl(2-\sin(2\pi\tfrac{\pointi_1}{\epsilon})\bigr)< \pointi_2 < 1 \Bigr\},
\]
with a periodic boundary at $\pointi_1 = 0$ and $\pointi_1=\perL$. The boundary condition is posed on the top boundary $\pointi_2 = 1$ and takes the expression $\bdryvel(\pointi_1, 1) = (2 + \sin (2\pi \pointi_2))e_1$ with $e_1 = (1, 0)$ the unit vector in the $\pointi_1$-direction.

The macro solver is an \emph{iterative spectral method} that uses a tensor basis with \emph{Chebyshev polynomials} in the $\pointi_2$-direction and a \emph{Fourier basis} in the periodic $\pointi_1$-direction. We use a tensor grid with $21$ discretization points in each direction to represent the macro solution, which is enough to make the model error dominant in our case. The \emph{extrapolation operator} uses 5:th degree polynomials $p(t)$ on the parameterisation variable $t$. The polynomial is fit to simultaneously meet the consistency criteria~\eqref{item_cont}--\eqref{item: macro_agreement} in \cref{sec: reconstruction operator}, and minimize the norm $\|\tfrac{\mathrm{d}}{\mathrm{d}t} p\|^2$. The resulting linear system is small and can be solved efficiently using Gaussian elimination. The \emph{interpolation operator} uses a Fourier basis with $13$ micro problems placed on uniform collocation points (\cref{fig: interpolation error} indicates that 10--20 micro problems is optimal), which enables the use of FFT to obtain the slip function. The optimal number of micro problems was determined from a small parameter study with a high accuracy micro solver.

\cref{fig: interpolation error} shows the interpolation error as defined in \cref{thm: bounded homogenized slip error} for Case II, as a function of the number of micro problems $N$ along the wall. The figure indicates the same asymptotics $\mathcal{O}(N^{-1} + N)$ as the proof of \cref{thm: bounded homogenized slip error} predicts. The asymptotics did not appear as clearly for all problem setups.

\Cref{fig: coupling error}a--f shows the \cpl error~\eqref{eq: hi} normalized by the smoothing-before error $\ehi$ as a function of the micro error~\eqref{eq: micro error} for six configurations of \emph{aspect ratio}, \emph{width}~\eqref{eq: domain generation} and \emph{offset} of the evaluation line. The error is approximately independent of $\epsilon$ due to scale invariance in the micro solver, and depends approximately linearly on the micro error. Furthermore, the \cpl error of \gls{FNO-HMM} is comparable to a 64-point discretization of \gls{BIE-HMM} in configurations (a) and (b), slightly worse in (c), (d) and (f), and matches a 32-point discretization in (e). The cases (a),(c) and (d) all have a smaller margin between the \cpl $\ecpl$ and model error $\emdl$, meaning that the $\cpl$ contributes more to the total error $\etot$ (since it is bounded above by $\emdl+\ecpl$). These setups are likely harder since the evaluation line is close to the boundary, causing a near-singularity in the boundary integral formulation. We therefore avoid placing the evaluation line too close to the boundary in Case III and IV.

\begin{wrapfigure}{R}{0.5\textwidth}
\centering
\includegraphics[width=\linewidth]{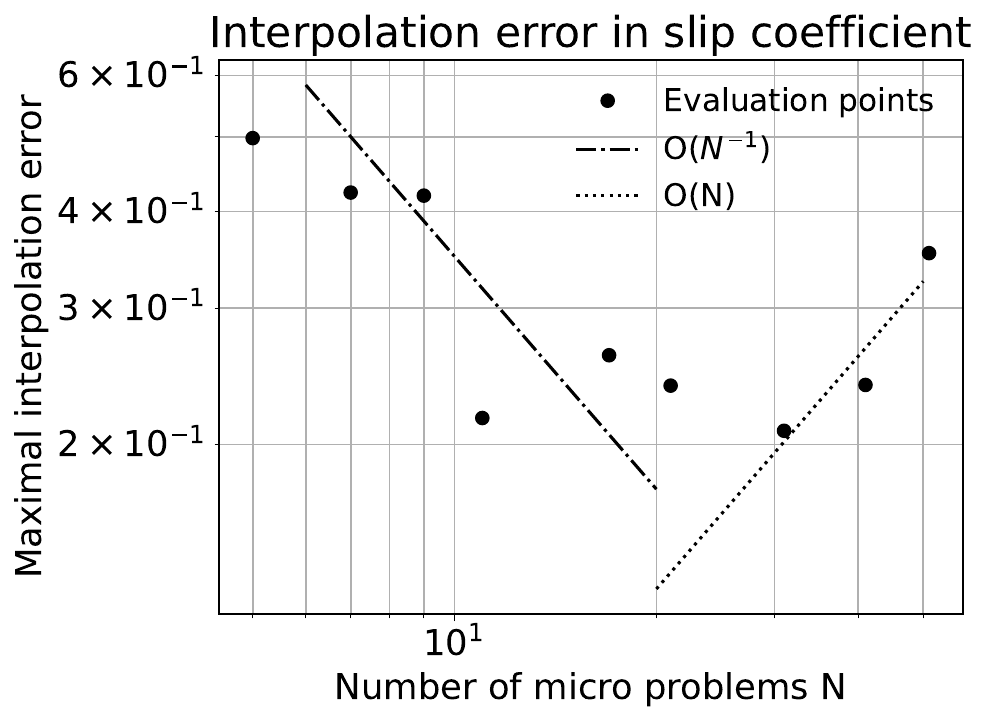}
\captionsetup{width=\linewidth}
\caption{Interpolation error in the slip coefficient on a periodic channel domain with random roughness, compared to the rates predicted in the proof of \cref{thm: final interpolation estimate}.}
\label{fig: interpolation error}
\end{wrapfigure}

Configuration (e) is likely harder for \gls{FNO-HMM} due to a combination of high aspect ratio and small width, which is underrepresented in the training data. In all cases, the \cpl error of the \gls{FNO-HMM} is between one and two orders of magnitude lower than the model error~\eqref{eq: hi}, meaning that it makes up for around 10\% of the total error~\eqref{eq: hi}. \Gls{FNO-HMM} and \gls{BIE-HMM} both stay within the $\epsilon$-scale fluctuations of the true solution in \cref{fig: pipe}b.

\begin{figure}[!htp]
 \centering
 \begin{tikzpicture}
    \node[anchor=north west] at (0,0) {\includegraphics[width=0.95\textwidth]{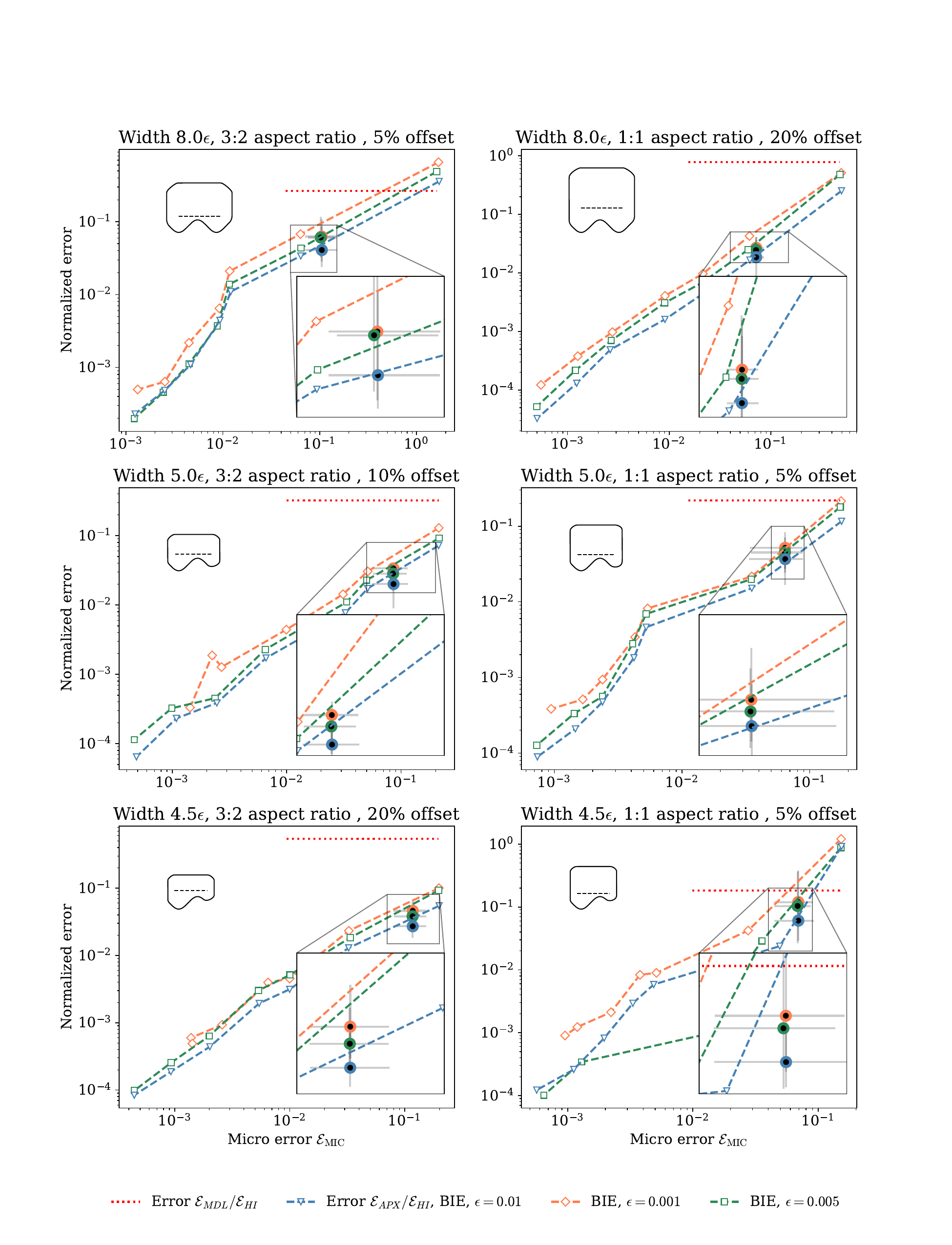}
};
 \foreach\x\y\lab in {0/0/a,1/0/b,0/1/c,1/1/d,0/2/e,1/2/f}{
    \node at (2+5.2*\x, -2.4-4.4*\y) {(\lab)};
 }
 \end{tikzpicture}
 \caption{Normalized \cpl error $\ecpl/\ehi$~\eqref{eq: hi} of \gls{FNO-HMM} in case II (black points) with 100\% confidence intervals over 5 training runs are marked with the grey crosses -- evaluated on the line $L_{2\epsilon}=\{(x,y), y=2\epsilon,x\in(-1,1)\}$ as a function of micro error $\emic$~\eqref{eq: micro error}. For reference, we show line plots of normalized \cpl error $\nrmse{\tilde u_J}{u_{\mathrm{HMM}}}{L_{2\epsilon}}/\ehi$ where $\tilde u_J$ is the \gls{BIE-HMM} solver with varying resolution $J$. Different colors for $\epsilon\in \{0.01,0.04,0.1\}$ and different width, aspect ratio and evaluation offset in the micro domains. The learned solver is highlighted in the zoomed boxes, and the total normalized error $\etot(L_{2\epsilon})/\ehi$ is marked as a red dashed line.}\label{fig: coupling error}
\end{figure}

As \cref{fig: pipe} shows, \gls{FNO-HMM} obtains an order of magnitude improvement over the smoothing before-error~\eqref{eq: hi}, and performs comparably to a fully resolved \gls{BIE-HMM}. Close to the boundary, the total error for \gls{FNO-HMM} is even comparable to the smoothing-after error~\eqref{eq: hi}. The smoothing-after error then decays quickly away from the boundary, whereas the errors from both \gls{HMM}-based methods decay more slowly.

\begin{figure}[!ht]
 \captionsetup[subfigure]{justification=centering}
 \begin{subfigure}{0.2\textwidth}
    \includegraphics[width=\textwidth]{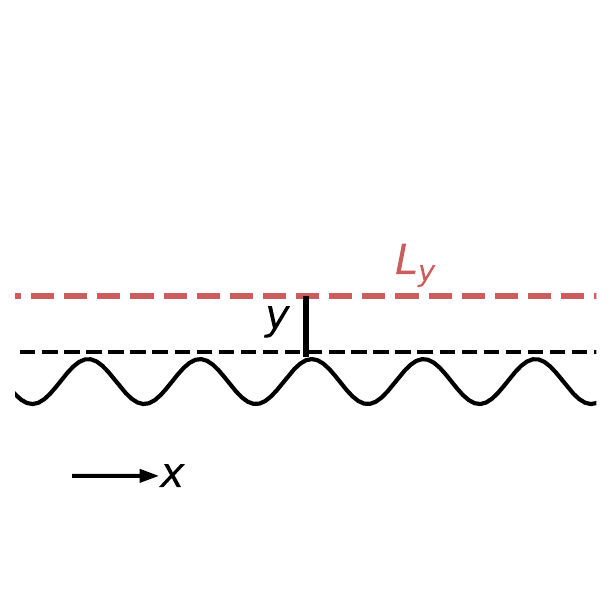}
 \end{subfigure}%
 \begin{subfigure}{0.8\textwidth}
    \includegraphics[width=\textwidth]{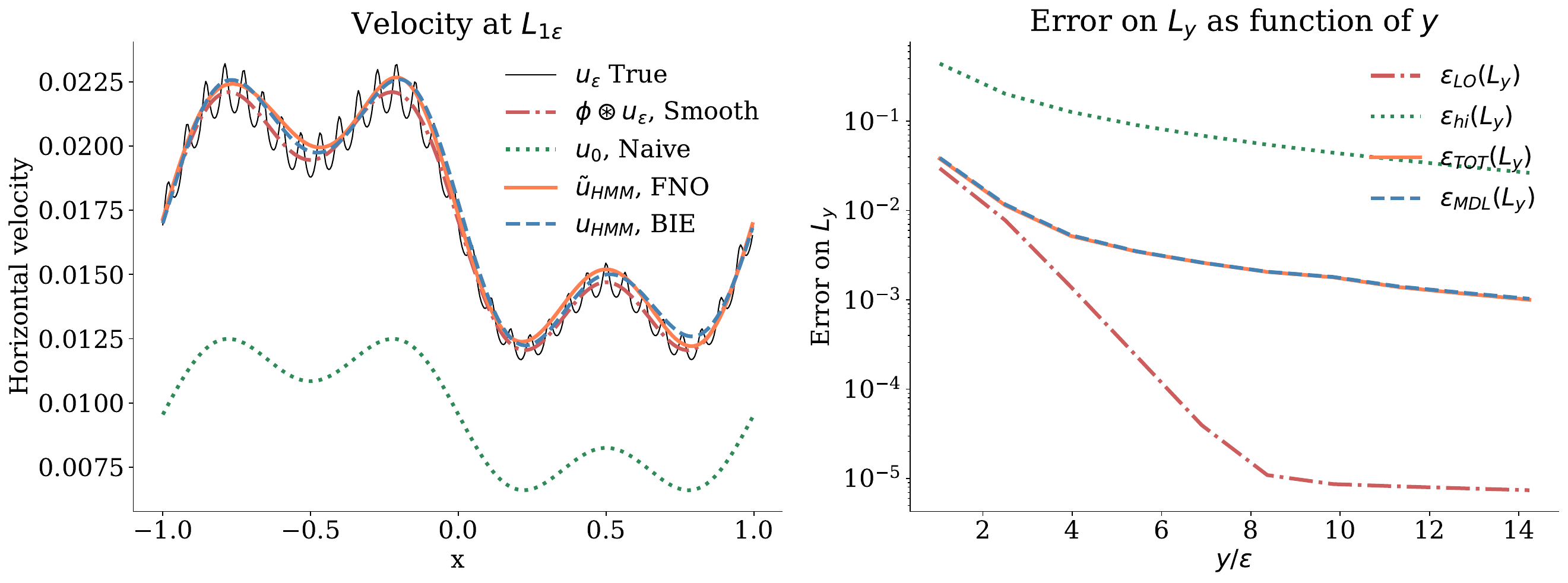}
 \end{subfigure}\\
 \captionsetup[subfigure]{justification=centering}
 \begin{subfigure}{0.2\textwidth}
    \includegraphics[width=\textwidth]{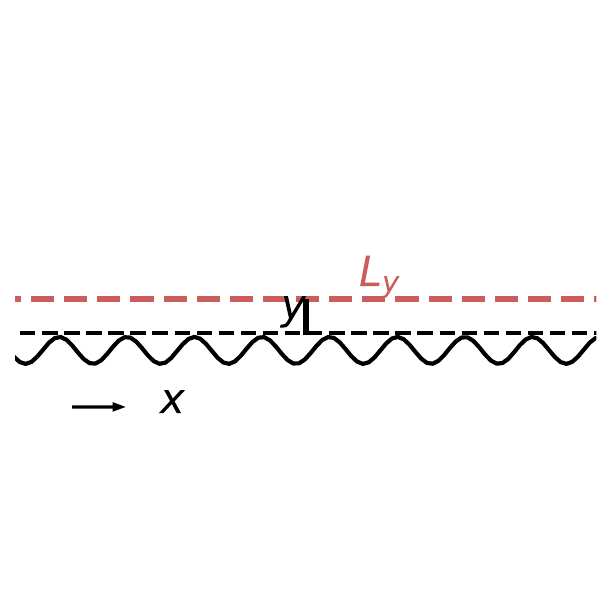}
 \end{subfigure}%
 \begin{subfigure}{0.8\textwidth}
    \includegraphics[width=\textwidth]{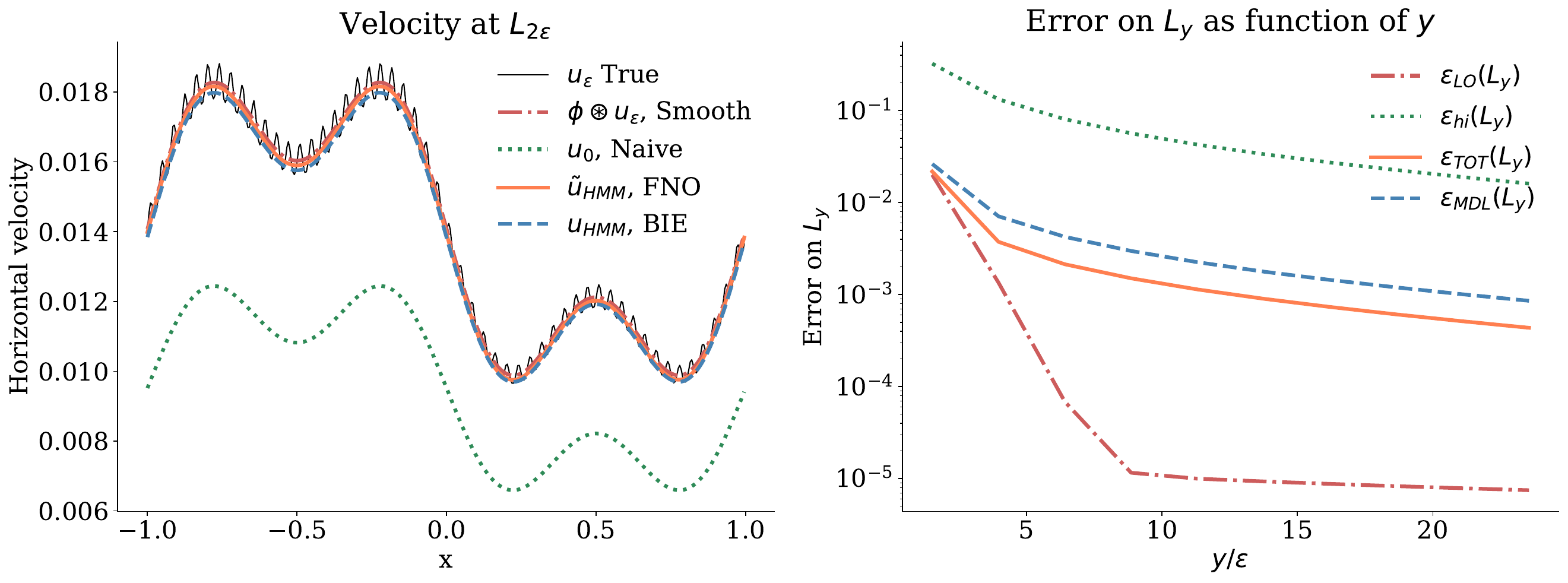}
 \end{subfigure}
 \caption{Case II, Periodic domain with $\epsilon=0.008$ (bottom) and $0.013$ (top). Evaluation line $L_y$ (left), horizontal velocity over the line $L_{2\epsilon}$ (middle) and errors $\etot$, $\emdl$, $\elo$ and $\ehi$ over $L_y$ as function of $y$ (right).}\label{fig: pipe}
\end{figure}
\begin{figure}[!ht]
 \captionsetup[subfigure]{justification=centering}
 \begin{subfigure}{0.2\textwidth}
     \includegraphics[width=\textwidth]{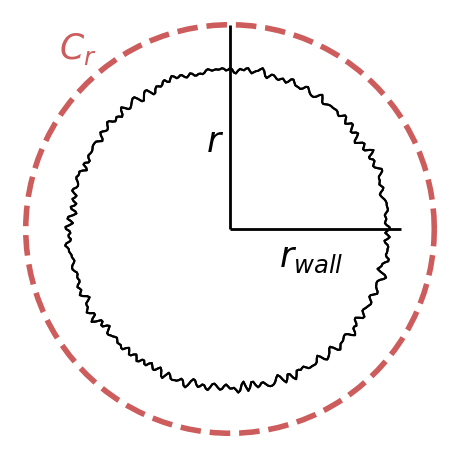}
 \end{subfigure}%
 \begin{subfigure}{0.8\textwidth}
     \includegraphics[width=\textwidth]{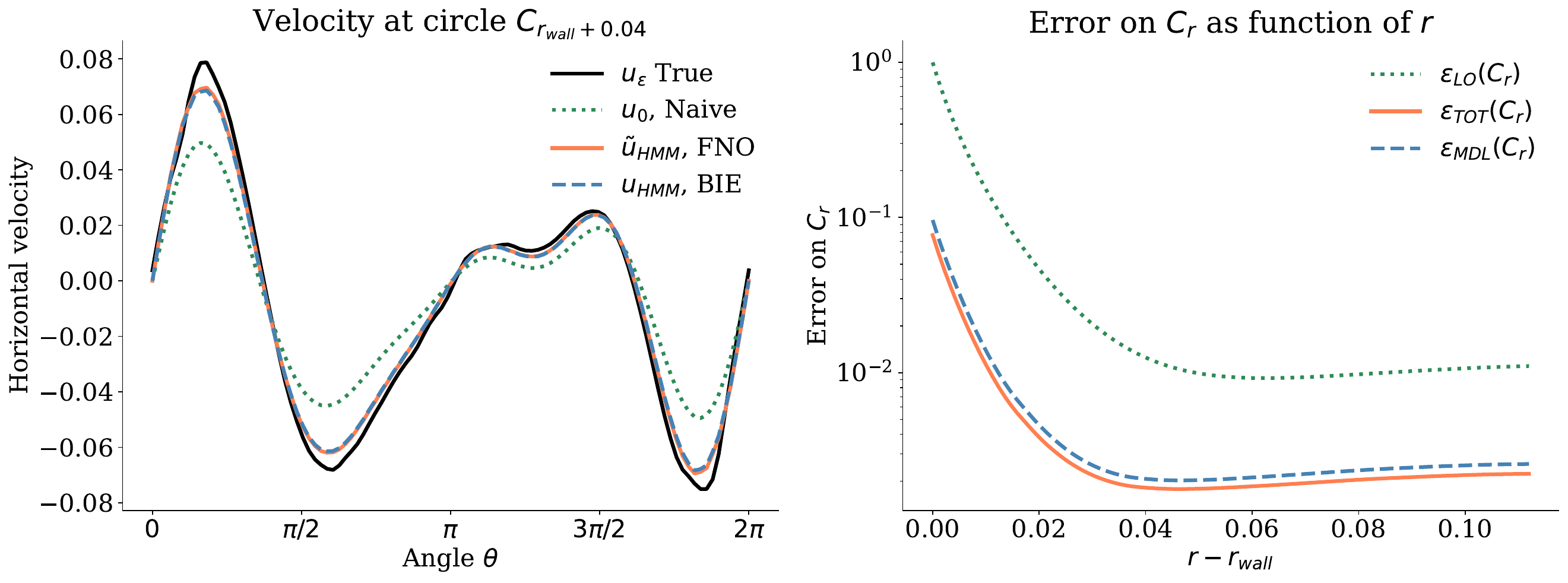}
 \end{subfigure}
 \caption{Case III, Error plots for the circular inset domain with $\epsilon=0.01$. Evaluation circle $C_r$ (left), horizontal velocity over $C_{r_{wall}+0.04}$ (middle) and errors $\etot$, $\emdl$, $\ehi$ over $C_r$ as function of $r-r_{wall}$ (right).}\label{fig: circular inset slice}
\end{figure}
    
\subsubsection{Case III: Circular Insets}
\begin{figure}[!ht]
 \centering
 \begin{subfigure}{\textwidth}
 \includegraphics[width=0.95\textwidth]{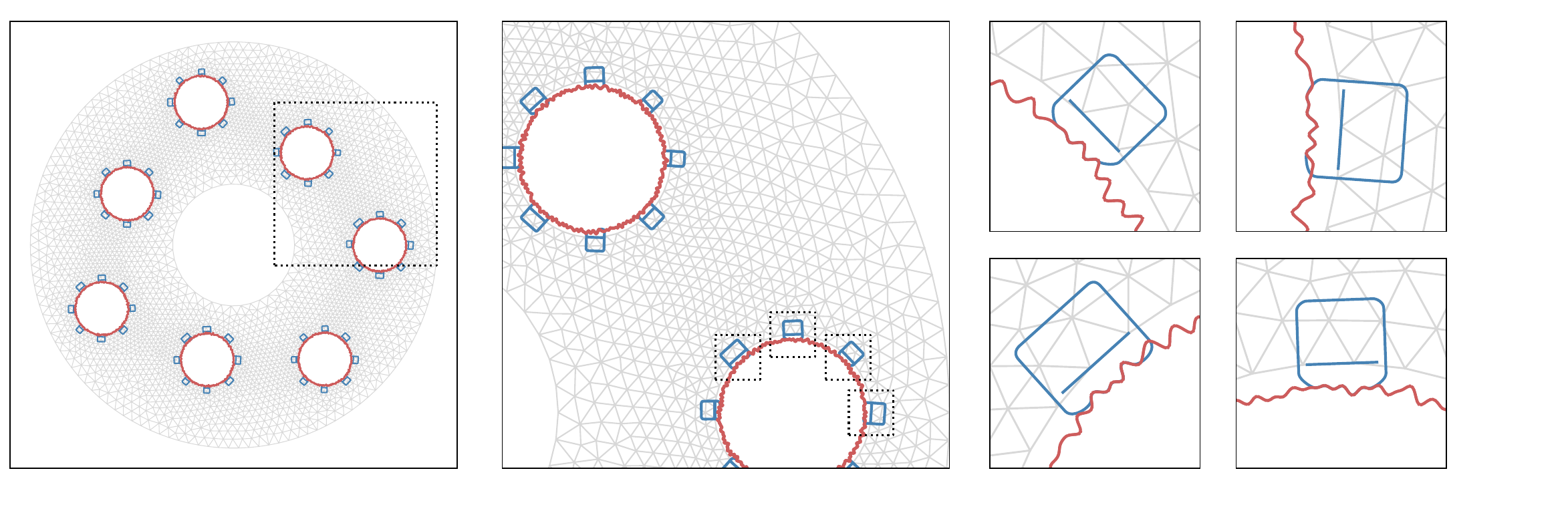}
 \caption{Case III, Circular inset with rough wall.}\label{fig: circular inset}
 \end{subfigure}\\
 \begin{subfigure}{\textwidth}
 \includegraphics[width=0.95\textwidth]{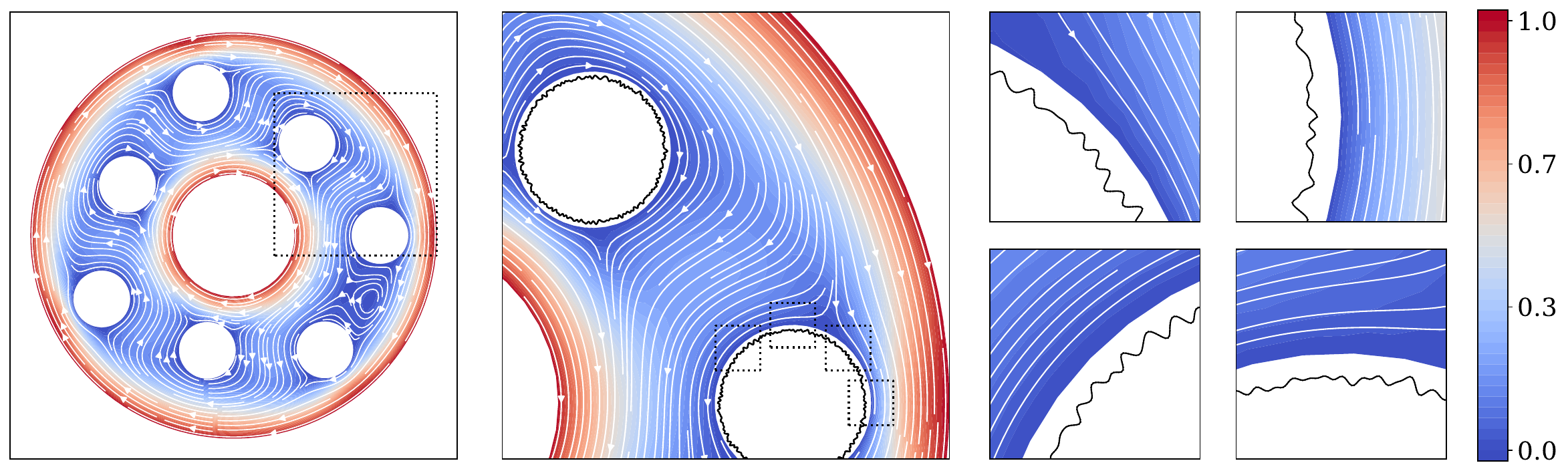}
 \caption{Case III, Streamlines and Velocity magnitude of \gls{FNO-HMM} solution.}\label{fig: circular inset deep}
 \end{subfigure}\\
 \begin{subfigure}{\textwidth}
 \includegraphics[width=0.95\textwidth]{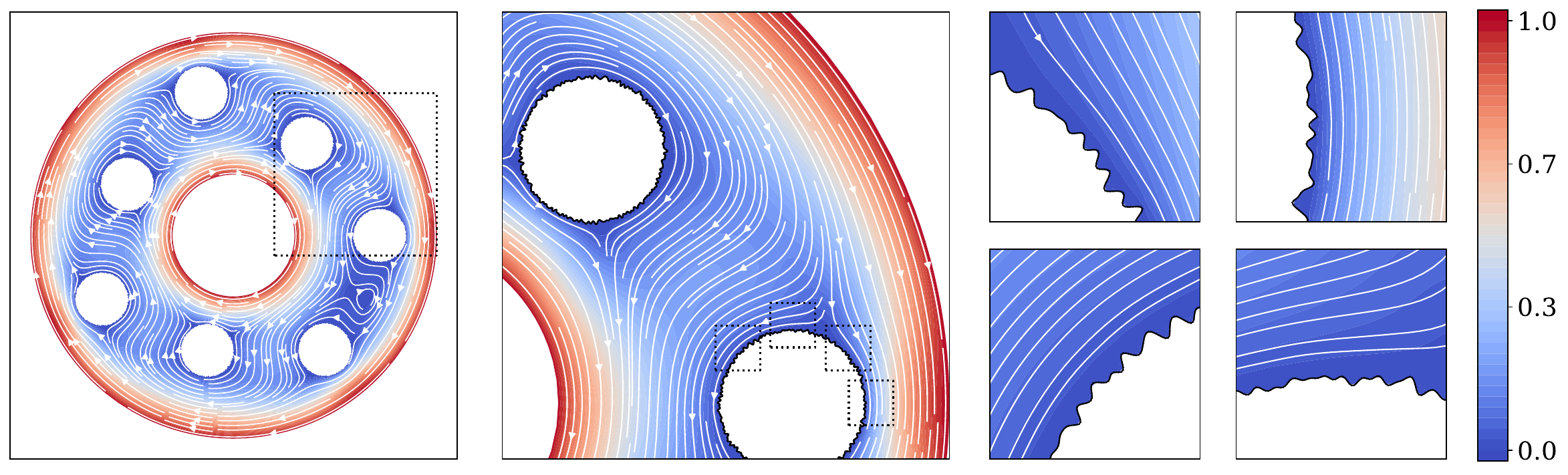}
 \caption{Case III, Streamlines and velocity magnitude of full solution.}\label{fig: circular inset full}
 \end{subfigure}\\
 \begin{subfigure}{\textwidth}
 \includegraphics[width=0.95\textwidth]{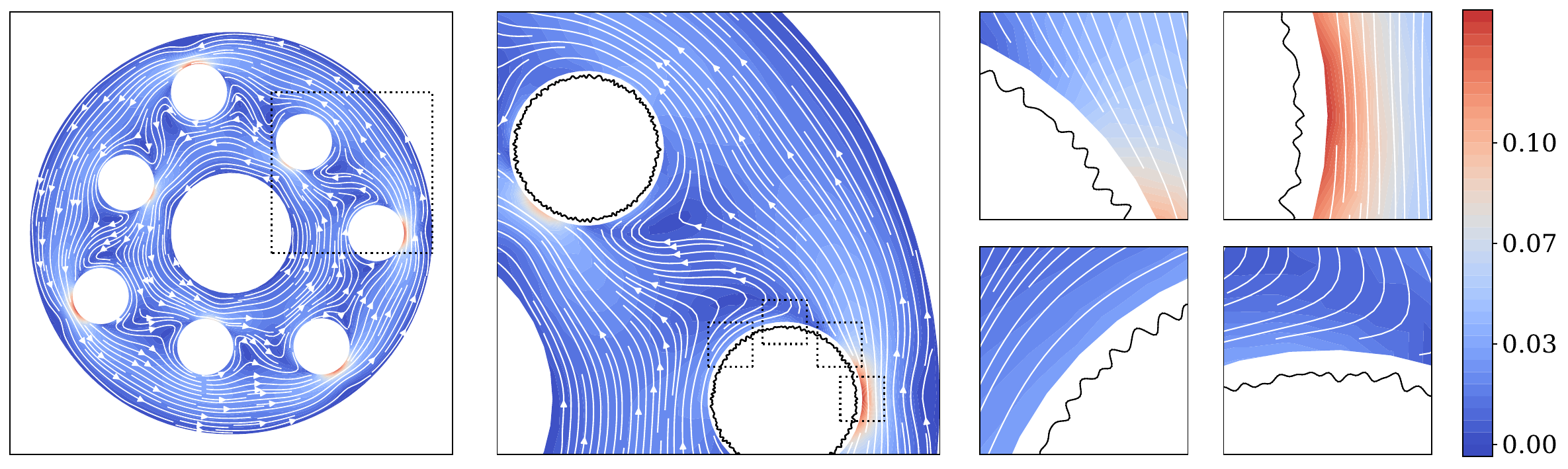}
 \caption{Case III, Streamlines and Velocity magnitude of \gls{FNO-HMM} error.}\label{fig: circular inset error}
 \end{subfigure}
 \caption{Case III, Plots of \gls{FNO-HMM} solution, true solution and total error.}
\end{figure}

Consider an annulus with radii $R>r$ with 7 rough circular insets $\cdom_{\epsilon,i}$:
\[
 \dom_\epsilon :=  \bigl\{\point\in\Rn{2}\mid r<\|\point\|<R \bigr\}\setminus(\cup_{i=1}^7\cdom_{\epsilon,i}).
\]
The circular insets are perfectly circular domains of radius $r_i$ centered at points $\point_i$, that have been perturbed by a Gaussian random field $\gpfunc\colon\Rn{2}\to\Reals$ with variance $1$ and exponential kernel with decay independent of $\epsilon$:
\[
 \cdom_{\epsilon, i} := \Bigl\{\point\in\Rn{2}\mid \|\point-\point_i\| <  r_i - \epsilon\bigl(\gpfunc(\point/\epsilon)-\gpfunc_{\mathrm{min}}\bigr)\Bigr\}.  
\]  
The Gaussian process has been shifted by $\gpfunc_{\mathrm{min}}:=\min_{\|\ypoint\|<R/\epsilon}(\gpfunc(\ypoint))$ to ensure that the smooth domain $\dom_0$ defined in \cref{sec: main problem} is precisely the domain with \emph{unperturbed} circular insets of radius $r_i$. The boundary condition is posed on the smooth boundaries $\|\point\|=R$ and $\|\point\|=r$ of the annulus and takes the form $\bdryvel(\point) := \bigl(\pointi_2/\|\point\|, -\pointi_1/\|\point\|\bigr)$, that is the clockwise tangental vector of size $1$. 

We solve the macro problem on the smooth domain $\dom_0$ with FEM using the \texttt{FEniCSx} package (a collection of the packages \texttt{DOLFINx}~\cite{baratta2023dolfinx}, Basix~\cite{BasixJoss2022basix} and UFL~\cite{AlnaesEtal2014ufl}). The computational mesh is generated by \texttt{gmsh} and contains 6998 elements. We use Taylor-Hood elements, and employ \emph{Nitsche's penalty method}~\cite{benzaken2022nitsche} for the Robin boundary condition. The full resolution ground truth is computed with FEM on a much finer mesh ($134\;648$ elements) of the multi scale domain $\dom_\epsilon$. The extrapolation operator is the same as for the periodic domain. Lastly, we use a kernel based interpolator (Gaussian process regressor) to interpolate the slip amount on the boundary. The interpolator is formulated in 2D assuming that the slip amount is a \emph{two-dimensional} random Gaussian field defined on the entire domain $\dom_0$. This means that the slip amount at one circular inset is allowed to influence the values at other insets.
    
\cref{fig: circular inset slice} compares function values and errors close to the rough wall. Similarly to the periodic pipe in Case II, \gls{FNO-HMM} outperforms the naive solution, and performs on par with a high accuracy \gls{BIE-HMM}. The numerical error in the training data is at most $10^{-5}$ and \cref{fig: micro solver generalizes} clearly shows that the learned model uncertainty is considerably larger than that. The flow plots in \cref{fig: circular inset}--d show the flow field for the full domain. The learned \gls{HMM} solver captures the flow field with a 10\% error with the exception of parts in the domain where the flow is confined to a narrow region, such as when the circular insets are close to the outer wall. The direction of the error vectors correlate strongly in space, as evident by the smooth flow structure of the error in~\cref{fig: circular inset error}. 

\subsubsection{Case IV: Hilbert Tube}

Lastly, we consider a costly PDE that is computationally infeasible to resolve fully on our system. We consider the fifth iteration of the Hilbert space-filling curve, spanning the domain $[0,1/2]\times[-1/4, 1/4]$. We mirror the path to create a simply connected curve. We fit an 800 mode Fourier series (symmetric around the origin) $\gamma\in\contspace^2([0,2\pi],\Rn{2})$ with corresponding normal vector $\normal\in\contspace^2([0,2\pi],\Rn{2})$ to the Hilbert curve. The macro domain $\dom_0$ is then defined by a tubular region surrounding $\gamma$ of thickness $H=0.01$, defined by its inner wall $\dombdry_{0,-}$ and outer wall $\dombdry_{0,+}$:
\[
    \dombdry_{0,\pm} := \left\{\gamma(t) \pm H\normal(t)\colon t\in [0,2\pi]\right\}.
\]
Finally, given a resulting parameterisation $\gamma_-$ and outward facing normal vector $\normal_-$ of the inner wall $\micdombdry_{0,-}$ we define $\roughness_\epsilon$ as 
\[
    \roughness_\epsilon := \left\{\gamma_-(t) + \epsilon \varphi(\gamma_-(t)/\epsilon)\normal_-(t)\colon t\in [0,2\pi]\right\},
\]
where the roughness field $\varphi\colon\Rn{2}\to\Rn{2}$ is defined by $\varphi(x_1,x_2)=\sin(x_1)\sin(x_2)$. We use a FEM macro solver with 16800 cells and Nitche's method for the Dirichlet boundary condition with a parameter $\lambda =10^4$. We use 1024 uniformly spaced micro problems, and the same interpolation as in Case III. \Cref{fig: tube problem} shows the domain, solution and difference between \gls{FNO-HMM} and \gls{BIE-HMM} with  $J=64$ discretization points. \Cref{tab: compute times} shows that the \gls{FNO-HMM} precomputation (i.e. approximation of the Riesz representors) is about 5 times faster than for \gls{BIE-HMM}, 

\begin{figure}[!ht]
 \centering
 \includegraphics[width=1.0\textwidth]{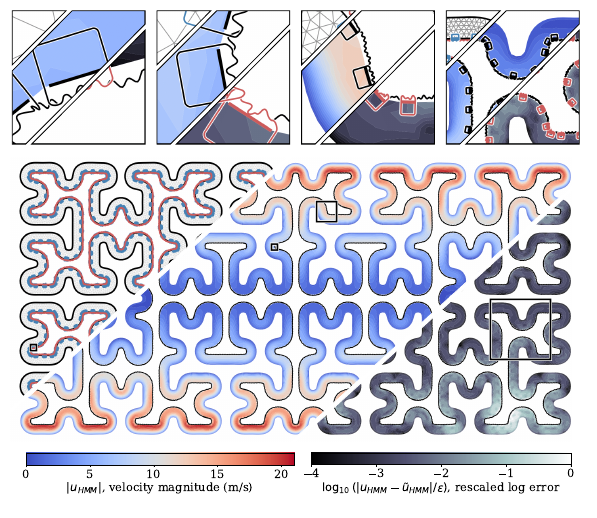}
 \caption{Case IV, Tube domain with rough wall. Solution predicted by \gls{FNO-HMM}. We show three slices of the domain: The left slice shows the mesh, the middle slice shows the predicted vector streamlines superimposed on a heatmap of the magnitude and the right plot shows the deviation from \gls{BIE-HMM}, rescaled by $\epsilon$.}\label{fig: tube problem}
\end{figure}

\begin{table}[!ht]
    \centering
    \begin{tabular}{c|c c c c}
         Routine & \begin{tabular}{c}\gls{FNO}-based \\ Precompute\end{tabular} & \begin{tabular}{c}\gls{BIE} (J=64)\\ Precompute \end{tabular} & Macro Solve & Overhead\\
         \hline
         Time [s] &  9.2 & 50.5 & 15.1 & 21.1\\
    \end{tabular}
    \caption{Compute times on Case IV, split into \gls{FNO-HMM} and \gls{BIE-HMM} precomputation with $J=64$ discretization points, macro solve and everything else (Overhead).}
    \label{tab: compute times}
\end{table}

\section{Discussion and Outlook}\label{sec: conclusion}
We propose a learning-based precomputation method for microscopic simulations in a \gls{HMM}, assuming a finite-dimensional, linear microscopic model. Applying this to laminar viscous flow over a rough wall, we develop an \gls{FNO}-based neural architecture that leverages the boundary integral formulation of Stokes flow. This micro solver integrates into the \gls{HMM}, and generalizes to arbitrary macroscopic domain and boundary values, provided the roughness distribution aligns with the training data. We show, under regularity assumptions, that a bound on training error leads to a bound on macroscopic error. Our approach achieves the same accuracy as numerical solvers in a third of the computation time on two test cases. Moreover, the error contribution from the learned component of \gls{HMM} is an order of magnitude lower than the total error.

Our work was motivated by two main observations. First, the model error inherent in \gls{HMM} allows for a coarse resolution of micro problems, but the microscopic solver remains the computational bottleneck. Secondly, deep learning surrogates are computationally faster than numerical solvers but have limited accuracy. Our numerical experiments largely support this analysis. Surprisingly, \gls{FNO-HMM} outperforms \gls{BIE-HMM} with discretization adjusted to match the \gls{FNO-HMM} in two categories: inference time and \emph{error variance} on the training data distribution, making it more predictable. We note that, although not done in our implementation, \gls{BIE-HMM} can trade variance in the error for variance in compute times via adaptive discretization.

A crucial design choice is learning the Riesz representors for the two linear functionals involved in computing slip amounts, rather than directly solving micro problems. Methods that do not exploit this structure will always be slower than numerical solvers that compute the representors, since it reduces each \gls{HMM} iteration to computing inner products. A major limitation of our approach is the reliance on linear, homogeneous microscale problems. In some non-linear flows with low Reynolds numbers such an approximation is accurate (for example,\cite{carney2021heterogeneous} study steady-state Navier Stokes). If the roughness scale is significantly larger than the width of the boundary layer however, another approach is needed. 

In realistic settings we expect the roughness distribution to be known, either by design, or by surface scans of the domain. The FNO then likely needs to be trained, or at the very least fine-tuned, on the new roughness distribution. This is a significant limitation, as the training data generation is computationally expensive. One promising direction is adaptive sampling, in which the training data is generated adaptively during training. Additionally, exchanging the \gls{FNO} for an interpretable model such as Gaussian-process regression could enable more accurate error estimates of the slip amount, and potentially allow for a scaling law on the amount of training data required at a given roughness scale.

\bibliographystyle{plain}
\bibliography{phd_bib}

\clearpage
\section{Appendix}
\subsection{Alignment Estimates}\label{sec: appendix: alignment estimates}
In the following section we derive bounds for the alignment between the micro solution and the tangent vector along the evaluation line.
Results are stated and proved in the special case of a micro domain with constant, horizontal velocity at the upper boundary and periodicity in the $x$-direction. This serves as a motivating example for \cref{cor: slip gen error}. The proof idea is similar to~\cite{carney2021heterogeneous}:
\begin{itemize}
  \item Show that the average flow satisfies a boundary value problem~\eqref{eq:1dBDPAverage} in one variable.
  \item Write its solution as a weighted average between the shearing flow at the top wall, and a counter-effect from the non-slip condition at the bottom wall \eqref{eq: banan}.
  \item Show that the friction from the boundary is overpowered by the shearing flow, causing a net flow that is bounded from below for sufficiently small $\epsilon$~\eqref{eq: estimate order averages}. 
\end{itemize} 
We begin by defining the (micro) domain, which is of the form
\[ \micdom := \bigl\{ (x, y)\colon x\in[0,\gamma_3], y\in [\varphi(x)-\gamma_1, \gamma_2] \bigr\},
\]
where $\varphi(x)\in(-\epsilon, 0)$ and $\gamma_1, \gamma_2 > 0$ denote the distance from the evaluation line $\lineseg$ to the upper and lower boundary, respectively, and $\gamma_3$ is the width of the domain. We assume
\begin{equation} 
\lim_{\epsilon\to 0}\gamma_1/\gamma_2 = 0,
    \quad \lim_{\epsilon\to 0}\epsilon/\gamma_2 = 0 \quad\text{and}\quad
    \lim_{\epsilon\to 0}\epsilon/\gamma_1=0,\label{eq: limits for gammas}
\end{equation}
so that asymptotically, $\epsilon \ll \gamma_1 \ll \gamma_2$. 
Next, let $\micu \colon \micdom \to \Reals$ be a function that is $\gamma_3$-periodic in the $x$-variable and that solves 
\begin{alignat}{2}
 \Delta\micu(\pmb x) - \nabla\micp(\pmb x) &= 0 & & \label{eq: micro special case}\\
 \nabla\cdot\micu(\pmb x) &= 0 &\quad &\text{for $\pmb x\in\micdom$,}\\
 \micu(x,\gamma_2) &= Ue_x& \quad &\text{for $y = \gamma_2$,} \\
 \micu\bigl(x,\varphi(x)-\gamma_1\bigr) &= 0& \quad &\text{for $x\in (0,\gamma_3)$.}\label{eq: micro special case 2}
\end{alignat}
Here, $e_x \in \Rn{2}$ is the unit vector in the $x$-direction and $U>0$ is constant. We now aim to estimate the following two averages of $\micu$:
\[
 \langle\micu_1\rangle(y) := \frac{1}{\gamma_3}\int_{0}^{\gamma_3} \micu(x, y)\cdot e_x\mathrm{d}x
 \quad\text{and}\quad 
 \langle \partial_y\micu_1\rangle(\ypoint) := \partial_y\langle\micu_1\rangle(y).
\]
Specifically, we want to evaluate the above at $y=0$.
\begin{proposition}\label{prop: aligned bounds}
 Assume $\micu$ solves~\eqref{eq: micro special case}--\eqref{eq: micro special case 2} and $\epsilon,\gamma_1,\gamma_2,\gamma_3$ satisfies \eqref{eq: limits for gammas}. Then, for $\epsilon>0$ sufficiently small, $\langle \micu_1\rangle (0)$ and $\langle\partial_y\micu_1\rangle(0)$ are $\eta_1$- and $\eta_2$- aligned, respectively:
 \[
     \frac{\bigl|\langle \micu_1\rangle(0)\bigr|}{U} \geq \eta_1
     \quad\text{and}\quad 
     \frac{\bigl|\langle \partial_y\micu_1\rangle(0)\bigr|}{U}\geq \eta_2
     \quad\text{with $\eta_1=\mathcal{O}(\gamma_1/\gamma_2)$ and $\eta_2=\mathcal{O}(1/\gamma_2)$.}
 \]
\end{proposition}
\begin{proof}
  Combining \eqref{eq: micro special case} with the periodicity of $\micu_1$ and $\pi$ in the $x$-direction yields that the function $y \mapsto \langle \micu_1\rangle(y)$ satisfies
 \begin{equation}\label{eq:1dBDPAverage}
    \begin{split}
     \frac{\mathrm{d}^2\langle \micu_1\rangle}{\mathrm{d} y^2}(y) &= \frac{1}{\gamma_3}\int_0^{\gamma_3} \partial_y^2\micu_1(x,y)\mathrm{d}x = \frac{1}{\gamma_3}\int_0^{\gamma_3} -\partial_x^2\micu_1(x,y)+\partial_x\pi(x,y)\mathrm{d}x = 0
     \\[0.75em]
     \langle\micu_1\rangle(\gamma_2) &=U. 
     \end{split}
 \end{equation}
 Next, define $C(\epsilon):=\langle \micu_1\rangle(-\gamma_1)$ for some, yet unknown $C(\epsilon)$.
 Then the functions $y \mapsto \langle \micu_1\rangle(y)$ and $y \mapsto \langle\partial_y\micu\rangle(y)$ solve a 1-dimensional boundary value problem that can be solved by linear interpolation:
 \begin{equation}\label{eq: banan}
     \langle \micu_1\rangle(y) = U\frac{y+\gamma_1}{\gamma_2+\gamma_1} + C(\epsilon)\frac{\gamma_2-y}{\gamma_2+\gamma_1}
     \quad \text{and}\quad 
     \langle\partial_y\micu\rangle(y) = \bigl(U-C(\epsilon)\bigr)\frac{1}{\gamma_1+\gamma_2}.
 \end{equation}
 Next, we estimate $C(\epsilon)$. We make the ansatz 
 \[ \micu(x,y) = U\frac{y + \gamma_1}{\gamma_2+\gamma_1}e_x + U \omega(x,y), \] where $\omega$ is a correction that satisfies Stokes equations with the boundary conditions
 \[
     \omega(x,\gamma_2) = 0
     \quad\text{and}\quad
     \omega(x, \varphi(x)-\gamma_1) = \varphi(x)\gamma_2^{-1}e_x = \mathcal{O}(\epsilon/\gamma_2)e_x.
 \]
 By continuity of the solution \cite[Thm~5 in Ch.\@ 3.5]{Ladyzhenskaia2014-ko}, $\bigl\|\micu(x,-\gamma_1)\bigr\|=\mathcal{O}(\epsilon/\gamma_2)$ and hence, $C(\epsilon) =\mathcal{O}(\epsilon/\gamma_2)$. A Taylor expansion of \eqref{eq: banan} then shows
 \begin{equation}
     \langle \micu_1\rangle(0)=  U\frac{\gamma_1}{\gamma_2}\bigl(1 + \mathcal{O}(\epsilon/\gamma_1+\gamma_1/\gamma_2)\bigr)
     \quad\text{and}\quad 
     \langle \partial_y \micu_1\rangle(0)= U\frac{1}{\gamma_2}\bigl(1+\mathcal{O}(\epsilon/\gamma_2+\gamma_1/\gamma_2)\bigr).
     \label{eq: estimate order averages}
 \end{equation}
 The result now follows by choosing $\epsilon$ sufficiently small and invoking~\eqref{eq: limits for gammas}.
\end{proof}

\begin{corollary}\label{cor: slip estimate} By the same construction as \cref{prop: aligned bounds}, we also obtain the following asymptotic expression for the estimated slip coefficient:
 \[
     \frac{\langle \micu_1\rangle (0)}{\langle\partial_y \micu_1\rangle(0)} = \gamma_1\bigl(1+\mathcal{O}(\epsilon/\gamma_1+\epsilon/\gamma_2)\bigr)
     \quad\text{as $\epsilon \to 0$.}
 \]
\end{corollary}
\begin{remark}
  Carney et al argue that $C(\epsilon)$ in \eqref{eq: banan} is $\mathcal{O}(U\epsilon)$\cite[appendix B]{carney2021heterogeneous}, but their proof relies on an upper bound of the line average $\langle\micu_1\rangle(0)$ by the $L^1$-norm in the whole domain, which to our knowledge is unmotivated. Our approach is based on using \cite[Thm~5 in Ch.\@ 3.5]{Ladyzhenskaia2014-ko} to control the maximal value of the solution.
\end{remark}

\subsection{Lipschitz Continuity of Slip Amount}\label{sec: lipschitz slip}
Here we give a heuristic motivation for why it is natural to assume that the slip coefficient is Lipschitz continuous. The computation is simplified significantly by solving an infinitely wide micro domain of the form
\[ 
\micdom := \bigl\{ (x,y)\colon x\in\Reals, y\in [\varphi(x)-\gamma_1, \gamma_2] \bigr\}.
\]
while we still only average over the interval $[0,\gamma_3]$. This avoids the need for more advanced theoretical machinery such as shape calculus. With the above construction and $\epsilon,\gamma_1,\gamma_2,\gamma_3$ as in \cref{sec: appendix: alignment estimates}, we write the slip amount $\alpha(z)$ on the following form:
\[
 \alpha(z) = \dfrac{\bigl\langle S_z[\micu_1]\bigr\rangle(0)}{\bigl\langle S_z[ \partial_y\micu_1]\bigr\rangle(0)}
 = \dfrac{\frac{1}{\gamma_3}\displaystyle{\int_0^{\gamma_3}}\micu_1(x+z,0)\,\mathrm{d}x}{\frac{1}{\gamma_3}\displaystyle{\int_0^{\gamma_3}}\partial_y\micu_1(x+z,0)\,\mathrm{d}x},
\]
where $S_z[u](x,y):=u(x+z,y)$ is the shift operator. 
We can now characterize the regularity of the slip amount:
\begin{proposition}
 For $\epsilon>0$ sufficiently small, the slip amount $z \mapsto \alpha(z)$ is  Lipschitz-continuous with a Lipschitz constant bounded by $2\gamma_1/\gamma_3$.
\end{proposition}
\begin{proof}
 Note first that any function $f$ defined by $f(x,z) = g(x+z)$ satisfies $\partial_x f = \partial_z f$. Differentiating $\alpha(z)$, we obtain
 \begin{multline*}
   \frac{\mathrm{d}\alpha(z)}{\mathrm{d}z} 
     =\frac{\mathrm{d}}{\mathrm{d}z} \biggl[\frac{\bigl\langle S_z[\micu_1]\bigr\rangle(0)}{\bigl\langle S_z[\partial_y\micu_1]\bigr\rangle(0)} \biggr] 
     =\alpha(z)\biggl(\frac{\bigl\langle \partial_x S_z[\micu_1]\bigr\rangle(0)}{\bigl\langle S_z[\micu_1]\bigr\rangle(0)}-\frac{\bigl\langle \partial_xS_z[\partial_y\micu_1]\bigr\rangle(0)}{\bigl\langle S_z[\partial_y\micu_1]\bigr\rangle(0)}\biggr)
     \\[0.75em]
     =\frac{\alpha(z)}{\gamma_3}\biggl(\frac{\micu_1(z+\gamma_3,0) - \micu_1(z, 0)}{\bigl\langle S_z[\micu_1]\bigr\rangle}-\frac{\partial_\ypoint\micu_1(z+\gamma_3,0)-\partial_y\micu_1(z,0)}{\bigl\langle S_z[\partial_\ypoint\micu_1]\bigr\rangle}\biggr).
 \end{multline*}
 Next, we evoke the estimates from \cref{cor: slip estimate} and \cref{prop: aligned bounds}:
 \begin{equation}
     \frac{\mathrm{d}\alpha(z)}{\mathrm{d}z} = \frac{\gamma_1}{\gamma_3}\bigl(1+\mathcal{O}(\epsilon/\gamma_1 + \epsilon/\gamma_2+\gamma_1/\gamma_2)\bigr)\leq 2\gamma_1/\gamma_3
     \quad\text{for $\epsilon>0$ small enough.}
 \end{equation}
 Hence,
 \[
     \bigl|\alpha(x)-\alpha(y)\bigr| = \left|\int_x^{y} \frac{\mathrm{d}\alpha(z)}{\mathrm{d}z}\mathrm{d}z\right|\leq 2\frac{\gamma_1}{\gamma_3}|x-y|,
 \]
 i.e., the Lipschitz constant is bounded by $2\gamma_1/\gamma_3$.
\end{proof}

\subsection{An Error Bound for the Slip Amount}\label{sec: appendix: slip estimates}
The final goal of this section is to prove~\cref{cor: slip gen error}. We begin with proving~\cref{lemma: expected ratio}.
\begin{proof}[Proof of \cref{lemma: expected ratio}]
 Define the random variables $\rv e_1 := \frac{\rv a_1-\rv b_1}{\rv a_1}$ and $\rv e_2 := \frac{\rv b_2-\rv a_2}{\rv b_2}$. 
 Next, we rewrite the error:
 \[
     \frac{\rv a_1}{\rv a_2}-\frac{\rv b_1}{\rv b_2}=\frac{\rv a_1 \rv b_2-\rv b_1 \rv a_2}{\rv a_2 \rv b_2} = \frac{\rv a_1 \rv b_2-\rv a_1 \rv b_2\left(1 - \rv e_1\right)\left(1-\rv e_2\right)}{\rv a_2 \rv b_2} = \frac{\rv a_1}{\rv a_2}(\rv e_1 \rv e_2-\rv e_1-\rv e_2).
 \]
 Taking the expectation of the absolute value and then using the triangle inequality and boundedness of $\rv a_1/\rv a_2$ now results in the following:
 \begin{multline*}
     \expectp{\Bigl|\frac{\rv a_1}{\rv a_2}-\frac{\rv b_1}{\rv b_2}\Bigr|} \leq \expectp{\Bigl|\frac{\rv a_1}{\rv a_2}\Bigr|\bigl(|\rv e_1|+|\rv e_2|+|\rv e_1 \rv e_2|\bigr)}\\\leq C\left(\sqrt{\expectp{\rv e_1^2}} + \sqrt{\expectp{\rv e_2^2}} + \sqrt{\expectp{\rv e_1^2}\expectp{\rv e_1^2}\,}\right).
 \end{multline*}
 The final estimate is obtained by using the given bounds $\expectp{\rv e_1^2}\leq \delta_1^2$ and $\expectp{\rv e_2^2}\leq \delta_2^2$.
\end{proof}

\begin{proof}[Proof of \cref{cor: slip gen error}]
Introduce the random variables 
\begin{xalignat*}{2}
  \rv a_1 &:=\inner{\rv \rone}{\micbdryvel}_{\rv\micdombdry}
  & \rv b_1 &:=\inner{\rv \trone}{\micbdryvel}_{\rv\micdombdry}
  \\
  \rv a_2 &:=\inner{\rv \rtwo}{\micbdryvel}_{\rv \micdombdry}
  & \rv b_2 &:=\inner{\rv \trtwo}{\micbdryvel}_{\rv\micdombdry}.
\end{xalignat*}
\cref{asm: slip operator taylor} allows us to apply \cref{lemma: expected ratio} on the above random variables.
This yields the desired result, so it remains to show that the expectations in \cref{lemma: expected ratio} are bounded. 
 \[
     \expectp{\left(\frac{\rv a_1-\rv b_1}{\rv a_1}\right)^2} = \expectp{\left(\frac{\inner{\rv\rone-\rv\trone}{\micbdryvel}_{\rv\micdombdry}}{\inner{\rv\rone}{\micbdryvel}_{\rv\micdombdry}}\right)^2} \leq \expectp{\frac{\|\rv\rone-\rv\trone\|^2_{\rv\micdombdry}\|\micbdryvel\|^2_{\rv\micdombdry}}{\inner{\rv\rone}{\micbdryvel}_{\rv\micdombdry}^2}},
 \]
 Where we applied Cauchy-Schwarz to the inner product in the denominator. The expectation of the numerator is bounded bounded by \cref{asm: well aligned}:
 \[
     \expectp{\frac{\|\rv\rone-\rv\trone\|^2_{\rv\micdombdry}\|\micbdryvel\|^2_{\rv\micdombdry}}{\inner{\rv\rone}{\micbdryvel}_{\rv\micdombdry}^2}}\leq \expectp{\frac{\|\rv\rone-\rv\trone\|^2_{\rv\micdombdry}}{\eta_1^2\|\rv\rone\|_{\rv\micdombdry}^2}} \leq \frac{\delta^2}{\eta_1^2}.
 \]
 The bound on $\expectp{(\rv a_2-\rv b_2)^2/\rv b_2^2}$ is obtained similarly.
    \end{proof}

\subsection{Diagonal Terms in the Layer Potential}~\label{sec: appendix: boundary limits}

We compute the limit $\lim_{\ypoint\in\micdombdry \to \ypoint_0}\stresslet_{klm}(\ypoint-\ypoint_0)\normal_l(\ypoint_0)$ where $\stresslet_{klm}$ is the stresslet in \eqref{eq:Stresslet}, 
$\ypoint_0 = \diffeo(t_0)$, $\normal(\ypoint_0)$ is the normal vector at the boundary point $\ypoint_0$ and $\diffeo$ is a smooth parameterisation of the boundary $\micdombdry$. Substitution yields:
\begin{equation}
   \lim_{t\to t_0}\stresslet_{klm}\bigl(\diffeo(t)-\diffeo(t_0)\bigr)\normal_l\bigl(\diffeo(t_0)\bigr)\label{eq: boundary limit}
\end{equation}
First, let $g(h):= \diffeo(t_0+h)-\diffeo(t_0)$ and note that 
$g'(h)=\diffeo'(t_0+h)$, $g(0)=0$, and $g'(0)=\diffeo(t_0)$.
Next, we have $w_l\normal_l\bigl(\diffeo(t_0+h)\bigr)=\mathrm{det}\bigl[w, g'(h)\bigr]/\bigl\|g'(h)\bigr\|$ for any vector $w\in\Rn{2}$. Let $a_h:=g_k(h)$, $b_h:=g_m(h)$, $c_h:=\mathrm{det}[g(h), g'(0)]/\|g'(0)\|$ and $d_h:=\|g(h)\|$. Then, the expression in~\eqref{eq: boundary limit} reduces to
\[
\stresslet_{klm}\bigl(\diffeo(t)-\diffeo(t_0)\bigr)\normal_l\bigl(\diffeo(t_0)\bigr) = \frac{a_h b_h c_h}{d_h^4},
\]
where $a_h, b_h, d_h$ are $\mathcal{O}(h)$ functions and $c_h$ is $\mathcal{O}(h^2)$. This allows us to simplify the computations considerably:

\begin{lemma}\label{lemma: boundary limit}
    Suppose $a_h, b_h, d_h$ are $\mathcal{O}(h)$ functions and $c_h$ is a $\mathcal{O}(h^2)$. Then,
    \[
 \lim_{h\to 0}\frac{a_h b_h c_h}{d_h^4} = \frac{a_0'b_0'c_0''}{2(d_0')^4},
    \]
    where $a_0' := \lim_{h\to 0}\frac{\mathrm{d}a_h}{\mathrm{d}h}$, $b_0' := \lim_{h\to 0}\frac{\mathrm{d}b_h}{\mathrm{d}h}$, $c_0'' := \lim_{h\to 0}\frac{\mathrm{d}^2 c_h}{\mathrm{d}h^2}$ and $d_0'' := \lim_{h\to 0}\frac{\mathrm{d}^2 d_h}{\mathrm{d}h^2}$.
\end{lemma}
\begin{proof}
By L'Hopital's rule:
\[
    \lim_{h\to 0}\frac{a_h b_h c_h}{d_h^4}=\lim_{h\to 0}\frac{\tfrac{\mathrm{d^4}}{\mathrm{d}h^4}(a_h b_h c_h)}{\tfrac{\mathrm{d^4}}{\mathrm{d}h^4}d_h^4}\Big|_{h=0}.
\]
The only surviving term in the numerator is $a'_h b'_h c''_h$ and in the denominator only $(d_h')^4$ survives. It remains to determine how many such terms there are. The former is the number of ways to distribute 4 elements in 3 bins of sizes 1,1 and 2. The answer is $4\cdot 3=12$. The latter is $4!=24$. Hence, the limit is given by:
\[
    \lim_{h\to 0}\frac{a_h b_h c_h}{d_h^4} = \frac{1}{2}\frac{a_0'b_0'c_0''}{(d_0')^4}.
\]
\end{proof}

\cref{lemma: boundary limit} combined with $\tfrac{\mathrm{d}}{\mathrm{d} h}\bigl\|g(h)\bigr\| = g'(h)\cdot g(h)/\|g(h)\|\overset{h\to 0}{\longrightarrow} \bigl\|g'(0)\bigr\|,$ results in
\[
\lim_{h\to 0}\frac{g_k(h)g_m(h)\mathrm{det}\bigl[g(h), g'(0)\bigr]/\|g'(0)\|}{\bigl\|g(t)\bigr\|^4} =\frac{g_k'(0)g_m'(0)}{\bigl\|g'(0)\bigr\|^2}\cdot\frac{\mathrm{det}[g''(0), g'(0)]}{2\bigl\|g'(0)\bigr\|^3}.
\]
With the notation $\tangent$ for the unit tangent vector along the boundary at the point $\ypoint_0$, and identifying the right hand factor as the curvature $\kappa$, we obtain the limit:
\[
\lim_{\ypoint\in\micdombdry \to \ypoint_0}\sum_{l=1}^2\stresslet_{klm}(\ypoint-\ypoint_0)\normal_l(\ypoint_0) = \frac{\kappa}{2}\tangent_k\tangent_m.
\]

\subsection{Formula for the Intermediate Representors}\label{sec: appendix: intermediate reps}
Here we prove~\cref{prop: intermediate representors}. The computation is considerably easier in complex coordinates. Let $a=z_{a}, b=z_{b}$, $\xi = z_{\point}, \tangent=iz_{\normal_\ell}$ be the complex representations of $a,b,\point$ and the tangent vector $\tangent$ on the line, and $e_\xi = z_{\normal(\point)}$ the tangent vector along the boundary. Evaluation of the potential on a complex-valued density $\omega$ at a point $z$ is then given by:
\[
   \dlp[\omega](z) = \frac{1}{i\pi}\int_\Gamma \omega(\xi)\ImOp\left[\frac{\mathrm{d}\xi}{\xi - z}\right] - \frac{1}{i\pi}\int_\Gamma\overline{\omega(\xi)}\frac{\ImOp[(\overline{\xi} - \overline{z})\,\mathrm{d}\xi]}{(\overline{\xi}-\overline{z})^2}.
\]
Now, suppose we wish to compute an integral of $\ReOp\{\dlp[\omega](z)\overline{t}\}$, the inner product of $\dlp[\omega](z)$ and $t$, across a line segment between two points $a$ and $b$. Switching order of integration:
\begin{align*}
    \int_a^b \ReOp&\bigl[\dlp[\omega](z)\overline{t}\bigr]\,|\mathrm{d}z| \\
    &= \int_a^b\ReOp\left[\overline{t}\left(\frac{1}{i\pi}\int_\Gamma \omega(\xi)\ImOp\left[\frac{\mathrm{d}\xi}{\xi - z}\right] - \frac{1}{i\pi}\int_\Gamma\overline{\omega(\xi)}\frac{\ImOp\bigl[(\overline{\xi} - \overline{z})\,\mathrm{d}\xi\bigr]}{(\overline{\xi}-\overline{z})^2}\right)\right]|\mathrm{d}z|
    \\    
    &= \int_\Gamma\ReOp\left[-\frac{1}{i \pi}\int_a^b \left[\ImOp\left[\frac{\mathrm{d}\xi}{\xi - z}\right]t|\mathrm{d}z| + \frac{\ImOp\bigl[(\overline{\xi} - \overline{z})e_\xi \bigr]}{(\overline{\xi}-\overline{z})^2}\,\overline{t}|\mathrm{d}z| \right]\overline{\omega(\xi)}\right]|\mathrm{d}\xi|.
\end{align*}
We identify the Riesz representor as the integral expression before the density $\overline{\omega(\xi)}$. Denote it by $v_1(\xi)$. Using $t|\mathrm{d}z| = \mathrm{d}z$, we then have
\begin{align*}
v_1(\xi)&=-\frac{1}{i\pi}\int_a^b\left(\ImOp\left[\frac{e_\xi}{\xi - z}\right]\,\mathrm{d}z+\frac{\ImOp\bigl[(\overline{\xi} - \overline{z})e_\xi\bigr]}{(\overline{\xi}-\overline{z})^2}\,\overline{\mathrm{d}z}\right).
\end{align*}
So far, we have not used that $\lineseg$ is a straight line. To find an explicit expression of the above, we need to do that. We simplify with a change of variables $\zeta=\xi-z$, $\mathrm{d}\zeta=-\mathrm{d}z$. Using $\ImOp[z] = (z-\overline{z})/2i$, the integrand can be split into terms:
\newcommand{\ddelta}{\mathrm{d}\zeta}
\[
    \frac{1}{i\pi}\left(\ImOp\left[\frac{e_\xi}{\zeta}\right]\ddelta+\frac{\ImOp[\overline{\zeta}e_\xi]}{\overline{\zeta}^2}\overline{\ddelta}\right) =\frac{1}{-2\pi}\left(\frac{e_\xi\ddelta}{\zeta} - \frac{\overline{e_\xi}\ddelta}{\overline{\zeta}} + \frac{e_\xi\overline{\ddelta}}{\overline{\zeta}} - \frac{\zeta\overline{e_\xi}\overline{\ddelta}}{\overline{\zeta}^2}\right) 
\]
By the chain rule, $\mathrm{d}(\zeta/\overline{\zeta}) = -\overline{\ddelta}(\zeta/\overline{\zeta}^2) + \ddelta/\overline{\zeta}$. Similarly, $\mathrm{d}\log(\zeta)=\ddelta/\zeta$. Furthermore, $\overline{t}\ddelta = t\overline{\ddelta}$ on the line. Therefore, the above left hand side can be expressed as 
\begin{multline*}
    \frac{1}{-2\pi}\left(\frac{e_\xi\ddelta}{\zeta} - 2\frac{\overline{e_\xi}\ddelta}{\overline{\zeta}}+ \frac{e_\xi\overline{\ddelta}}{\overline{\zeta}} +\overline{e_\xi}\mathrm{d}\left(\frac{\zeta}{\overline{\zeta}}\right)\right)
    \\
    = \frac{1}{-2\pi}\left(e_\xi\mathrm{d}\log\zeta + (e_\xi-2\overline{e_\xi}\tfrac{t}{\overline{t}})\overline{\mathrm{d}\log\zeta} +\overline{e_\xi}\mathrm{d}\left(\frac{\zeta}{\overline{\zeta}}\right)\right)
    \\
    = \frac{1}{-2\pi}\left(t\left(\tfrac{e_\xi}{t}\mathrm{d}\log\zeta + (\tfrac{e_\xi}{t}-2\tfrac{\overline{e_\xi}}{\overline{t}})\overline{\mathrm{d}\log\zeta}\right) +\overline{e_\xi}\mathrm{d}\left(\frac{\zeta}{\overline{\zeta}}\right)\right).
\end{multline*}
Inserting the above back into the integral and using $\ddelta/\zeta = \mathrm{d}(\log\zeta)$, together with $\Delta a=\xi-a$ and $\Delta b = \xi-b$, we obtain
\begin{align*}
    v_1(\xi) &=\frac{1}{-2\pi}\int_{\Delta a}^{\Delta b} \left(t\left(\tfrac{e_\xi}{t}\mathrm{d}\log\zeta + (\tfrac{e_\xi}{t}-2\tfrac{\overline{e_\xi}}{\overline{t}})\overline{\mathrm{d}\log\zeta}\right) +\overline{e_\xi}\mathrm{d}\left(\frac{\zeta}{\overline{\zeta}}\right)\right) \\
    &= -\frac{1}{\pi}\left[\tfrac{t}{2}\left(\tfrac{e_\xi}{t}\log\zeta + (\tfrac{e_\xi}{t}-2\tfrac{\overline{e_\xi}}{\overline{t}})\overline{\log\zeta}\right) +\tfrac{\overline{e_\xi}}{2}\left(\frac{\zeta}{\overline{\zeta}}\right)\right]_{\Delta a}^{\Delta b}\\
    &=-\frac{1}{\pi}\left(\tfrac{it}{2}\left(\tfrac{n_\xi}{t}\log\tfrac{\Delta b}{\Delta a} + (\tfrac{n_\xi}{t}+2\tfrac{\overline{n_\xi}}{\overline{t}})\overline{\log\tfrac{\Delta b}{\Delta a}}\right)+\overline{n_\xi}\frac{\ImOp\left[\Delta b\overline{\Delta a}\right]}{\overline{\Delta b}\overline{\Delta a}}\right),
\end{align*}
which is the desired result.

The second kind intermediate representor can be written in a similar fashion, as
\[
    v_2(\xi) = -\frac{1}{i\pi}\int_a^b\left(\ImOp\left[\frac{t  e_\xi}{(\xi - z)^2}\right]\mathrm{d}z - \frac{\ImOp[\overline{t}e_\xi]}{(\overline{\xi} - \overline{z})^2}\overline{\mathrm{d}z} + 2\overline{t}\frac{\ImOp\bigl[\overline{t}e_\xi(\overline{\xi}-\overline{z})\bigr]}{(\overline{\xi}-\overline{z})^3}\,\overline{\mathrm{d}z}\right).    
\]
An analogous technique yields the desired expression.

\subsection{An Upper Bound for The Macro Error}\label{sec: appendix: pde bounds}
The aim is to prove~\cref{prop: pde error bound}. The approach in the proof follows a classical energy estimate using elementary properties of Hilbert space norms. The idea is as follows. Let $w = u-v$ and let $q$ be the corresponding difference in pressure. Then, we will show that the $H^1$-seminorm of $w$ satisfies an inequality involving also the norm of $w$ on the boundary $\dombdry$. A simple general inequality result will then finish the proof. We therefore begin by stating this helpful inequality.
\begin{claim}\label{lemma: completion of squares bound}
    Let $a,b,c,d >0$ such that
    $a^2 + b^2 \leq 2ca + 2db$.
    Then, 
    $a \leq \sqrt{c^2 + d^2} + c$.
\end{claim}
\begin{proof}
    By completion of squares, $(a-c)^2 + (b-d)^2 \leq c^2 + d^2$. Inserted into the triangle inequality $a \leq |a-c| + c$, we obtain $a \leq \sqrt{c^2 + d^2} + c$.
\end{proof}

We are now ready for the main proof. 
\begin{proof}[Proof of~\cref{prop: pde error bound}]
We begin by estimating the norm of $w$.
\[
    |w|_{H^1(\dom_0,\Rn{2})}^2 + \|(w\cdot \tangent)\alpha^{-1/2}\|_{L^2(\roughness_0,\Reals)} = \int_{\dom_0}\nabla \cdot (w : \nabla w) - w \cdot \Delta w  + \int_{\roughness_0} \frac{(w\cdot \tangent)^2}{\alpha}.
\]
Here we use the notation $a:b = \sum_{i,j}a_{ij}b_{ij}$ for matrices $a,b$. 
Next, by Gauss' theorem, the above simplifies to
\[
    \int_{\dombdry_0}\frac{1}{\alpha}(\alpha w \cdot\partial_\normal w + (w\cdot \tangent)^2) - (w\cdot \normal)q = \int_{\roughness_0} \frac{\alpha-\beta}{\alpha}(w\cdot \tangent)(\partial_\normal v\cdot \tangent), 
\]
where we used the boundary conditions of $u,v$ to conclude that $u=v$ on $\micdombdry_0\setminus\roughness_0$, and to rewrite $\alpha w\cdot \partial_\nu w + (w\cdot \tangent)^2 = (\beta-\alpha)(\partial_\normal v\cdot \tangent)(w\cdot \tangent)$ on $\roughness_0$. By Cauchy-Schwartz, we obtain
\begin{equation}
    |w|_{H^1(\dom_0,\Rn{2})}^2 + \|\tfrac{w\cdot \tangent}{\sqrt{\alpha}}\|_{L^2(\roughness_0,\Reals)}^2 \leq \|\tfrac{w\cdot \tangent}{\sqrt{\alpha}}\|_{L^2(\roughness_0,\Reals)}\|\tfrac{\partial_\normal v\cdot \tangent}{\sqrt{\alpha}}\|_{L^2(\roughness_0,\Reals)}\sup_{x\in\roughness_0}|\beta(x)-\alpha(x)|.\label{eq: perturbation}
\end{equation}
Applying \cref{lemma: completion of squares bound} with
\begin{xalignat*}{2}
a &:= |w|_{H^1(\dom_0,\Rn{2})}
& 2c &:=0
\\
b &:=\|\tfrac{w\cdot \tangent}{\sqrt{\alpha}}\|_{L^2(\roughness_0,\Reals)}
& 2d &:= \|\tfrac{\partial_\normal v\cdot \tangent}{\sqrt{\alpha}}\|_{L^2(\roughness_0,\Reals)}\sup_{x\in\roughness_0}|\beta(x)-\alpha(x)|
\end{xalignat*}
yields the estimate
\begin{align}
    |w|_{H^1(\dom_0,\Rn{2})} &\leq \|\tfrac{\partial_\normal v\cdot \tangent}{\sqrt{\alpha}}\|_{L^2(\roughness_0,\Reals)}\sup_{x\in\roughness_0}|\beta(x)-\alpha(x)|\nonumber\\
    &\leq \alpha_0^{-1/2}\|\partial_\normal v\cdot \tangent\|_{L^2(\roughness_0,\Reals)}\sup_{x\in\roughness_0}|\beta(x)-\alpha(x)|.\label{eq: h1 error bound}
\end{align}
It remains to estimate $\partial_\normal v\cdot \tangent$. We start by estimating $\|v\cdot \tangent\|_{L^2(\roughness_0,\Reals)}$, and then apply the boundary condition on $\roughness_0$ to write $\partial_\normal v\cdot \tangent = -v\cdot\tangent/\sqrt{\beta}$. We decompose $v$ as $G + z$, where $G\in C^2(\dom_0,\Rn{2})$ is twice differentiable, has zero divergence and satisfies the Dirichlet boundary condition on $\dombdry_0\setminus\roughness_0$ and is zero at $\roughness_0$. It follows that $\|\Delta G\|_{L^2(\dom_0,\Rn{2})}\leq C$ and $\|\partial_\normal G\cdot \tangent\|_{L^2(\roughness_0,\Reals)}\leq C$ for some $C>0$ independent of $\alpha,\beta$. Hence, $z\in H^1(\dom_0,\Rn{2})$ has zero divergence, is zero on $\dombdry_0\setminus \roughness_0$ and satsifies the slip condition with slip amount $\beta$ at the wall $\roughness_0$. But then, $v\cdot t=z\cdot t$ on $\roughness_0$, and 
\begin{align*}
    |z|_{H^1(\dom_0,\Rn{2})}^2 &
    + \|(v\cdot \tangent)/\sqrt{\beta}\|_{L^2(\roughness_0,\Reals)}^2 = \int_{\dom_0}\nabla z :\nabla z +\int_{\roughness_0}\frac{(z\cdot t)^2}{\beta}\\
    &= \int_{\dombdry_0}\partial_\normal z\cdot z - \int_{\dom_0}z\cdot \Delta z+\int_{\roughness_0}\frac{(z\cdot t)^2}{\beta}\\
    & =-\int_{\dom_0}z\cdot (\nabla p - \Delta G) + \int_{\roughness_0}\frac{1}{\beta}(\beta\partial_\normal z\cdot \tangent + z\cdot \tangent)\\
    &= \int_{\dombdry_0}\sqrt{\beta}\partial_n G\cdot \tfrac{z}{\sqrt{\beta}} - \int_{\dom_0}\nabla G : \nabla z,
\end{align*}
where we leverage integration by parts to rewrite the integrals to the boundary. The final step is to use Cauchy-Schwartz:
\begin{multline*}
    |z|_{H^1(\dom_0,\Rn{2})}^2 + \|\tfrac{(v\cdot \tangent)}{\sqrt{\beta}}\|_{L^2(\roughness_0,\Reals)}^2 \\
    \leq  \|\sqrt{\beta}\partial_n G\|_{L^2(\roughness_0,\Reals)}\|\tfrac{(v\cdot \tangent)}{\sqrt{\beta}}\|_{L^2(\roughness_0,\Reals)} + \|\nabla G\|_{L^2(\dom_0,\Rn{2})}|z|_{H^1(\dom_0,\Rn{2})}.
\end{multline*}
Now, applying \cref{lemma: completion of squares bound} with 
\begin{xalignat*}{2}
a &:= |z|_{H^1(\dom_0,\Rn{2})}
& 2c &:= \|\nabla G\|_{L^2(\dom_0,\Rn{2})}
\\
b &:=\|\tfrac{(v\cdot \tangent)}{\sqrt{\beta}}\|_{L^2(\roughness_0,\Reals)}
& 2d &:= \|\sqrt{\beta}\partial_n G\|_{L^2(\roughness_0,\Reals)}
\end{xalignat*}
gives 
\begin{multline*}
    \|\partial_\normal (v \cdot t)\|_{L^2(\roughness_0,\Reals)} =\dfrac{1}{\sqrt{\alpha_0}}\Bigl\|\dfrac{(v \cdot t)}{\sqrt{\beta}}\Bigr\|_{L^2(\roughness_0,\Reals)}
    \\[0.5em]
    \leq \dfrac{1}{2\sqrt{\alpha_0}}\left(\sqrt{\|\nabla G\|_{L^2(\dom_0,\Rn{2})}^2 + \|\sqrt{\beta}\partial_n G\|_{L^2(\roughness_0,\Reals)}^2} + \|\sqrt{\beta}\partial_n G\|_{L^2(\roughness_0,\Reals)}\right)
    \\[0.5em]
    \leq \dfrac{1}{2\sqrt{\alpha_0}}\left(\sqrt{1+\sup_{x\in\roughness_0}\beta(x)}+\sup_{x\in\roughness_0}\sqrt{\beta(x)}\right)C \leq 2\sqrt{\alpha_0}^{-1}C.
\end{multline*}
Inserting the above estimate for $\|\partial_\normal (v \cdot t)\|_{L^2(\roughness_0,\Reals)}$ into~\eqref{eq: h1 error bound} concludes the proof.
\end{proof}

\end{document}